\documentclass[12pt]{article}

\usepackage{fullpage}
\usepackage{stmaryrd}
\usepackage{graphicx}
\usepackage{makeidx}
\makeindex

\usepackage{amsthm}
\usepackage{amsmath}
\usepackage{amsxtra}
\usepackage{amssymb}
\theoremstyle{definition}
\newtheorem{dfn}{Definition}[section]
\newtheorem*{dfn*}{Definition}
\newtheorem{rem}[dfn]{Remark}

\theoremstyle{plain}

\newtheorem{cor}[dfn]{Corollary}
\newtheorem*{cor*}{Corollary}
\newtheorem{thm}[dfn]{Theorem}
\newtheorem*{thm*}{Theorem}
\newtheorem{lem}[dfn]{Lemma}
\newtheorem{fact}[dfn]{Fact}
\newtheorem{clm}{Claim}
\newtheorem{conj}[dfn]{Conjecture}
\newtheorem{ques}[dfn]{Question}

\theoremstyle{remark}
\newtheorem{sclm}{Subclaim}
\newtheorem{case}{Case}

\newcommand{\iso}{\cong}

\newcommand{\Cc}{\mathbb C}

\newcommand{\RR}{\mathbb R}

\newcommand{\PP}{\mathbb P}

\newcommand{\sub}{\subseteq}
\newcommand{\cross}{\times}
\newcommand{\all}{\forall}
\newcommand{\ex}{\exists}

\newcommand{\inter}{\cap}
\newcommand{\bigun}{\bigcup}

\newcommand{\om}{\omega}
\newcommand{\pow}{\mathcal{P}}
\newcommand{\OR}{\mathrm{OR}}

\newcommand{\Def}{\mathrm{Def}}
\newcommand{\Hull}{\mathrm{Hull}}
\newcommand{\transcl}{\mathrm{tranclos}}
\newcommand{\cut}{-}
\newcommand{\N}{N}
\newcommand{\Aa}{\mathcal{A}}
\newcommand{\Tt}{\mathcal{T}}
\newcommand{\Ss}{\mathcal{S}}
\newcommand{\Uu}{\mathcal{U}}
\newcommand{\Vv}{\mathcal{V}}
\newcommand{\Ww}{\mathcal{W}}
\newcommand{\Ll}{\mathcal{L}}
\newcommand{\Ttmeas}{\mathcal{T}_{\mathrm{meas}}}
\newcommand{\Uubar}{\bar{\mathcal{U}}}
\newcommand{\Ttbar}{\bar{\mathcal{T}}}
\newcommand{\rg}{\mathrm{rg}}
\newcommand{\dom}{\mathrm{dom}}
\newcommand{\pins}{\lhd}

\newcommand{\ins}{\trianglelefteqslant}
\newcommand{\crit}{\mathrm{crit}}

\newcommand{\union}{\cup}
\renewcommand{\empty}{\emptyset}
\newcommand{\rest}{\!\upharpoonright\!}
\newcommand{\com}{\circ}
\newcommand{\comp}{\com}

\renewcommand{\int}{\cap}
\newcommand{\fully}{(\omega,\omega_1,\omega_1+1)}

\newcommand{\lh}{\mathrm{lh}}
\newcommand{\Ult}{\mathrm{Ult}}
\newcommand{\Ebar}{\bar{E}}
\newcommand{\Fbar}{\bar{F}}

\newcommand{\sats}{\models}
\newcommand{\elem}{\preccurlyeq}
\newcommand{\J}{\mathcal{J}}

\newcommand{\Union}{\bigcup}

\newcommand{\VHOD}{\mathsf{V=HOD}}
\newcommand{\AD}{\mathsf{AD}}
\newcommand{\AC}{\mathsf{AC}}
\newcommand{\DC}{\mathsf{DC}}
\newcommand{\VK}{\mathsf{V=K}}
\newcommand{\HOD}{\mathrm{HOD}}
\newcommand{\HC}{\mathrm{HC}}
\newcommand{\ZFC}{\mathsf{ZFC}}
\newcommand{\KP}{\mathsf{KP}}
\newcommand{\ZFmin}{\mathsf{ZF^{-}}}
\newcommand{\theory}{\mathsf{KP^*}}

\newcommand{\pistol}{\P}
\newcommand{\longsword}{\mathrm{long}}
\newcommand{\Coll}{\mathrm{Col}}
\newcommand{\es}{\mathbb{E}}
\newcommand{\abar}{\bar{a}}
\newcommand{\qbar}{\bar{q}}
\newcommand{\xbar}{\bar{x}}

\newcommand{\nubar}{{\bar{\nu}}}
\newcommand{\mubar}{{\bar{\mu}}}
\newcommand{\taubar}{{\bar{\tau}}}
\newcommand{\betabar}{{\bar{\beta}}}
\newcommand{\gammabar}{{\bar{\gamma}}}

\newcommand{\thetabar}{{\bar{\theta}}}
\newcommand{\etabar}{{\bar{\eta}}}

\newcommand{\Pbar}{\bar{P}}
\newcommand{\Rbar}{\bar{R}}
\newcommand{\eps}{\varepsilon}

\newcommand{\Qbar}{{\bar{Q}}}
\newcommand{\Ubar}{{\bar{U}}}
\newcommand{\jbar}{{\bar{j}}}
\newcommand{\tbar}{{\bar{t}}}
\newcommand{\ph}{\mathcal{P}}
\newcommand{\omi}{\zeta}
\newcommand{\core}{\mathfrak{C}}
\newcommand{\Core}{\mathfrak{C}}
\newcommand{\her}{\mathcal{H}}

\newcommand{\pred}{\mathrm{-pred}}
\newcommand{\tc}{\mathrm{tc}}
\newcommand{\trivcom}{\mathrm{tc}}
\newcommand{\dam}{\mathrm{dam}}
\newcommand{\dirlim}{\mathrm{dir lim}}
\newcommand{\un}{\union}
\newcommand{\super}{\supseteq}
\newcommand{\into}{\to}
\newcommand{\goesto}{\longmapsto}
\newcommand{\id}{\mathrm{id}}

\newcommand{\sq}{\mathrm{sq}}
\newcommand{\nth}{\textrm{th}}
\newcommand{\conc}{\ \widehat{\ }\ }
\newcommand{\supconc}{\widehat{\ }}

\newcommand{\supseg}{\unrhd}
\newcommand{\Col}{\Coll}
\newcommand{\forces}{\sats}
\newcommand{\bfPi}{\mathbf{\Pi}}
\newcommand{\bfSigma}{\mathbf{\Sigma}}
\newcommand{\bfDelta}{\mathbf{\Delta}}

\DeclareMathOperator{\Th}{Th}
\DeclareMathOperator{\card}{card}
\DeclareMathOperator{\cof}{cof}
\DeclareMathOperator{\wfp}{wfp}

\begin{document}
\title{Measures in Mice}
\author{Farmer Schlutzenberg\\farmer@math.berkeley.edu}
\date{December, 2007}
\maketitle
\begin{center}\begin{large}
Ph.D. thesis, UC Berkeley\\
Dissertation Advisor: John Steel\\
\end{large}\end{center}

\section*{Abstract}
This thesis analyses extenders appearing in fine structural mice. Kunen showed that in the inner model for one measurable cardinal, there is a unique normal measure. This result is generalized, in various ways, to mice below a superstrong cardinal.

The analysis is then used to show that certain tame mice satisfy $\VHOD$. In particular, the approach provides a new proof of this result for the inner model $M_n$ for $n$ Woodin cardinals. It is also shown that in $M_n$, all homogeneously Suslin sets of reals are $\mathbf{\Delta}^1_{n+1}$.
\pagebreak
\tableofcontents
\pagebreak

\section{Introduction}\label{sec:intro}
In \cite{kunen}, Kunen showed that if $V=L[U]$ where $U$ is a normal measure, then $U$ is the unique normal measure, and all measures are reducible to finite products of this normal measure. Mitchell constructed inner models with sequences of measurables in \cite{mitchell} and \cite{mitchellrevisit}, and proved related results characterizing measures in those models. Dodd similarly characterized the extenders appearing in his inner models for strong cardinals, in \cite{doddstrong}. The first few sections of this thesis extend these results to inner models below a superstrong cardinal. However, the models we deal with are fine structural premice as in \cite{outline}.

Given a mouse $N$ satisfying ``$E$ is a total, wellfounded extender'', we are interested in how $E$ relates to $N$'s extender sequence $\es^N$. In particular, we would like to know whether $E$ is on $\es^N$, or on the sequence of an iterate, or more generally, whether $E$ is the extender of an iteration map on $N$. Although we have only a partial understanding, we do, happily, have some affirmative results.

Some of the following theorems are stated without any smallness assumption on the mice involved. However, the premice we work with do not have extenders of superstrong type indexed on their sequence (see \cite{outline}). Removing this restriction, a counterexample to \ref{thm:cohere} is soon reached, though it seems the statement of the theorem might be adapted to deal with this. In the following, $\nu_E$ denotes the natural length (or support) of an extender $E$. (See the end of this introduction for notation and definitions.)

\begin{cor*}[\ref{cor:coarse} (Steel, Schlutzenberg)]
Let $N$ be an $\fully$-iterable mouse satisfying $\ZFC$, and suppose
\[ N\sats E \textrm{ is a short, total extender, } \nu=\nu_E \textrm{ is a cardinal, and } \her_\nu\sub\Ult(V,E). \]
Then the trivial completion of $E$ is on $\es^N$.
\end{cor*}

Here and below, one can make do with much less than $\ZFC$ (see \ref{cor:coarse}). A stronger theorem is actually proven (see \ref{thm:coarseDodd}). Notice that if $E$ is a normal measure in $N$, then $\nu_E=(\crit(E)^+)^N$ is an $N$-cardinal, and $\her_{\nu_E}^N\sub\Ult(N,E)$. As a corollary to the proof, we'll also obtain that if $N$ is a mouse modelling $\ZFC$, and $\kappa$ is uncountable in $N$, then $L(\pow(\kappa)\int N)\sats\AC$. This confirms a conjecture of Hugh Woodin. The hypothesis ``$\kappa$ uncountable in $N$'' is necessary, since Woodin proved that $\AD^{L(\RR)}$ holds assuming there are $\om$ Woodin cardinals with a measurable above (see \cite{detLR}).

\begin{thm*}[\ref{thm:cohere}]
Let $N$ be an $\fully$-iterable mouse satisfying $\ZFC$, and suppose in $N$, $E$ is a wellfounded extender, which is its own trivial completion, and $(N||\lh(E),E)$ is a premouse. Then $E$ is on $\es^N$.\end{thm*}

Again a stronger theorem is proven, in which $E$ fits on the sequence of an iterate of $N$ instead.

An extender $F$ (possibly partial) is \emph{finitely generated} if there is $s\in[\nu_F]^{<\om}$ so that for each $\alpha<\nu_F$, there's $f$ satisfying $\alpha=[s,f]_F$.

\begin{thm*}[\ref{thm:exactmeas}]
Let $N$ be an $\fully$-iterable mouse satisfying $\ZFC$, and suppose in $N$, $E$ is a countably complete ultrafilter. Then there is a finite normal (fine-structural) iteration tree $\Tt$ on $N$, with last model $\Ult(N,E)$, and $i^N_E$ is the main branch embedding of $\Tt$. Moreover, $\Tt$'s extenders are all finitely generated.\end{thm*}

In \S\ref{sec:meas} we also investigate partial measures: if $N$ is a mouse, with a normal measure $E$ as its active extender, we consider when $E\int N||\alpha$ is on the $N$ sequence (for $\alpha<(\crit(E)^+)^N$).

Steel showed that for $n\leq\om$, $M_n$ satisfies a variant of $\VK$ in the intervals between its Woodin cardinals. This gives that $M_n$ satisfies $\VHOD$. (Of course, $M_n$ isn't allowed to refer to $\es^{M_n}$ to achieve this feat!) A proof for the $n<\om$ case can be seen in \cite{cmpw}. This method extends somewhat further into the mouse hierarchy, but just how far seems to be unknown. We show $\VHOD$ in certain mice by another (but related) method. A premouse is \emph{self-iterable} if it satisfies ``I am iterable''. (This isn't intended to be precise - there are of course varying degrees of self-iterability.) In \S\ref{sec:stacking}, we isolate a quality, \emph{extender-fullness} (closely related to the reuslts of \S\ref{sec:cohere}), that an iteration strategy might enjoy. We show that a premouse modelling $\ZFC$ and being sufficiently extender-full self-iterable can identify its own extender sequence. Indeed, it is the unique one that the premouse thinks is sufficiently extender-full iterable. Therefore $\VHOD$ in such a premouse. For $n\leq\om$, the results of \S\ref{sec:cohere} will show that $M_n$'s self-iteration strategy is extender-full, so we obtain a new proof of $\VHOD$ there. We then prove that various other mice also have (barely) enough extender-full self-iterability. It's critical, though, that the mouse $M$ in question is \emph{tame}, a fine-structural statement of ``there is no $\kappa$ that's strong past a Woodin''. Our proof of $\VHOD$ works in particular below the strong cardinal of the least non-tame mouse. Our approach breaks down soon after non-tame mice are reached, because of a lack of self-iterability (see \ref{fact:nontame}).

In \S\ref{sec:hom} we look at homogeneously Suslin sets in mice. Using Kunen's analysis of measures in $L[U]$, Steel observed that all homogeneously Suslin sets in $L[U]$ are $\mathbf{\Pi}^1_1$ (so the two pointclasses coincide, by (\cite{kanamori}, 32.1) or (\cite{jech}, 33.30)). This was generalized by Schindler and Koepke in \cite{hsssim} (we give some details in \S\ref{sec:hom}). We show that in mice in the region of $0^\pistol$ or below, and modelling $\ZFC$, all homogeneously Suslin sets are $\mathbf{\Pi}^1_1$. We also show that for $n<\om$, in $M_n$, all homogeneously Suslin sets are $\mathbf{\Delta}^1_{n+1}$; in fact they are correctly so, in that the definition also yields a $\mathbf{\Delta}^1_{n+1}$ set in $V$. By Martin and Steel's results in \cite{projdet}, all $\bfPi^1_n$ sets are homogeneously Suslin in $M_n$. Therefore:
\begin{cor*}[\ref{cor:whomMn}]
In $M_n$, the weakly homogeneously Suslin sets of reals are precisely the $\bfSigma^1_{n+1}$ sets.
\end{cor*}
We also show that the correctly $\bfDelta^1_{n+1}$ sets of $M_n$ are exactly the $\Col(\om,\delta_0)$-universally Baire sets of $M_n$, where $\delta_0$ is the least Woodin cardinal of $M_n$. The question of the precise extent of the homogeneously Suslin sets in the projective hiearchy of $M_n$ remains unsolved. (\S\ref{sec:hom} is fairly independent of the rest of this thesis. It uses the notion of finite support discussed in \S\ref{sec:meas}, but this is straightforward.)

Finally, in \S\ref{sec:copying}, we discuss and correct some problems in the copying construction of \cite{fsit} (some details are supplied in \cite{outline}), and simultaneously prove that ``freely dropping'' iterability follows from normal iterability. Here the antagonist of the iterability game may enforce drops in model and degree at will. This fact is needed in proving some of the preceding theorems.\\
\pagebreak

\noindent\emph{Conventions and Notation}\\

Our notion of \emph{premouse}\index{premouse} is that of \cite{outline}.

For $\nu$ a cardinal, $\her_\nu$\index{$\her_\nu$} denotes the collection of sets of size hereditarily less than $\nu$.

Whenever we refer to an ordering on $\OR^{<\om}$, it is to the lexiographic order with larger ordinals considered first. For instance, $\{4,12,182\}<\{3,15,182\}$ and $\{6\}<\{4,6\}$.

We discuss the use of the terms \emph{pre-extender} and \emph{extender}. Given a rudimentarily closed $M$, a \emph{pre-extender over} $M$ and an \emph{extender over} $M$ are as in \cite{outline}. Otherwise, we liberally use the term \emph{extender} to mean ``pre-extender over some $M$''. An extender is \emph{total} if its measures measure all sets in $V$. We sometimes emphasise that an extender need not be total by calling it \emph{partial}. If $E$ is a total extender, $E$ can be \emph{countably complete} or \emph{wellfounded} as usual.

All extenders we use have support of the form $X\sub\OR$. Typically $X=\gamma\un q$ for some ordinal $\gamma$ and $q\in\OR^{<\om}$. Suppose $E$ is an extender. $\tc(E)$\index{$\trivcom(E)$} denotes the trivial completion of $E$, $\nu_E$\index{$\nu_E$} denotes the natural length of $E$, and $\lh(E)$\index{$\lh(E)$} denotes the length of $\trivcom(E)$ (see \cite{outline} for definitions). For $X\sub\nu_E$, $E\rest X$\index{$E\rest X$} is the sub-extender using only co-ordinates in $X$. If $\kappa=\crit(E)$, $E\sub M\cross[\gamma]^{<\om}$ and $E$'s component measures are $M$-total, then we say $E$ \emph{measures exactly}\index{measures exactly} $\pow(\kappa)\int M$. If also $\alpha\in M$ whenever there is a wellorder of $\kappa$ of ordertype $\alpha$ in $M$, then $(\kappa^+)^E$ denotes $(\kappa^+)^M$. Given some $N$ over which $E$ is a pre-extender, $i^N_E$\index{$i_E$, $i^N_E$} denotes an ultrapower embedding from $N$ to $\Ult(N,E)$, or $\Ult_k(N,E)$, depending on context; $i_E$ will be used when $N$ is understood. The notation $\Ult_k(E,F)$ (for $F$ another extender) and $\com_k$ are introduced in \ref{dfn:extcomp}.

For $P$ a premouse, $\es^P$\index{$\es$, $\es_+$} denotes the extender sequence of $P$, not including any active extender. $F^P$\index{$F^P$ ($P$ a premouse)} denotes the active extender. $\es^P_+$ denotes $\es^P\conc F^P$. Let $\alpha\leq\OR^P$ be a limit ordinal. Let's define $P|\alpha$\index{$P|\alpha$, $P||\alpha$} and $P||\alpha$. $P|\alpha\ins P$, and $\OR^{P|\alpha}=\alpha$. $P||\alpha$ is the passive premouse of height $\alpha$ agreeing with $P$ strictly below $\alpha$. If $\es$ is a good extender sequence, $\J^\es$\index{$\J^\es$} denotes the premouse constructed from $\es$. $\J^\es_\alpha$ denotes $\J^\es|(\alpha\cdot\alpha)$. (I.e. it is the $\alpha^\nth$ level in the $\J^\es$-hierarchy. We generally only use this $\J$-notation when a premouse needs to refer to its own levels.) If $P$ is sound, $\J_1(P)$ denotes the premouse of height $\OR^P+\om$ extending $P$. If $P$ is active, $\mu_P$\index{$\mu_P$ ($P$ a premouse)} and $\nu_P$\index{$\nu_P$ ($P$ a premouse)} denote $\crit(F^P)$ and $\nu_{F^P}$ respectively.

For definability over premice, we use the $r\Sigma_n$ hierarchy as in \cite{outline}. Since $r\Sigma_1=\Sigma_1$, we just write ``$\Sigma_1$''. For $P$ a premouse, $X\sub P$ and $n\leq\om$, $\Def^P_n(X)$ denotes the set of points in $P$ definable with an $r\Sigma_n$-term from parameters in $X$. $\Hull_n^P(X)$ denotes the transitive collapse of $(\Def^P_n(X),\es^P\int P,F^P\int P)$ (where $F^P$ is coded amenably as for a premouse).

Given an iteration tree $\Tt$, $\kappa^\Tt_\alpha=\crit(E^\Tt_\alpha)$, $\nu^\Tt_\alpha=\nu_{E^\Tt_\alpha}$ and $\lh^\Tt_\alpha=\lh(E^\Tt_\alpha)$. $(M^*)^\Tt_{\alpha+1}$ is the model to which $E^\Tt_\alpha$ applies after any drop in model, and
\[ (i^*)^\Tt_{\alpha+1,\beta}:(M^*)^\Tt_{\alpha+1}\to M^\Tt_\beta \]
is the canonical embedding, if it exists. If $\Tt$ has a last model and there is no drop on $\Tt$'s main branch (from its root to its last model), then $i^\Tt$ denotes the corresponding embedding. When $\Tt$ is clear from context, we may drop the superscript in any of this notation.

Iterability for a phalanx is for iterations such that $\lh(E)$ is strictly above all exchange ordinals for each $E$ used in the iteration.

$\theory$\index{$\theory$} is the theory $\KP$ + ``There are unboundedly many $\alpha$'s such that $\J^\es_\alpha\sats\KP$''.

\pagebreak
\section{Extenders Strong Below a Cardinal}\label{sec:esc}
Suppose $N$ is a fully iterable mouse modelling $\ZFC$, and $E$ is a total, wellfounded extender in $N$. We are interested in just how $E$ was constructed from $\es^N$. In this section we'll show that if $E$ is nice enough, things are simple as possible: $E$'s trivial completion is $\es^N_\alpha$ for some $\alpha$. In order to state the main theorem of this section (\ref{thm:coarseDodd}) in full generality, we first need to discuss \emph{Dodd soundness}. However, this definition isn't required for the statement or proof of the simpler corollary \ref{cor:coarse}, which still carries much of the utility of the theorem, so the impatient reader might skip ahead to there.

\begin{dfn}[Generators]\label{dfn:gen}\index{generate}\index{suffices}\index{finitely generated} Let $E$ be a short extender with $\crit(E)=\kappa$. Suppose $P$ is a premouse such that $E$ measures exactly $\pow(\kappa)\int P$ and $X\sub\nu_E$. $X$ \emph{generates} $\alpha$ if $\alpha=[a,f]^P_E$ for some $f\in P$ and $a\in X^{<\om}$. $X$ \emph{generates} $E$ or \emph{suffices as generators} for $E$ if every $\alpha<\nu_E$ is generated by $X$. Let $\alpha<\nu_E$ and $t\in\nu_E^{<\om}$. Then $\alpha$ is a $t$-\emph{generator} of $E$ iff $\alpha$ is not generated by $\alpha\un t$. Note that an $\empty$-generator is just a generator (in the sense of \cite{outline}). $E$ is \emph{finitely generated} if some finite set generates $E$.
\end{dfn}

The following notion, due to Steel and Dodd, is taken from (\cite{combin}, \S3). Notation there is a little different.

\begin{dfn}[Dodd parameter and projectum]\label{dfn:Dodd}\index{Dodd parameter}\index{Dodd projectum}
Let $E$ be a short extender. We define $E$'s Dodd parameter $t_E=\{(t_E)_0,\ldots,(t_E)_{k-1}\}$ and Dodd projectum $\tau_E$. Given $t_E\rest i=\{(t_E)_0,\ldots,(t_E)_{i-1}\}$, $(t_E)_i$ is the largest $t_E\rest i$-generator of $E$ that is $\geq(\kappa^+)^E$. Notice $(t_E)_i<(t_E)_{i-1}$; $k=|t_E|$ is large as possible. $\tau_E$ is the sup of $(\kappa^+)^E$ and all $t_E$-generators of $E$.
\end{dfn}

The following is straightforward to prove (see (\cite{combin}, \S3) for some details).

\begin{fact}\label{fact:Dodd} Let $E$ be a short extender. Then $\tau_E$ is the least $\tau\geq(\kappa^+)^E$ such that there is $t\in\OR^{<\om}$ with $\tau\un t$ generating $E$. $t_E$ is least in $\OR^{<\om}$ witnessing this fact.

Suppose that $P$ is a premouse, $\kappa=\crit(E)$, $E$ measures exactly $\pow(\kappa)\int P$, and
\[ P|(\kappa^+)^P = \Ult(P,E)|(\kappa^+)^{\Ult(P,E)}. \]
 If $\tau_E\in\wfp(\Ult(P,E))$ then $\tau_E$ is a cardinal of $\Ult(P,E)$. If $\tau_E=(\kappa^+)^P$ then $E$ is generated by $t_E\un\{\crit(E)\}$. If $\tau_E>(\kappa^+)^E$ then $E$ is not finitely generated.
\end{fact}

\begin{rem}
If $\tau_E\notin\wfp(\Ult(P,E))$ then $\tau_E=\nu_E$ is a limit of $\Ult(P,E)$-cardinals.
\end{rem}

\begin{dfn}\index{Dodd soundness} Let $E,P$ be as in \ref{dfn:gen} and $t=t_E$. $E$ is \emph{Dodd-solid} iff for each $i<|t|$, $E\rest (t_i\un t\rest i)\in\Ult(P,E)$. $E$ is \emph{Dodd-sound} iff it is Dodd-solid and, if $\tau_E>(\kappa^+)^E$ then for each $\alpha<\tau_E$, $E\rest(\alpha\un t_E)\in\Ult(P,E)$.\end{dfn}

\begin{rem} With $E,P,\kappa$ as above, if $\tau_E=(\kappa^+)^E$, it may be the case that $\kappa$ is a $t_E$-generator. In this case one might expect Dodd-solidity should require $E\rest t_E\in\Ult(P,E)$ also. This, however, is not the standard definition. It may follow from the proof of \ref{fact:Doddsound}, though we haven't investigated this.\end{rem}

The following critical fact is proven by (\cite{combin}, 3.2) and a correction in  (\cite{deconstruct}, 4.1).

\begin{fact}[Steel]\label{fact:Doddsound}
Let $P$ be an active $1$-sound, $(0,\om_1,\om_1+1)$-iterable premouse. Then $F^P$ is Dodd-sound.\end{fact}

\begin{lem}\label{lem:Ecc} Suppose $N$ is a premouse modelling $\theory$ and ``$E$  a total (pre-)extender with critical point $\kappa$, and for each $\alpha$, $\alpha^\kappa$ exists''. Then $\Ult(N,E)$ is wellfounded iff $E$ is countably complete in $N$. If this is so, then $\Ult(N,E)$ is a transitive class of $N$.\end{lem}
\begin{proof} Because $N$ satisfies ``$\alpha^\kappa$ exists for each $\alpha$'', the membership relation of $\Ult(N,E)$ is essentially set-like in $N$. Otherwise the argument is as for when $N\sats\ZFC$.\renewcommand{\qedsymbol}{$\Box$(Lemma \ref{lem:Ecc})}\end{proof}

We're now ready to state the main theorem of this section.

\begin{thm}\label{thm:coarseDodd} Let $N$ be an $\fully$-iterable mouse such that
\[ N\sats \theory\ +\ E\ \textrm{is a total, short, countably complete, Dodd-sound extender,}\ \crit(E)=\kappa,\]
\[ \tau=\tau_E\ \textrm{is a cardinal,}\ (\tau^\kappa)^+\ \textrm{exists, and}\ \her_\tau \sub\Ult(L_{\kappa^+}[\es],E). \]
Then $E$ is on $\es^N$.\end{thm}

\begin{cor}[Steel, Schlutzenberg]\label{cor:coarse} Let $N$ be an $\fully$-iterable mouse such that
\[ N\sats \theory\ +\ E\ \textrm{is a total, short, countably complete extender,}\ \crit(E)=\kappa,\]
\[ \nu=\nu_E\ \textrm{is a cardinal,}\ (\nu^\kappa)^+\ \textrm{exists, and}\ \her_\nu\sub\Ult(L_{\kappa^+}[\es],E). \]
Then $E$ is on $\es^N$.\end{cor}

\begin{rem} The basic idea behind the proof of \ref{thm:coarseDodd} is like that of the initial segment condition in (\cite{fsit}, \S10). Steel first proved \ref{cor:coarse} in the case that $\nu_E$ is regular in $N$ and $E$ coheres $\es^N$ below $\nu_E$. The author then generalized this to obtain \ref{thm:coarseDodd}. The proof of \ref{cor:coarse} is the first half of the proof of \ref{thm:coarseDodd}, with $\tau=\nu$. This half does not involve the Dodd soundness of $E$.\end{rem}

\begin{proof}[Proof of Theorem \ref{thm:coarseDodd}]\setcounter{clm}{0}
After some motivation, the proof will work through claims 1 to 6 below.

Suppose for the moment that $E$ is type 3, and that $\Ult(N,E)$ is sufficiently iterable that we can successfully compare $N$ with $\Ult(N,E)$. Suppose this results in iteration trees $\Tt$ on $N$ and $\Ss$ on $\Ult(N,E)$, such that
\begin{itemize}
 \item Both trees have the same final model $Q$,
 \item Neither tree drops on the branch leading to $Q$,
 \item The resulting embeddings commute; i.e., $i^\Ss\com i_E=i^\Tt$,
 \item $\crit(i^\Ss)\geq\nu_E$.
\end{itemize}
Then by standard arguments, the first extender $F$ used on $\Tt$'s main branch is compatible with $E$, and $\nu_F\geq\nu_E$, and it follows that $E$ is on $\es^N$. We won't reach this directly, but first replace $N$ with a hull $M$ of $N|\lambda$ for some $\lambda$. With $\Ebar$ the collapse of $E$, the countable completeness of $E$ will guarantee the iterability of $\Ult(M,\Ebar)$ and (the $M$ level versions of) the first three properties listed above. The fourth will require the iterability of the phalanx $(M,\Ult(M,\Ebar),\nu_{\Ebar})$. Most of the proof of the type 3 case is in establishing that iterability. In general, $\nu_E$ will be replaced with $\tau_E$ (these coincide when $E$ is type 3).

We now drop the assumptions of the previous paragraph, and proceed with the proof. We first obtain the hull $M$.

We may assume that $\lambda=((\tau^\kappa)^+)^N$ is the largest cardinal of $N$. $E\in N|\lambda$ and $\lh(E)<\lambda$, since $E$ is coded by a subset of $\tau$ and computes a surjection of $\tau$ onto $\lh(E)=(\nu_E^+)^{\Ult(N|(\kappa^+)^N,E)}$. We may also assume $E$ is the least counter-example to the theorem in the order of construction of $N$. Then $E$ is actually definable over $N|\lambda$, since it's the least $E'$ in $N|\lambda$ such that $N|\lambda$ satisfies ``$E'\notin\es$, $\tau_{E'}^{\crit(E')}$ exists'', and all first-order hypotheses of the theorem bar ``$(\tau_{E'}^{\crit(E')})^+$ exists''. Let
\[ M=\Hull_\om^{N|\lambda}(\empty) \]
and
\[ \pi:M\to N|\lambda \]
be the hull embedding. So $E\in\rg(\pi)$.

By \ref{lem:Ecc} and that $(\tau^\kappa)^+$ is the largest cardinal of $N$, $\Ult(N,E)$ is a (wellfounded) transitive class of $N$. Let $\eta$ be the index of least difference between $N$ and $\Ult(N,E)$. Let $\theta=\card^N(\eta)$. Let $\pi(\bar{E})=E$, $\pi(\thetabar)=\theta$, etc. The first claim will get us half way.

\begin{clm}\label{clm:escit} The phalanx
\[ \ph=(M, \Ult(M,\bar{E}), \thetabar) \]
is $\omega_1+1$ iterable.\end{clm}
\begin{proof} 
\begin{case}\label{case:escreg} $\theta$ is regular in $N$.\end{case}

We want to obtain ordinals $\gamma,\xi$ and embeddings
\[ \psi:M\to N|\gamma,\ \ \sigma:\Ult(M,\Ebar)\to N|\xi,\]
such that
\[ \psi\rest\thetabar=\sigma\rest\thetabar.\]
Using the freely dropping iterability of $N$, established in \S\ref{sec:copying}, these maps will allow us to copy an iteration tree on $(M,\Ult(M,\Ebar),\thetabar)$ to a tree on $N$, completing the proof of the claim. We will in fact find such a $\psi,\sigma,\gamma,\xi$ inside $N$. The existence of such objects is a first order fact about $N$, and $i_E(M)=M$, so it suffices to show it is true in $\Ult(N,E)$ instead.

Let $\sigma^*=\pi\rest\Ult(M,\Ebar)$. Note that
\begin{equation}\label{eqn:escshift} \sigma^*:\Ult(M,\Ebar)\to\Ult(N|\lambda,E) = \Ult(N,E)|i_E(\lambda)=\Ult(N,E)|\lambda, \end{equation}
$\sigma^*$ is elementary, and $\sigma^*\rest\thetabar=\pi\rest\thetabar$. (Applying the shift lemma also yields $\sigma^*$.) In fact there is $\sigma'$ with these properties in $\Ult(N,E)$. For $\theta$ is regular in $N$ and $\pi, M\in N$, so we have $\pi\rest\thetabar$ is bounded in $\theta$, and thus $\pi\rest\thetabar\in N|\theta$. By agreement below $\theta$, $\pi\rest\thetabar\in\Ult(N,E)$. So $\Ult(N,E)$ has the (illfounded) tree searching for an embedding $\sigma'$ with the properties above, and since $\Ult(N,E)$ is wellfounded and models $\theory$, it has a branch. So we have
\[ \Ult(N,E)\sats \sigma':\Ult(M,\Ebar)\to\J^\es_{i_E(\lambda)}\ \&\ \sigma'\rest\thetabar = \pi\rest\thetabar. \]

We now want to convert $\pi$ into an appropriate map $\psi:M\to N|\gamma$ for some $\gamma<\theta$. This will be done by taking some hulls of $N|\lambda$; for that we need some condensation facts.

\begin{lem}\label{lem:escpassive}
Suppose $P$ is an $\fully$-iterable mouse satisfying $\theory$, $\delta$ is a cardinal of $P$ and $\cof^P(\delta)>\om$. Suppose $H\pins P$, $\delta\leq\OR^H$ and $H$ projects to $\delta$. Then there are unboundedly many $\delta'<\delta$ such that
\[ \Hull_\om^H(\delta')\ins P.\]
\end{lem}
\begin{proof}
Let $\beta<\delta$. Let $P_0=\Def^P(\beta)$, and 
\[ P_{n+1}=\Def_\om^P(P_n\union\{P_n\inter\delta\}).\]
Let $\delta'=\sup_n(P_n\inter\delta)$ and $f:P'\to P$
be the uncollapse embedding of $\Def_\om^P(\delta')=\Union_n P_n$. Since $P$ sees this construction, $\delta'<\delta$. We claim $\delta'$ works. For $\crit(f)=\delta'$ and $f(\delta')=\delta$, so $\rho_\omega(P')=\delta'$. Thus by condensation (\cite{outline}, 5.1), we either have $P'\ins P$ or $P'\pins\Ult(P|\delta',\es^P_{\delta'})$. In the latter case, $\delta'$ is a successor cardinal in the ultrapower, but the $P_n$'s are constructed from $P'$ in the same way they were from $P$, showing that $\cof^\Ult(\delta')=\om$, a contradiction.\renewcommand{\qedsymbol}{$\Box$(Lemma \ref{lem:escpassive})}\end{proof}

\begin{lem}\label{lem:esccond}
Suppose $P$ be an $\fully$-iterable mouse, not of the form $\J_1(P')$, $\delta<\zeta<\OR^P$, $\delta$ is a $P$-cardinal, $P|\zeta$ is passive and $P|\zeta\sats\ZFmin$. Let $H=\Hull_\om^{P|\zeta}(\delta)$. Then $H\pins P$; moreover $\rho_1^{\J_1(H)}=\delta$, $p_1^{\J_1(H)}=\{\OR^H\}$, $\J_1(H)$ is $\om$-sound.\end{lem}

Granting this lemma, we can establish the iterability of our phalanx (of Claim \ref{clm:escit}). Let $H=\Hull_\om^{N|\lambda}(\theta)$. The lemmas give $\J_1(H)\ins N$, and a $\theta'<\theta$ such that $\J_1(H')\ins N$, where $\J_1(H')=\Hull_\om^{\J_1(H)}(\theta')$ and $\rg(\pi)\int\theta\sub\theta'$. There is a unique elementary $\psi':M\to H'=N|\gamma$, and it agrees with $\pi$, and $\sigma'$, below $\thetabar$. Since $N|\theta=\Ult(N,E)|\theta$, we have $\psi'\in\Ult(N,E)$, so we're done.

\begin{proof}[Proof of Lemma \ref{lem:esccond}]\setcounter{sclm}{0}
A strengthening of this lemma, in which $P=\J_1(P')$ is allowed, can be proven using the degree 1 version of condensation (\cite{fsit}, \S8). However, the stated version is sufficient for our purposes, and we give a direct argument which involves a little less fine structure, and involves some arguments to be used later.

The failure of the lemma is a first order statement satisfied by $P$. So we may assume that $P=\Hull_\om^P(\phi)$. Let $H=\Hull^{P|\zeta}(\delta)$.

\begin{sclm}\label{sclm:cond.proj} $\rho_1^{\J_1(H)}=\delta$ and $\J_1(H)$ is $\omega$-sound.\end{sclm}
\begin{proof}  As $H=\Hull^H(\delta)$, we get $\rho_1^{\J_1(H)}\leq\delta$. (Note $\Th^H(\delta)$ is not in $H$ and $H\sats\ZFmin$, so it's not in $\J_1(H)$ either.) In fact $\rho_1^{\J_1(H)}=\delta$, as $\delta$ is a cardinal of $P$, $\J_1(H)\in P$, and $P$ and $\J_1(H)$ agree below $\delta$.

We claim that since $H\elem_1\J_1(H)$. This is because $H\sats\ZFmin$. For if $\varphi$ is pure $\Sigma_1$ (\cite{fsit}, \S2), in the language of passive premice, let $\varphi'_n$, in the same language, be such that for any sound passive premouse $B$, and $x\in B$,
\[ S_n(B)\sats\varphi(x)\ \iff\ B\sats\varphi'_n(x). \]
(Here $S_n(B)$ is the transitive structure $n$ levels into the rud closure of $B\un\{B\}$; $S_\om(B)=\J_1(B)$.) Suppose $\J_1(H)\sats\varphi(x)$ with $x\in H$. Then there's $n\in\om$ such that $S_n(B)\sats\varphi(x)$. Let $\theta\leq\OR^H$ be minimal such that $x\in H||\theta$ and $H||\theta\sats\varphi'_n(x)$. Since $H\sats\ZFmin$, $\theta<\OR^H$, and by minimality, $H||\theta\not\sats\ZFmin$. Therefore $H|\theta$ is passive, so $H|(\theta+\omega)\sats\varphi(x)$, so $H\sats\varphi(x)$.

So the theories in question agree about pure $\Sigma_1$ formulae, which in fact implies they also agree about generalized $\Sigma_1$ formulae, as required.

So $\Th_1^{\J_1(H)}(\OR^H)\in \J_1(H)$. Therefore $p_1^{\J_1(H)}=\{\OR^H\}$ is $1$-solid and since $H=\Hull_\om^H(\delta)$, $\J_1(H)$ is $1$-sound. Since $\rho_1^{\J_1(H)}=\delta$ is a cardinal of $P$, $\J_1(H)$ is $\om$-sound.\renewcommand{\qedsymbol}{$\Box$(Subclaim \ref{sclm:cond.proj})}\end{proof}

\begin{sclm}\label{sclm:cond.phit}
\begin{itemize}\item[]
\item[(a)] For all $\beta\in H$, $^\beta H\inter (\J_1(H)) \subset H$;
\item[(b)] The phalanx
\[ (P,\J_1(H),\delta) \]
is $\omega_1+1$-iterable.
\end{itemize}\end{sclm}
\begin{proof} (a) is because $H$ models $\ZFmin$.

For (b): The phalanx $(P,H,\delta)$ is iterable since the embeddings into $P$ levels agree below the $P$ cardinal $\delta$. (So an iteration on this phalanx lifts to a freely droppping iteration of $P$, as discussed in \S\ref{sec:copying}.)

A normal degree $\om$ tree $\Tt$ on $H$ is essentially a normal degree $0$ tree $\Tt'$ on $\J_1(H)$, using the same extenders.  For $\alpha<\lh(\Tt)$, if $[0,\alpha]_\Tt$ drops then $M^{\Tt'}_\alpha = M^{\Tt}_\alpha$, and otherwise $M^{\Tt'}_\alpha = \J_1(M^{\Tt}_\alpha)$. By (a), one can inductively keep the association going. This extends easily to trees on $\ph$ and $(P,\J_1(H),\delta)$.\renewcommand{\qedsymbol}{$\Box$(Subclaim \ref{sclm:cond.phit})}\end{proof}

Now compare $P$ with $(P,\J_1(H),\delta)$, giving trees $\Tt$ and $\Uu$ respectively, with final models $Q^\Tt$ and $Q^\Uu$. As $\delta$ is a cardinal of $P$, no dropping occurs in $\Uu$ moving from the root $P$. The Closeness Lemma (\cite{fsit}, 6.1.5) shows that all other ultrapowers in $\Uu$ are close, so fine structure is preserved by the branch embeddings, as with normal iterations. Suppose $Q^\Uu$ is above $P$. Since $P$ is pointwise definable (so $\J_1(P)$ projects to $\om$), it is straightforward to show $Q^\Tt=Q^\Uu$, $i^\Tt$ and $i^\Uu$ exist and $i^\Tt=i^\Uu$. This gives a contradiction via compatible extenders.

So $Q^\Uu$ is above $\J_1(H)$, and again since $P$ is coded by $\Th(P)$, $Q^\Uu\ins Q^\Tt$. So $\Uu$'s main branch doesn't drop, so $Q^\Uu=\J_1(H)$. If $Q^\Tt=\J_1(H)$, then again since $\Th(P)\notin P$ and $\J_1(H)\in P$, there must be a drop from $P$ to $Q^\Tt$, so $\J_1(H)$ is not sound, contradicting Subclaim \ref{sclm:cond.proj}. Since $\J_1(H)$ projects to $\delta$,
\[ \J_1(H) \pins Q^\Tt|(\delta^+)^{Q^\Tt} = P||(\delta^+)^{Q^\Tt}. \]
\renewcommand{\qedsymbol}{$\Box$(Lemma \ref{lem:esccond})}\end{proof}

As discussed after the statement of \ref{lem:esccond}, this completes the proof of Claim \ref{clm:escit} in Case \ref{case:escreg}.

\begin{case}\label{case:escsing} $\theta$ is singular in $N$.\end{case}
This case works similarly; we just explain the differences. Let $\mubar=(\cof(\thetabar))^M$ and $\left<\theta_\alpha\right>_{\alpha<\mubar}\in M$ be an increasing sequence of successor cardinals of $M$ converging to $\thetabar$. We claim
\begin{equation}\label{eqn:escsing} \Ult(N,E)\sats\exists \sigma,\gamma,\left<\pi_\alpha, \omi_\alpha\right>_{\alpha<\mubar},\end{equation}
\[ \ \ \ \ [\sigma:\Ult(M,\Ebar)\to\J^\es_{\gamma},\ \pi_\alpha :M\to \J^\es_{\omi_\alpha},\]
\[ \ \ \ \ \ \pi_\alpha\rest\theta_\alpha = \sigma\rest\theta_\alpha]. \]
From this we get that $N$ has a similar $\sigma$ and sequence of $\pi_\alpha$'s. Say $\kappa'<\thetabar$. In iterating $\ph$, we use $\sigma$ as a copy map to lift $\Ult(M,\Ebar)$, and when using an extender with crit $\kappa'$, we use $\pi_\alpha$ to lift $M$, where $\alpha$ is least such that $(\kappa'^+)^M\leq\theta_\alpha$.

To see (\ref{eqn:escsing}), note (\ref{eqn:escshift}) from the previous case is still true. However, if $\pi$ is unbounded in $\theta$, we may not actually have that $\pi\rest\thetabar\in\Ult(N,E)$. So we might not have a $\sigma'\in\Ult(N,E)$ agreeing with $\pi\rest\thetabar$. The existence of $\sigma^*$ on the outside will be sufficient though. By the same argument as in Case \ref{case:escreg}, for each $\alpha<\mubar$ there is $\omi_\alpha<\theta$ and $\pi'_\alpha\in N\int\Ult(N,E)$, with $\pi'_\alpha:M\to N|\omi'_\alpha=\Ult(N,E)|\omi'_\alpha$, and $\pi'_\alpha\rest\theta_\alpha=\pi\rest\theta_\alpha
=\sigma^*\rest\theta_\alpha$. Now $\Ult(N,E)$ has the tree of attempts to simultaneously build a $\sigma$ and sequence of $\pi_\alpha$'s with the desired properties. The existence of $\sigma^*$ and the $\pi'_\alpha$'s shows this tree is illfounded, and therefore $\Ult(N,E)$ has a branch.

This completes the proof for Case \ref{case:escsing}.
\renewcommand{\qedsymbol}{$\Box$(Claim \ref{clm:escit})}\end{proof}

\begin{clm}\label{clm:esccoh}
$M|\taubar=\Ult(M,\Ebar)|\taubar$ and $\thetabar=\taubar$.\end{clm}
\begin{proof}
Recall $\thetabar=|\etabar|^M$, where $\etabar$ is the least disagreement between $M$ and $\Ult(M,\Ebar)$. So it suffices to show $\thetabar=\taubar$. Note $\thetabar\leq\taubar$, since $\Ebar$ is coded by a subset of $\taubar$ in $M$.

A comparison of $M$ with $\ph=(M,\Ult(M,\Ebar),\thetabar)$ only uses extenders with index above $\thetabar$. So by Claim \ref{clm:escit}, there is a successful one, producing trees $\Tt$ on $M$ and $\Uu$ on $\ph$. Let $U=\Ult(M,\Ebar)$. Similarly to the proof of Subclaim \ref{sclm:cond.phit} in Lemma \ref{lem:esccond}, since $M$ is coded by $\Th(M)=\Th(U)$, the same final model $Q$ is produced by $\Tt$ and $\Uu$, $Q$ is above $U$ in $\Uu$, and we have branch embeddings $i^\Tt:M\to Q$ and $i^\Uu:U\to Q$. $\thetabar\leq\crit(i^\Uu)$ by construction.

Assume $\thetabar<\taubar$; then
\[ \thetabar<\etabar<(\thetabar^+)^M\leq\taubar \]
as $\taubar$ is an $M$-cardinal. Now $\etabar$ is a cardinal of $Q$, as it indexes an extender used during the comparison. Also $(\thetabar^+)^M=(\thetabar^+)^U$ since $\her^M_\taubar\sub U$. So
\[ \thetabar\leq\crit(i^\Uu)<\etabar<(\thetabar^+)^U. \]
But then the first extender $E^\Uu_\alpha$ hitting $U$ on the main branch cannot be complete over $U$, because it only measures sets in $U|\etabar$. As $\Uu$'s main branch does not drop, this is a contradiction.\renewcommand{\qedsymbol}{$\Box$(Claim \ref{clm:esccoh})}\end{proof}
\noindent \emph{Notation.} We proceed to show $\Ebar$ is on $\es^M$, from the
first order properties of $M$ and the previous two claims. Since we will no
longer refer directly to objects at the $N$ level, we drop the bar notation.\\

The first paragraph of the proof of Claim \ref{clm:esccoh} still holds, giving
trees $\Tt$ on $M$ and $\Uu$ on $\ph=(M,U,\tau)$ with common last model $Q$. Let
$i=i^\Uu$, $j=i^\Tt$. Then since $M$ is pointwise definable, $i\com i_E=j$. In the case $E$ is
type 1 or 3, so that $\tau=\nu_E$, we can now easily complete the proof. Since
$\crit(i^\Uu)\geq\nu_E$, we have the first extender $E^\Tt_\alpha$ used on the
main branch of $\Tt$ is compatible with $E$. Since $\lh(E^\Tt_0)>\nu_E$, $\lh(E^\Tt_\alpha)>\nu_E$ also, so by the initial segment condition, $E$
is on $\es^M$. This gives \ref{cor:coarse}, but we haven't
finished proving \ref{thm:coarseDodd}.

So assume $E$ is type 2. Let $Q$ be the last model of the comparison and $b,c$
the main branches on $U,M$ respectively. Let $t$ be the Dodd parameter of $E$.

\begin{clm}\label{clm:escG}
\begin{itemize}\item[]
\item[(a)] There is only one extender $G$ used in $c$, which is
type 2, with $\nu_G=i(\max t)+1$.
\item[(b)] Suppose $P$ is a model along $b$ before $Q$.
Then
$i_{U,P}(\max t)\geq\crit(i_{P,Q})$.
\end{itemize}\end{clm}

\begin{proof}\setcounter{sclm}{0} Toward (a), we first show that $i(\max t)+1\leq\nu_G$ by finding fragments of $G$ in $Q$. By Dodd solidity of $E$,
\[ W=E\rest\max t\in U. \] 
We claim that
\[ E_j\rest i(\max t)=i(W)\in Q.\]
To see this, consider $W$ as the set of pairs
\[ (A,i_E(A)\int\max t) \]
such that $A\sub\crit(E)$ and $A\in M$. Since $\crit(E)<\crit(i)$ and $j=i\com i_E$, $i(W)$ is the set of pairs
\[ (i(A),i(i_E(A)\int\max t)) = (A,j(A)\int i(\max t)) \]
such that $A\int\pow(\crit E)\int M$, which is equivalent to $E_j\rest i(\max t)$. But now if $\nu_G<i(\max t)+1$, then
\[ G = E_j\rest\nu_G = i(W)\rest\nu_G\in Q. \]
On the other hand, $G\notin\Ult(M,G)$, which contradicts the fact that $Q$ and $\Ult(M,G)$ have the same subsets of $\nu_G$. So $i(\max t)+1\leq\nu_G$.

Now we prove $\nu_G\leq i(\max t)+1$ and (b) together. If (b) holds let $P=Q$, and if (b) fails let $P$ be the first counterexample to (b) along $b$. Either way, $i(\max t)=i_{U,P}(\max t)$. Moreover, all generators of extenders used in $\Uu$ along the branch leading to $P$ are below $i(\max t)$. Since $U=\Hull^U(\max t+1)$,
\[ P=\Hull^P(i(\max t)+1).\]
Since $\crit(i_{P,Q})>i(\max t)$,
\[ (i(\max t)^+)^Q\sub\Def^Q(i(\max t)+1).\]
Therefore no generators of $E_{i\com i_E}=E_j$ lie between $i(\max t)+1$ and $(i(\max t)^+)^Q$. So letting $G$ be the first extender hitting $M$ along $c$, $G$ cannot have generators in that region either, so $\nu_G\leq i(\max t)+1$. This gives (a).

Therefore
\[ P=\Hull^P(i(\max t)+1)=\Hull^Q(i(\max t)+1)=\Hull^{\Ult(M,G)}(\nu_G) =
\Ult(M,G), \]
since $i^\Uu_{P,Q}$ and $i^\Tt_{\Ult(M,G),Q}$ have critical points above $i(\max t)+1=\nu_G$. This gives $P=\Ult(M,G)$. But these models appeared during a comparison, which implies they are the common final model, $Q$, giving (b).\renewcommand{\qedsymbol}{$\Box$(Claim \ref{clm:escG})}\end{proof}

We now know $i(t)$ is a set of generators of $G$. We need to investigate more carefully their roles in generating $G$. At this point, it's not clear that $i(t)$ is the Dodd parameter of $G$. We need to introduce a variant of the Dodd parameter and projectum, more analogous to the standard parameter and projectum. $i(t)$ is in fact this parameter. We'll also establish that iterations preserve this parameter and projectum nicely, which will allow us to trace the origins of $G$ in $\Tt$.

\begin{dfn}\label{dfn:Dsp}\index{Dodd fragment} Let $\pi:R\to S$ be a $\Sigma_0$-elementary embedding between premice. Suppose $\pi$ is
cardinal preserving and $\mu=\crit(\pi)$ inaccessible in $R$. Suppose $E_\pi\notin S$, where $E_\pi$ is the extender
derived from $\pi$ of length $\pi(\mu)$. The \emph{Dodd-fragment parameter} of $\pi$, denoted
$s_\pi=s=\{s_0,\ldots,s_{k-1}\}\in\OR^{<\om}$, is defined recursively as follows.
Give $s\rest i=\{s_0,\ldots,s_{i-1}\}$, $s_i$ is the largest $\alpha\geq(\mu^+)^R$ such that
\[ E_\pi\rest (\alpha\un s\rest i)\in S, \]
if such exists. $k$ is large as possible (note $s_{i+1}<s_i$). The
\emph{Dodd-fragment projectum} of $\pi$, denoted $\sigma_\pi$, is then the sup of $(\mu^+)^R$ and
all $\alpha$ such that
\[ E_\pi\rest (\alpha\un s_\pi)\in S. \]
(Note that since $\pi(\mu)$ is an $S$-cardinal, $s_\pi$ and $\sigma_\pi$ are
in fact determined by $\pi(R|\mu)$.) One can also give a characterization as that of the Dodd parameter and projectum given in \ref{fact:Dodd}.

Given a premouse $P$ active with extender $E$, the Dodd-fragment
parameter and projectum of $E$ (or of $P$) are the parameter and projectum of $i^P_E$.

We'll often refer to the Dodd-fragment parameter and projectum collectively as
the \emph{Dodd-fragment ordinals}.\end{dfn}

\begin{clm}\label{clm:escDj} The Dodd-fragment ordinals of
$i_G:M\to Q=\Ult(M,G)$ are $s_G=i(t_E)=i(t)$ and $\sigma_G=\tau_E=\tau$.\end{clm}

\begin{proof} As $E$ is Dodd-sound, $s_E=t$ and $\sigma_E=\tau$. (Since $E$ is coded by $E\rest\tau\un t$ and $E\notin U$, the Dodd-fragment ordinals can't be any higher.)

Suppose $E\rest X\in U$. Then
\[ i(E\rest X) = G\rest i(X). \]
This is just as for $i(E\rest\max t)=G\rest i(\max t)$, shown at the start of proving Claim \ref{clm:escG} (note we now have $j=i_G$). So letting $X=(t_E)_k\un t_E\rest k$,
\[ G\rest \left[(i(t_E))_k\un i(t_E)\rest k\right]\in Q. \]
Similarly, if $\tau_E>(\crit(E)^+)^M$ and $\alpha<\tau_E$ then $G\rest\alpha\un i(t_E)\in Q$. Therefore, sufficient fragments of $G$ with generators ``below'' $\tau_E\un i(t_E)$ are in $Q$, to witness Claim \ref{clm:escDj} - we just need to see that $G\rest\tau_E\un i(t_E)$ is not in $Q$.

Suppose $s_G\rest k=i(t_E)\rest k$ but $(s_G)_k>i(t_E)_k$. Then
\[ G \rest \left[(i(t_E)_k+1)\un (i(t_E)\rest k)\right]\in Q. \]
But this gives
\[ E_{i\com i_E}\rest\tau_E\un i(t_E)=G\rest\tau_E\un i(t_E)\in Q, \]
Since $\crit(i)\geq\tau_E$, this fragment of $E_{i\com i_E}$ is isomorphic to
$E\rest\tau_E\un t_E$, which fully determines $E$, and is coded as a subset of
$\tau_E$, But
$Q$ and $U$ agree about such sets, so $E\in U$, a contradiction. This argument
shows $s_G=i(t_E)$ and $\sigma_G=\tau_E$.\renewcommand{\qedsymbol}{$\Box$(Claim \ref{clm:escDj})}\end{proof}

The next claim will motivate the rest of the proof.

\begin{clm}\label{clm:escDsound}
\begin{itemize}\item[]
\item[(a)] If $G$ is Dodd-sound then $G=E$;
\item[(b)] If $G$ is not the active extender of the model it is
taken from, then $G=E$ is on $\es^M$.
\end{itemize}\end{clm}
\begin{proof} By Dodd-soundness, $\sigma_G\un s_G$ generates $G$. By Claim
\ref{clm:escDj}, commutativity and that $\tau_E\leq\crit(i)$,
\[ E \rest \tau_E\un t_E \iso G \rest \tau_E\un i(t_E)=G\rest\sigma_G\un s_G =
G.
\]
This gives (a). For (b), $G$ must be Dodd-sound by \ref{fact:Doddsound}. Thus $G=E$, and
$U=\Ult(M,E)=Q$, so there is no movement on the $U$ side during
the comparison. Recall $\eta$ is the least disagreement between $M$ and $U$. Since $\tau_G<\eta\leq\lh(G)$, $\eta$ cannot
be a cardinal in the model $G$ comes from. Since $\eta$ indexes
the least disagreement, it must be that $G$ is indexed at
$\eta$ in $M$.\renewcommand{\qedsymbol}{$\Box$(Claim \ref{clm:escDsound})}\end{proof}

So we now assume that $G$ is the active extender of the model $R$ from which it
is taken. Let $R^*$ be $\Core_\omega(R)$. Notice the iteration from $R^*$ to $R$
drops immediately to degree 0 (thanks to Ralf Schindler for pointing out this
simplification of our original argument). This is because $G\rest\sigma_G\un
s_G$ is a $\bfSigma_1$ subset of $R$ missing from $R$, so
$\rho_1^R\leq\sigma_G=\tau_E$. But the comparison started above $\tau_E$, so ultrapowers of degree $\geq 1$ lead to models with first projectum strictly above $\tau_E$.

Let $G^*$ be $R^*$'s active extender. The following lemmas establish how the
Dodd fragment ordinals of $G$ are related to those of $G^*$.
Note though that by \ref{fact:Doddsound}, $G^*$ is Dodd-sound. The first two lemmas are
(\cite{covering1woodin}, 2.1.4), which is based on (\cite{fsit}, 9.1).

\begin{lem}\label{lem:Dsp3} Let $P$ be an active premouse with $F=F^P$, and $H$
a short
extender over $P$ with $\crit(H)<\nu_F$. Let $W=\Ult_0(P,H)$ and $i_H:P\to W$ be the
canonical embedding. If $A_H\sub\OR$, $A_F\sub\nu_F$ suffice
as generators for $H$ and $F$ respectively, then
\[ A_H\un i_H``A_F \]
suffices as generators for $F^W$. (If $H$ measures more sets than are in $P$, the hypothesis on $A_H$ should be taken with respect to $P$, so $W=\{i^P_H(f)(a)\ |\ f\in P\ \&\ a\in A_H^{<\om}\}$.)\end{lem}

\begin{proof} Let $\alpha<\nu_{F^W}$; we want to generate $\alpha$ using ordinals
in $A_H\un i_H``A_F$. We have the ultrapower maps $i_H$, $i_F:P\to\Ult_0(P,F)$
and $i_{F^W}:W\to\Ult_0(W,F^W)$. Let $\psi:\Ult_0(P,F)\to\Ult_0(W,F^W)$ be given
by the shift lemma. Then
\begin{equation}\label{eqn:comm} \psi\com i_F = i_{F^W}\com i_H\ \ \&\ \
\psi\rest\nu_F=i_H\rest\nu_F. \end{equation}
Let $f\in P$ and $\betabar\in A_H^{<\om}$ be such that $\alpha=i_H(f)(\betabar)$. Let
$g\in P$ and $\gammabar\in A_F^{<\om}$ be such that $f=i_F(g)(\gammabar)$. Let
$\gammabar^*=i_H(\gammabar)$. Then using (\ref{eqn:comm}),
\[ \alpha = \psi(f)(\betabar)=\psi(i_F(g)(\gammabar))(\betabar) = (\psi\com
i_F)(g)(\gammabar^*)(\betabar) = i_{F^W}(i_H(g))(\gammabar^*)(\betabar). \]
\renewcommand{\qedsymbol}{$\Box$(Lemma \ref{lem:Dsp3})}\end{proof}

\begin{lem}\label{lem:Dsp4} Let $P$ be a type 2 premouse with $F=F^P$ Dodd
sound. Suppose $H$ is a short
extender over $P$ with $\crit(H)<\tau_F$. Let
$W=\Ult_0(P,H)$. Let $i_H:P\to W$ be the canonical embedding. Then
$F^W$ is Dodd sound, $t_{F^W}=i_H(t_F)$ and
$\tau_{F^W}=\sup(i_H``\tau_F)$.\end{lem}

\begin{proof} By Lemma \ref{lem:Dsp3}, we have
\[ A = i_H(t_F) \un \sup i_H``\tau_F \]
suffices as generators for $F^W$. Conversely, if $F'$ is a fragment of $F$ then
$i_H(F')$ is a fragment of $F^W$ (being a fragment is $\Pi_1$). Applying this to witnesses $F'$ to the Dodd soundness of $F$, one sees that $W$ has the desired fragments of $F^W$.\renewcommand{\qedsymbol}{$\Box$(Lemma \ref{lem:Dsp4})}\end{proof}

\begin{lem}\label{lem:Dsp5} Let $P$ be a type 2 premouse and $F=F^P$. Suppose $H$ is a short extender over $P$ with $\sigma_F\leq\crit(H)$, such that $P$ agrees with $W=\Ult_0(P,H)$ about $\pow(\sigma_F)$. Let $i_H:P\to W$ be the canonical embedding. Then $F^W$ is not Dodd sound;
$s_{F^W}=i_H(s_F)$ and
$\sigma_{F^W}=\sigma_F$.\end{lem}

\begin{proof} Again $i_H$ maps fragments of $F$ to fragments of $F^W$. If any stronger fragments of $F^W$ are in $W$ then in fact
\[ F^W \rest \sigma_F\un i_H(s_F) \in W. \]
Now $\sigma_F\leq\crit(H)$ so this fragment is isomorphic to
$F\rest(\sigma_F\un s_F)$, which isn't in $P$. Since $P$ agrees with $W$ about $\pow(\sigma_F)$, it's not in $W$ either.

Since $\crit(H)$ is an $i_H(s_F)$-generator with respect to $F^W$, $F^W$ is not Dodd sound. (Note $F$ and $F^W$ have the same critical point, and $P$ and $W$ agree through its successor.)\renewcommand{\qedsymbol}{$\Box$(Lemma \ref{lem:Dsp5})}\end{proof}

Now we return to the origins of $G$. Part (c) of the final claim completes the proof of the theorem.

\begin{clm}\label{clm:escG*=E}\begin{itemize}\item[]
\item[(a)] $\crit(R^*\to R)\geq\tau_{G^*}=\sigma_G=\tau_E$;
\item[(b)] $R^*\ins M$;
\item[(c)] $E=G^*$ and $E$ is on $\es^M$.
\end{itemize}\end{clm}

\begin{proof}
\begin{itemize}\item[]
\item[(a)] Applying \ref{lem:Dsp4} and \ref{lem:Dsp5} to the (degree 0) branch leading from $R^*$ to $R$, we get $\tau_{G^*}\leq\sigma_G$. Let $\zeta=\crit(R^*\to R)$. If $\zeta<\tau_{G^*}$, then
\[ \sigma_G>i_{R^*,R}(\zeta)>\eta>\tau_E, \]
using the lemmas for the first inequality. This contradicts Claim \ref{clm:escDj}. So $\crit(R^*\to R)\geq\tau_{G^*}$, so \ref{lem:Dsp5} gives $\tau_{G^*}=\sigma_G$.
\item[(b)] $R^*$ is a proper segment of some model on the tree;
$\rho^{R^*}_1\leq\tau_{G^*}<\eta$; $\eta$ is a cardinal of all models on the tree other than $M$.
\item[(c)] By Claim \ref{clm:escDj}, $i(t_E)=s_G$. By (a) and Lemma
\ref{lem:Dsp5}, $s_G=i_{R^*,R}(s_{G^*})$. So $G^*$ is isomorphic to
$G\rest(\tau_{G^*}\un s_G)=G\rest(\tau_E\un i(t_E))$, which is isomorphic to
$E$. So by (b), $E$ is on $\es^M$.
\end{itemize}
\renewcommand{\qedsymbol}{$\Box$(Claim \ref{clm:escG*=E})(Theorem \ref{thm:coarseDodd})}
\end{proof}
\renewcommand{\qedsymbol}{}
\end{proof}

\begin{cor}\label{cor:seqdef} Let $N\sats\ZFC$ be a mouse. Then
\begin{itemize}
 \item[(a)] Every normal measure of $N$ is on $\es^N$.
 \item[(b)] If $\kappa$ is strong or Woodin in $N$, then it is so via extenders on $\es^N$.
 \item[(c)] If $\kappa$ is strong in $N$, then $\es^N$ is definable from $N|\kappa$ in the universe of $N$.
 \item[(d)] If $\tau$ is an $N$-cardinal, and $P$ is a premouse projecting to $\tau$, extending $N|\tau$, and $N$ has a total wellfounded short extender $E$ such that $P\ins\Ult(N,E)$, then $P\ins N$.
\end{itemize}
\end{cor}

\begin{proof}
We leave (a)-(c) to the reader, and just prove (d). Assume $N,E,P,\tau$ are as there. Let $M$, $\Ebar$, $\taubar$, etc., be defined as in the proof of the theorem. The iterability of the phalanx
$(M,\Ult(M,\Ebar),\taubar)$ works as there. Then $\Pbar\ins Q$, since
$\crit(i_{\Ubar,Q})\geq\taubar$. But all extenders used in the comparison have
index $>\taubar$, and $\Pbar$ projects to $\taubar$, so $\Pbar\ins M$.
\renewcommand{\qedsymbol}{$\Box$(Corollary \ref{cor:seqdef})}
\end{proof}

Steel noticed that combining the last statement of \ref{cor:seqdef} with an argument of Woodin's, we also get the following.

\begin{cor}\label{cor:AC} Let $N\sats\ZFC$ be a mouse and $\theta$ be an uncountable cardinal in $N$. Then $N|(\theta^+)^N$ is definable from parameters over $\her_{(\theta^+)^N}^N$. Therefore $L(\pow(\theta)^N)\sats\AC$.\end{cor}

\begin{proof}
Now we prove \ref{cor:AC}. There are two cases: first suppose that there is no cut-point $\gamma$ of $N$ such that $\theta\leq\gamma<(\theta^+)^N$. (Recall $\gamma$ is a cut-point\index{cut-point} of $N$ iff whenever $E$ is on $\es_+^N$, if $\crit(E)<\gamma$ then $\lh(E)<\gamma$.) This implies that there are unboundedly many $\gamma<(\theta^+)^N$ indexing a total extender. So by \ref{cor:seqdef}(d), $\her_{(\theta^+)^N}^N$ can just look at total extenders $E$ such that $N|\theta\ins\Ult(N|\theta,E)$ and $\Ult(N|\theta,E)$ is wellfounded, to determine the levels of $N$ projecting to $\theta$. (If $\Ult(N|\theta,E)$ is wellfounded then $E$ is countably complete in $N$.)

The proof of the second case is due to Woodin. Let $\tilde{M}$ be an enumeration of all countable elementary submodels of levels of $N$, that are members of $N$. $\tilde{M}$ is coded as a subset of $\om_1^N$ in $N$. Let $\gamma$ be a cut-point of $N$ with $\theta\leq\gamma<(\theta^+)^N$. Let $P$ be a sound premouse, $N|\gamma\ins P$, $P$ projecting to $\gamma$. Then we claim
\[ N\sats P\ins\J[\es]\ \iff\ \textrm{every countable elementary submodel of } P \textrm{ is on } \tilde{M}. \]
If $P$'s submodels are on the list, then let $P\in P'\ins N$, with $P'$ projecting to $\gamma$. In $N$, let $X\elem P'$ be countable, with $P\in X$. The collapses $\Pbar$ of $P$ and $\Pbar'$ of $P'$ are on the list, so can be compared in $V$. Standard arguments show $\Pbar\ins\Pbar'$, so $P\ins P'$. Everything actually took place in $\her_{(\theta^+)^N}^N$, so we're done.
\renewcommand{\qedsymbol}{$\Box$(Corollary \ref{cor:AC})}
\end{proof}

\begin{rem}\label{rem:deg1dfp}
The rest of this section contains a generalization of Lemma \ref{lem:Dsp4} to higher degrees, which we'll need in the next section.\end{rem}

\begin{lem}\label{lem:gen1hull} Let $P$ be a type 2 premouse, $\gamma_P$ the index of the largest proper initial segment of $F=F^P$ on $\es^P$, $\kappa=\crit(F)$, and $X\sub\nu^P$. Suppose
\[ \kappa,\gamma_P\in\Lambda = \{ x\in P\ |\ \exists f\in P,a\in X^{<\om}\ (x=[a,f]^P_F)\}. \]
Then $\Lambda=\Def_1^P(X\un(\kappa^+)^P)$.
\end{lem}
\begin{proof}
Clearly $\Lambda\sub H=\Hull_1^P(X\un(\kappa^+)^P)$ in general; we want to see $H\sub\Lambda$.

Let $F'=F\rest X$ and
\[ \tilde{\pi}:\Ult(P|(\kappa^+)^P,F')\to\Ult(P|(\kappa^+)^P,F) \]
be the canonical map. Because $\gamma_P\in\rg(\tilde{\pi})$, so is $\nu_F$; in fact $\tilde{\pi}(\nu_{F'})=\nu_F$. Let $R'=\Ult(P|(\kappa^+)^P,F')$ and
\[ P' = (R'|(\nu_{F'}^+)^{R'},F') \]
(where $F'$ is coded amenably as for a premouse). Let $\pi:P'\to P$ be the restriction of $\tilde{\pi}$. So $\Lambda=\rg(\pi)$. We claim $\pi$ is $\Sigma_1$-elementary, which gives the lemma, since $\pi\rest(\kappa^+)^P=\id$.

For this, note $\tilde{\pi}\com i_{F'}=i_F$, which implies that for $A\in\pow(\kappa)\int P$ and $\gamma<\OR^{P'}$,
\[ \pi(i_{F'}(A)\int\gamma) = i_F(A)\int\pi(\gamma). \]
Also $\pi(\nu_{F'})=\nu_F$. So $\pi$ respects the predicates for the active extenders of $P$ and $P'$, and $\pi(\gamma_{P'})=\gamma_P$, where $\gamma_P$ indexes the last proper segment of $F'$ on $\es^{P'}$. Since also $\pow(\kappa)\int P\sub P'$, it follows that $\pi$ is cofinal in $P$ (see (\cite{outline}, 2.9)). Therefore $\pi$ is $\Sigma_1$-elementary.\renewcommand{\qedsymbol}{$\Box$(Lemma \ref{lem:gen1hull})}
\end{proof}

\begin{lem}\label{lem:Dsp6} Let $P$ be a $k$-sound, type 2, Dodd sound premouse, with $k\geq 1$. Let $\mu=\crit(F^P)$. Let $H$ be a short extender over $P$ with $\crit(H)<\rho_k^P$. If $R=\Ult_k(P,H)$ is wellfounded, then it is also Dodd sound. Moreover, $t_R=i_H(t_P)$; if $(\mu^+)^P=\tau_P$ then $\tau_R=i_H(\tau_P)$; if $(\mu^+)^P<\tau_P$ then $\tau_R=\sup i_H``\tau_P$.\end{lem}

\begin{proof} Of course if $R$ is $(0,\om_1,\om_1+1)$-iterable, the lemma holds by \ref{fact:Doddsound}. But we need to know it's true more generally, as we'll apply it where $R$ is obtained as an iterate of $P$, via a strategy that, abstractly, we won't know is an $\fully$-strategy.

If $k>1$, elementarity considerations easily imply $R$ is Dodd sound.

Suppose $k=1$. Let $F=F^P$, $\mu=\crit(F)$, $t=t_F$ and $\tau=\tau_F$. 
If $\tau=(\mu^+)^P$ then $F$ is generated by $t\un\{\mu\}$, a $\Pi_2$ condition, preserved by $i_H$. Also $i_H$ maps fragments of $F$ to fragments of $F^R$, so $R$ is Dodd sound.

So assume $\tau_F>(\crit(F)^+)^P$. Then $R$ has all fragments of $F^R$ with generators ``below'' $i_H(t)\un\sup i_H``\tau$. So it is enough to show this set generates $F^R$. Deny. Let $\kappa=\crit(H)$. Let $[a,f]$ represent an $i_H(t)$-generator for $F^R$ at least $\sup(i_H``\tau)$. Here $f:\kappa^n\to P$ is given by a $\Sigma^P_1(\{q\})$ term for some $q\in P$. Since the statement ``$\alpha$ is an $i_H(t)$-generator'' is $\Pi_1$, $f(u)$ must be a $t$-generator for measure one many $u$'s. In particular, $\rg(f)$ includes a $\tau$-cofinal set of $t$-generators.

Now \ref{lem:gen1hull} applies with $X=\lambda\un t$, where $\lambda<\tau$ is large enough that $\mu,q,\gamma_P\in\Lambda$ (notation as in \ref{lem:gen1hull}). But this contradicts the last sentence of the previous paragraph.\renewcommand{\qedsymbol}{$\Box$(Lemma \ref{lem:Dsp6})}
\end{proof}

\begin{cor}\label{cor:tau}
If $P$ is a type 2, Dodd sound premouse, then $\tau_P=\max((\mu_P^+)^P,\rho_1^P)$.
\end{cor}
\begin{proof}
This is like the proof of \ref{lem:gen1hull}, but Dodd soundness is used to show the small ultrapowers, and so corresponding $\Sigma_1$ theories, are in $P$.\renewcommand{\qedsymbol}{$\Box$(Corollary \ref{cor:tau})}
\end{proof}

\pagebreak
\section{Cohering Extenders}\label{sec:cohere}
In this section we analyse the situation when an extender in a mouse ``fits'' on the mouse's sequence, or an iterate thereof. The first theorem stated is just a special case of \ref{thm:cohere} below. 

\begin{thm}\label{thm:easycoh}
Let $N$ be an $\fully$-iterable mouse satisfying ``there is no largest cardinal, and $E$ is a wellfounded total short extender'', and suppose $(N||\lh(E),E)$ is a premouse. Then $E$ is on $\es^N$.\end{thm}
\begin{proof} See \ref{thm:cohere}.\renewcommand{\qedsymbol}{$\Box$(Theorem \ref{thm:easycoh})}\end{proof}

The full theorem below assumes we have extender $E$ fitting on $\es^P$ for some internal iterate $P$ of $N$, and under enough hypotheses, shows that $E$ is in fact on the $P$-sequence. An analogous theorem was proven by Schimmerling and Steel in (\cite{maxcore}, \S2) (using the theory of \cite{cmip}), assuming there's no inner model with a Woodin, and a measurable exists. They showed that given a wellfounded iterate $W$ of $K$, and an extender $E$ which fits on the $W$ sequence (in that $(W||\lh(E),E)$ is a premouse), such that $E$ has sufficient weak background certificates, $E$ was either used in the iteration or is on $\es^W_+$. They also proved another version dealing with sound mice projecting to $\om$ (or at least below all critical points on the mouse's sequence) instead of $K$. Our theorem has a similar statement, but the iteration tree $\Tt$ and extender in question must be inside the mouse, and the requirement of background certificates is replaced by demanding the extender produce a wellfounded, class size ultrapower when used as a normal extension of $\Tt$.

Here is a related question.

\begin{ques}[Steel] Suppose $V=L[\es]$, and $L[\es]$ is fully iterable. Let $E$ be an extender. Is $E$ the extender of an iteration map?\end{ques}

The theorem does deal in part with self-iterable mice, and we first discuss definability of iteration strategies a little.

\begin{dfn}\index{above} An iteration tree $\Tt$ is \emph{above} $\rho$ if $\rho\leq\crit(E)$ for each extender $E$ used in $\Tt$. Let $P$ be a premouse. $\Sigma$ is an $\eta$-iteration strategy for $P$ \emph{above} $\rho$ if $\Sigma$ works as an iteration strategy for normal trees on $P$, above $\rho$, of length $<\eta$.\end{dfn}

\begin{dfn}\index{tame}
A premouse $P$ is \emph{tame overlapping} $\alpha$ if whenever $E$ is on $\es^P_+$ and $\crit(E)<\alpha\leq\delta<\lh(E)$,
\[ P|\lh(E)\sats \delta\textrm{ is not Woodin.} \]
A premouse is \emph{tame} if it is tame overlapping $\alpha$ for every $\alpha$.
\end{dfn}

\begin{lem}\label{lem:uniquestrat} Let $P$ be an $\om$-sound premouse, with $\rho^P_\om\leq\rho$, and suppose $P$ is tame overlapping $\rho$. Then $P$ has at most one $(\rho^++1)$-iteration strategy above $\rho$.\end{lem}
\begin{proof}
 This generalizes the uniqueness of $(\om_1+1)$-iteration strategies for $\om$-sound mice projecting to $\om$, so we just outline the differences. Suppose $\Sigma$, $\Gamma$ are two strategies, $\Tt$ is above $\rho$, via both $\Sigma$ and $\Gamma$, and $b=\Sigma(\Tt)\neq c=\Gamma(\Tt)$. The least difference between $M^\Tt_b$ and $M^\Tt_c$ must occur within the Q-structure of the active side. Therefore one can compare the two sides, forming normal continuations $\Tt\conc b\conc\Uu_\Sigma$ and $\Tt\conc c\conc\Uu_\Gamma$: $\delta(\Tt)$ remains Woodin up to the length of any extender used, so critical points never go below $\rho$. The remaining details are left to the reader.
\renewcommand{\qedsymbol}{$\Box$(Lemma \ref{lem:uniquestrat})}\end{proof}

Now suppose $W$ is a premouse modelling $\theory$ (or $\ZFmin$ suffices for our purposes), and either there's no largest cardinal in $W$, or the largest does not have measurable cofinality in $W$. Suppose $\Sigma\sub W$ is a set of pairs $(\Tt,b)\in W$, definable from parameters over $W$, and $\alpha\in W$. Then we claim that the statement
\[ W\sats``\Sigma \textrm{ is an }\alpha\textrm{-iteration strategy for the universe''} \]
makes sense. That is, there is a formula $\varphi$, with a predicate for $\Sigma$ and constant for $\alpha$, which asserts this, for premice of the form of $W$. We leave the justification of this to the reader in the case that there is no largest cardinal in $W$. Otherwise, let $\delta$ be the largest. In computing $j(\delta)$ for some non-dropping iteration map $j$, the hypothesis on $\delta$ means we need only consider functions bounded in $\delta$. If $\delta<\gamma<\mu$ and $g:\delta\to\gamma$ is a bijection, then to compute $j(\gamma)$ we need consider only functions $g\com f$, where $f$ is bounded in $\delta$. Thus the membership relation is set-like in $W$ (note all the trees and branches being considered are themselves members of $W$). Since $W\sats\KP^*$, wellfoundedness of iterates is therefore first-order. Moreover, if $\Sigma$ is indeed an iteration strategy, then all models on trees via $\Sigma$ are transitive classes of $W$.

For the theorem, we need to slightly strengthen the hypothesis that $\delta$ not have measurable cofinality in $W$.

\begin{dfn}\index{almost measurable} Let $N$ be a premouse, $\kappa\in N$. $\kappa$ is \emph{almost measurable} in $N$ if $(\kappa^+)^N\in N$ and for unboundedly many $\alpha<(\kappa^+)^N$, $\es^N_\alpha$ has critical point $\kappa$.\end{dfn}

\begin{thm}\label{thm:cohere}
Let $N$ be a premouse satisfying $\theory$. Suppose that in $N$, $\mu$ is regular and is not the successor of a singular cardinal with almost measurable cofinality. Suppose that $\Tt\in N|\mu$ is a normal iteration tree with final model $P$, and either
\begin{itemize}
\item[(a)] $N$ is $\fully$-iterable and $\Tt$ is a finite tree on $N|\mu$, or
\item[(b)] $(\chi^+)^N<\mu$, and in $N|\mu$: $\Sigma$ is a $(\chi^++1)$-iteration strategy for the universe, definable from parameters, $\Tt$ is via $\Sigma$, and $\lh(\Tt)<(\chi^+)$, or
\item[(c)]
\[ N\sats\textrm{I'm tame overlapping } \chi,\ \chi^+<\mu,\ N|\mu \textrm{ is } (\chi^++1)\textrm{-iterable via } \Sigma\]
\[ \textrm{and }\Tt \textrm{ is based on } N|\chi^+, \textrm{ via } \Sigma, \textrm{ above } \chi.\]
\end{itemize}
In any case, suppose further that $E\in N$, $(P||\lh(E),E)$ is a premouse; $\sup_\beta\lh(E^\Tt_\beta)<\lh(E)$; $\xi$ is the least $\xi'$ such that $\crit(E)<\nu^\Tt_{\xi'}$; $i^\Tt_{0,\xi}$ exists; $E$ measures all of $\pow(\crit(E))\int M^\Tt_\xi$; $\Ult(M^\Tt_\xi,E)$ is wellfounded.

Then $E$ is on $\es_+^P$.
\end{thm}

\begin{rem}
One might formulate other suitable self-iterability hypotheses, though our proof depends on the definability of the strategy. We will use the following generalization in \S\ref{sec:stacking}. \ref{thm:cohere} works just as well if $\Sigma$ is only an iteration strategy above some cut-point $\eta$ of $N$. We leave the generalization of the proof to the reader. We will also use a cross between (a) and (c) in \S\ref{sec:stacking}.
\end{rem}

\begin{rem} There is a counterexample beyond a superstrong. Suppose $M$ is a structure satisfying the premouse axioms, except for the requirement that no extender on $\es_+^M$ be of superstrong type. Let $E$ be a total type 2 extender on $M$'s sequence, and $\kappa$ the largest cardinal of $\Ult(M,E)$ below $\lh(E)$. Suppose $\kappa$ is superstrong in $\Ult(M,E)$, and $F$ is an extender on $\es^{\Ult(M,E)}$ witnessing superstrongness. Note that $F$ (normally) applies exactly to $M|\lh(E)$, with degree 0. But $\Ult(M|\lh(E),F)$ is an active premouse, distinct from $\Ult(M,E)|\lh(F)$, with the same reduct. (However, since both extenders are on the sequence of iterates that aren't far from being normal, this doesn't appear to be a strong failure. Moreover, in Jensen's $\lambda$-indexing, the extenders are not indexed at the same point.)\end{rem}

\begin{proof}[Proof of \ref{thm:cohere}]\setcounter{clm}{0}
The proof will take the remainder of this section. The overall approach is like that of \ref{thm:coarseDodd}, but the details differ. We'll first give a brief introduction (plenty of details are to follow). Assume for simplicity that $E$ coheres with $N$ (so $\Tt=\left<N\right>$). As in \ref{thm:coarseDodd}, we'd like to compare $\Ult(N,E)$ with $N$, producing a common final model, commuting maps, and a critical point at least $\nu_E$ on the $\Ult(N,E)$ side. Again, we'll replace $N$ with a pointwise definable $M$, and (with bar notation as usual), compare $M$ with $(M,\Ult(M,\Ebar),\thetabar)$. But to prove this phalanx is iterable (before knowing the conclusion of the theorem) the highest we can set $\thetabar$ at is the largest cardinal of $\Pbar||\lh(\Ebar)$. This has the undesired consequence that if $\Ebar$ (or $E$) is type 2, so that $\thetabar<\nu_{\Ebar}$, then the iteration map on $\Ult(M,\Ebar)$ may move generators of $\Ebar$. We dealt with this in \ref{thm:coarseDodd} by analyzing how the witnesses to Dodd soundness move under iteration, on both sides of the comparison. Here things are more complicated, as $E$ may well be Dodd unsound. However, by similar analysis, we can still match up the Dodd-fragment ordinals of $\Ebar$ with those of $i_{M,Q}$. We'll see that the subextender $\Ebar\rest X$ they generate appears on a normal iterate of $M$, derived from the comparison tree on $M$. But then we can factor $i_{\Ebar}$ and $i_{M,Q}$ through $\Ult(M,\Ebar\rest X)$, and the analysis can be repeated, with this ultrapower replacing $M$. Continuing in this way, through larger subextenders of $\Ebar$, we will see that $\Ebar$ itself appears on a normal iterate $M'$ of $M$. But we're assuming (for this paragraph) that $\Ebar$ actually coheres $\es^M$, and therefore $M'=M$, finishing the proof. We now drop the simplifying assumptions of this paragraph and commence with the details.

We may assume that $\mu$ is the largest cardinal in $N$.

Because $\cof^N(\delta)$ isn't almost measurable, the discussion preceding theorem statement shows that $\Ult(M^\Tt_\xi,E)$ is also a transitive class of $N|\mu$. (If $\cof^N(\delta)$ isn't measurable but \emph{is} almost, then maybe $E$  is type 1 and $\crit(E)=i^\Tt_{0,\xi}(\cof^N(\delta))$.) So the wellfoundedness of this model is a definable notion over $N|\mu$.

In case (b) or (c), we claim that in $N$, $\Sigma$ is a $(0,\chi^++1)$-iteration strategy on the universe, for $N|\mu$-based trees. For otherwise there is a non-dropping branch to an illfounded model $P$ by $\Sigma$. Let $j$ be the iteration map. Since $\mu$ is $N$'s largest cardinal, $(^{<\mu}\mu)^N\in N$ and $N\sats\theory$, the membership relation of $P$ is set-like in $N$, and $N$ has sequences $\left<f_i\right>\in N$, $\left<a_i\right>\in N|\mu$ such that $j(f_{i+1})(a_{i+1})\in j(f_i)(a_i)$. Since all critical points are $<\mu$ and $\mu$ is regular, $\left<f_i\right>$ can be converted to $\left<\bar{f}_i\right>\in N|\mu$ which still witness illfoundedness; contradiction.

In any case, (a), (b) or (c), let $\Tt'$ be the liftup of $\Tt$ to a degree-$0$ tree on $N$. As above, the models of $\Tt'$ and $\Ult(M^{\Tt'}_\xi,E)$ are wellfounded and are transitive classes of $N$.

In case (c), note that the restriction of $\Sigma$ to a $\chi^+$-iteration strategy above $\chi$ on $N|(\chi^+)^N$ is a point in $N|\mu$, definable over $N|\mu$ without parameters. Indeed, in $N|\mu$ it is the unique such strategy which extends to a $\chi^++1$-strategy, in that each tree of length $\chi^+$ has a cofinal (so wellfounded) branch. (Even if $\mu=(\chi^++)^N$, this is expressible over $N|\mu$.) By \ref{lem:uniquestrat} applied in $N$, such a strategy must agree with $\Sigma$.

By the preceding discussion, being a counterexample to the theorem is first-order over $N|\mu$, so we may assume ours is the least such. Let $M=\Hull_\om^{N|\mu}(\empty)$, $\pi:M\to N|\mu$, and let $\pi(\Ttbar)=\Tt$, etc. $\pi$ provides liftup maps from $\Phi(\Ttbar)$ to $\Phi(\Tt)$ (i.e. $\pi\rest M^{\Ttbar}_{\gamma}$ is the liftup from its domain to $M^\Tt_{\pi(\gamma)}$). (The reader can check that $\pi$'s elementarity ensures that it lifts a normal continuation of $\Phi(\Ttbar)$ to a normal continuation of $\Phi(\Tt)$.) So we know the $\Phi(\Ttbar)$ part of the claim to follow.

Let $\thetabar$ be the largest cardinal of $\Pbar||\lh(\Ebar)$ and $\pi(\thetabar)=\theta$.

\begin{clm}\label{clm:iter}
The phalanxes $\Phi(\Ttbar)$ and
\[ (\Phi(\Ttbar),\Ult(\Rbar,\Ebar),\thetabar)\]
are $\om_1+1$-iterable, in $V$ for case (a), or in $N$ for case (b) or (c). (A tree on the latter should return to the appropriate model of $\Phi(\Ttbar)$ for crits in below $\thetabar$.)\end{clm}

\begin{proof}
The iterability of the corresponding phalanx in \ref{thm:coarseDodd} used heavily that $E$ was strong below $\tau_E$. Because we don't have this here, we need another approach. We'll retain $\pi$ as our lifting map for the models on $\Phi(\Ttbar)$, and show that there's an embedding of $\Ult(\Rbar,\Ebar)$ into a level of $\Pbar$ agreeing with $\pi$ below $\thetabar$. We'll also discuss what happens when an extender returns to $\Pbar$ or $\Ult(\Rbar,\Ebar)$.

Let $R=M^\Tt_\xi$ and $R'=M^{\Tt'}_\xi$, the models $E$ returns to in $\Tt$ and $\Tt'$; note $\Ult(R,E)=\Ult(R',E)|\mu$, and
\[ \pi\rest\Ult(\Rbar,\Ebar):\Ult(\Rbar,\Ebar)\to\Ult(R,E).\]
Now by \ref{lem:esccond},
\[ N\sats \all\lambda<\mu\ [\textrm{if}\ \lambda\ \textrm{is a cardinal then}\ \Hull_\om^{N|\mu}(\lambda)\ins\J^\es\ ]. \]
$\Ult(R',E)$ satisfies the same statement, since the iteration maps fix $\mu$ and are at least $\Pi_1$ elementary. So letting 
\[ H'=\Hull_\om^{\Ult(R,E)}(\lh(E)),\]
$\J_1(H')$ projects to $\lh(E)$, and $H'\pins\Ult(R,E)$. $\Ult(R,E)$ satisfies Lemma \ref{lem:escpassive} since $N$ does, so we can get an $\gamma<\lh(E)$, with $\pi(\nu_{\Ebar})<\gamma$, and $\sigma$ such that $\sigma:\J_1(H)\to\J_1(H')$ is the uncollapse map of $\Hull_\om^{\J_1(H')}(\gamma)$, with $\crit(\sigma)=\gamma$, and $\J_1(H)\pins\Ult(R,E)$. But $\Ult(R,E)$ and $P$ agree below $\lh(E)$. So we get $\psi:\Ult(\Rbar,\Ebar)\to H\pins P$, with $\psi\rest\nu_{\Ebar}+1=\pi\rest\nu_{\Ebar}+1$. (In fact if $R=N|\mu$, then $\cof^N(\lh(E))=(\crit(E)^+)^N$, so then we could take $\pi``\lh(\Ebar)\sub\gamma$, so that $\psi$ agrees with $\pi$ below $\lh(\Ebar)$, but this doesn't help in the end.)

Consider building a tree $\Uu$ on $(\Phi(\Ttbar),\Ult(\Rbar,\Ebar),\thetabar)$ (so $\lh(E^\Uu_0)>\thetabar$). By our hypothesis on $\lh(\Ebar)$, $\sup_\beta\lh(E^{\Ttbar}_\beta)\leq\thetabar$, so $\psi(E^\Uu_0)$ is a suitable exit extender from $P$. Except for using an extender applying to $\Pbar$ or $\Ult(\Rbar,\Ebar)$ (discussed next), the remarks above discussing the iterability of $\Phi(\Ttbar)$ show the copying process works.

Suppose $F=E^\Uu_\alpha$ has crit $\eta$, where $\sup_\beta\nu^{\Ttbar}_\beta\leq\eta<\thetabar$, so that $F$ applies to $\Pbar$. Let $\Uu'$ be the liftup of $\Uu$, $\psi_\alpha$ be the $\alpha^\nth$ liftup map and $F'=\psi_\alpha(F)$. Since $\thetabar<\lh(E^\Uu_0)$ and $\thetabar$ is a cardinal of $\Ult(\Rbar,\Ebar)$, $F$ measures exactly $\pow(\eta)\int\Pbar|\thetabar=\pow(\eta)\int\Pbar|\lh(\Ebar)$. Since $\theta<\lh(E^{\Uu'}_0)<\lh(E)$, $F'$ measures exactly $\pow(\pi(\eta))\int P|\theta=\pow(\pi(\eta))\int P|\lh(E)$. Since $\pi(\lh(\Ebar))=\lh(E)$, $\pi$ preserves the correct level to drop to. Moreover, $\pi$ agrees with $\psi$, and therefore $\psi_\alpha$, below $\thetabar$, so certainly on the measured subsets of $\eta$. Therefore the shift lemma applies.

Suppose $F$ applies to $\Ult(\Rbar,\Ebar)$. If it causes a drop, clearly $\psi$ preserves the level to drop to. If not, we simply lift to $H$, applying $N|\mu$'s freely dropping iterability (\S\ref{sec:copying}).

Thus the lifted tree is a (freely dropping) normal continuation of $\Tt$.
\renewcommand{\qedsymbol}{$\Box$(Claim \ref{clm:iter})}\end{proof}

\begin{rem}
The reason we can't move the exchange ordinal above $\thetabar$ is as follows. Suppose we do move it up. It may be that $\Pbar|\lh(\Ebar)$ projects to $\thetabar$. (In fact, the theorem implies that it does.) Suppose we want to apply an extender with crit $\thetabar$. This extender must go back to $\Pbar$, and it measures exactly $\Pbar|\lh(\Ebar)$, so there is a drop to that level. $\psi$ was arranged with its range bounded in $\lh(E)$ (which was needed to get $H\ins P$), and
\[ \theta=\psi(\thetabar)<\psi(\lh(E^\Uu_0))<\lh(E).\]
Since $\theta$ is largest cardinal below $\lh(E)$ in $P||\lh(E)$, the least level projecting to $\theta$ after $\psi(\lh(E^{\Uu}_0))$ is strictly below $\lh(E)$. But $\pi(P|\lh(\Ebar))=P|\lh(E)$, so $\pi$ does not preserve the correct level to drop to. So the copying process breaks down. We get iterability with the exchange ordinal at $\thetabar$, but at the cost of possibly moving some generators of $\Ebar$ whilst iterating.

Also notice we used the fact that $\Sigma$ is a full class strategy for $N|\mu$. Suppose instead $\Sigma\in N|\mu$ is an $N|\alpha$-based strategy for $N|\mu$, where $\alpha<\mu$. Then the above lifting doesn't work, as $\pi$ can produce extenders with length $>\alpha$.
\end{rem}

\noindent\emph{Notation.} We proceed to show $\Ebar$ is on $\es^{\Pbar}$, from the first order properties of $M$ and the phalanx iterability. Since we will no longer refer directly to objects at the $N$ level, we drop the bar notation.\\

We expect the following statement to be extractable from the rest of the proof. It is a variation of (\cite{cmip}, 8.6), which is a related statement about $K$, although there, the exchange ordinal $\theta=\nu_E$ instead.

\begin{conj} Let $M$ be an $\om$-sound mouse projecting to $\om$, and $\Tt$ a correct tree on $M$, with last model $P$. Suppose $(P||\lh(E),E)$ is a premouse, and $\crit(E)$ is such that $E$ would apply normally to $R$, with degree $k$ (if it were on the $P$-sequence). Let $\theta$ be the largest cardinal of $P||\lh(E)$. Then $E$ is on $\es^P_+$ iff the phalanx $(\Phi(\Tt),\Ult_k(R,E),\theta)$ is $\om_1+1$-iterable.\end{conj}

Let $U=\Ult(R,E)$. Since $U$ agrees with $P$ below $\lh(E)>\theta$, the iterability gives us successful comparison trees $\Uu$ on $(\Phi(\Tt),\Ult(R,E),\theta)$ and $\Vv$ on $\Phi(\Tt)$. As in Lemma \ref{lem:esccond}, since $M$ is pointwise definable, the same model $Q$ is produced on both sides. We claim that $Q$ is above $U$ in $\Uu$, and there is no dropping leading to $Q$ on either tree. This follows by standard arguments using the hull property and that $M$ is pointwise definable. For example, suppose $Q$ is fully sound, so that there is no dropping leading to it. Suppose $Q$ is not above $U$; let $M^\Tt_\alpha$ be the root of $Q$ in $\Uu$. Let $j:M^\Tt_\alpha\to Q$ be the iteration map and let $E^\Uu_\beta$ be the first extender of $j$. Note that $\nu_F\geq\lh(E)$ for all $F$ used in $\Uu$. Note $\Ult(M^\Tt_\alpha,E^\Uu_\beta\rest\gamma)=\Hull_\om^Q(\gamma)$. Thus the initial segment condition for $E^\Uu_\beta$ shows $\crit(j)$ has the $Q$-hull property (i.e. $\pow(\crit(j))^ Q\sub\Hull_\om^Q(\crit(j))$), but that for $(\crit(j)^+)^Q\leq\alpha<\lh(E)$, $\alpha$ does not. This just depends on $Q$, and gives that $M^\Tt_\alpha$ is the root of $Q$ in $\Vv$, and the embedding is the same on both sides. But then the corresponding first extenders were compatible, and used in a comparison; contradiction. Similar reasoning shows there is no dropping leading to $Q$, and that $R$ is the root of $Q$ in $\Vv$.

\begin{clm}{\label{clm:1extG}} There's only one extender $G$ used on $\Vv$'s branch from $R$ to $Q$ and $\nu_G=i^\Uu(\gamma+1)$, where $\gamma$ is the largest generator of $E$.
\begin{proof}
This isn't quite as in \ref{thm:coarseDodd} as we don't know $E$ is Dodd sound (and it may not be), and there is a difficulty if $E$ is just beyond a type Z segment.

Let $G$ be the first extender used on the $\Vv$ branch from $R$ to $Q$. Let $\sigma\in[U,Q]_\Uu$ be such that $\sigma$ is least with $M^\Uu_\sigma=Q$ or $\crit(i^\Uu_{\sigma,Q})>i^\Uu_{0,\sigma}(\gamma)=\gamma'$. Then $\nu_G\leq\gamma'+1$: otherwise $G\rest\gamma'+1\in M^\Uu_\sigma$, but $G\rest\gamma'+1$ is the length $\gamma'+1$ extender derived from $i^\Uu_{U,\sigma}\com i_E$ (as $i^\Uu\com i_E=i^\Vv_{R,Q}$ and $\crit(i^\Uu)>\crit(E)$). So by definition of $\sigma$, $M^\Uu_\sigma=\Ult(R,G\rest\gamma'+1)$, and $G\rest\gamma'+1$ collapses the successor of $\gamma'$ in $M^\Uu_\sigma$.

Now if $E\rest\gamma\in\Ult(R,E)$, then $i^\Uu_{R,\sigma}(E\rest\gamma)\in Q$ is compatible with $G$, so $\nu_G\geq\gamma'+1$. Otherwise $E$ has a last proper segment, $E\rest\delta+1$, and it is type Z. So $E\rest\delta\in\Ult(R,E)$. Let $\sigma_1\in[U,Q]_\Uu$ be such that $\sigma_1$ is least with $M^\Uu_{\sigma_1}=Q$ or $\crit(i^\Uu_{\sigma_1,Q})>i^\Uu_{R,{\sigma_1}}(\delta)=\delta_1$. As in the previous paragraph $\nu_G\geq\delta_1+1$. Since $E\rest\delta+1$ is type Z, $\gamma=(\delta^+)^{\Ult(R,E\rest\delta)}$, so $E\rest\delta$ collapses $\gamma$ to $\delta$. So $\gamma_1=i^\Uu_{R,\sigma_1}(\gamma)$ is collapsed below $\delta_1+1$ in $M^\Uu_{\sigma_1}$. Therefore if $M^\Uu_{\sigma_1}\neq Q$, then $\crit(i^\Uu_{\sigma_1,Q})>\gamma_1$. So in fact $\sigma_1=\sigma$ and $\gamma_1=\gamma'$. Now
\[ \crit(i^\Vv_{\Ult(R,G),Q})\geq\nu_G\geq\delta_1+1. \]
But then $M^\Uu_\sigma$, $Q$ and $\Ult(R,G)$ agree about $\pow(\delta_1)$, so in fact $ \crit(i^\Vv_{\Ult(R,G),Q})>\gamma'$. So \[ M^\Uu_\sigma=\Hull^{M^\Uu_\sigma}_\om(\gamma'+1)=\Hull^Q_\om(\gamma'+1)=\Hull^{\Ult(R,G)}_\om(\gamma'+1)=\Ult(R,G). \]
So in fact $M^\Uu_\sigma=Q=\Ult(R,G)$. However, $\gamma'\notin\Hull^Q_\om(\gamma')$ since $\gamma$ is a generator of $E$ (a fact coded in $\Th^{\Ult(R,E)}_\om(\{\gamma\})$), so $\gamma'+1=\nu_G$.
\renewcommand{\qedsymbol}{$\Box$(Claim \ref{clm:1extG})}\end{proof}\end{clm}

We now begin to work on analysing the Dodd unsoundness of an extender $F$ used in an iteration tree, looking at the Dodd-fragment ordinals of $F$, and of extenders that brought about $F$'s Dodd unsoundness, and so on. The following elementary fact underlies this.

\begin{dfn}\label{dfn:extcomp}\index{$\Ult_k(F,E)$}\index{$\com_k$} Let $P$ be a premouse with $F=F^P$, and $F'$ an extender over $P$. Then $F'\com_k P$ denotes $\Ult_k(P,E)$; $F'\com_k F$ or $\Ult_k(F,F')$ denotes $F^{\Ult_k(P,F')}$. In the absence of parentheses, we take association of $\com_k$ to the right; i.e.
\[ F'\com_k F\com_l Q = F'\com_k (F\com_l Q).\]
\end{dfn}

\begin{lem}[Associativity of Extenders]\label{lem:extass}
Let $P_u$ and $P_l$ (upper and lower) be active premice, $E_u=F^{P_u}$, $E_l=F^{P_l}$, such that $\crit(E_u)>\crit(E_l)$, and $Q$ a premouse. Suppose $E_u$ measures exactly $\pow(\crit(E_u))\int P_l$ and $\crit(E_u)<\nu_{E_l}$, and likewise $E_l$ with respect to $Q$ (though $Q$ may be passive), and $\crit(E_l)<\rho^Q_k$. Then\index{$\com_k$}
\[ (E_u\com_0 E_l)\com_k Q= E_u\com_k(E_l\com_k Q), \]
where $\com_i$ denotes applying the extender on the left to the object on the right with degree $i$. Moreover, $i^Q_{E_u\com_0 E_l}=i^U_{E_u}\com i^Q_{E_l}$ (the degree $k$ ultrapower embeddings).\end{lem}
\begin{proof}
Let $U=\Ult_k(Q,E_l)$, and $i^{P_l}_{E_u}$ and $i^U_{E_u}$ be the ultrapower embeddings associated to $\Ult_0(P_l,E_u)$ and $\Ult_k(U,E_u)$ respectively. Note $\crit(E_u)<\rho_k^U$, $\lh(E_l)=\OR^{P_l}$ is a cardinal in $U$, and $P_l||\OR^{P_l}=U|\OR^{P_l}$. So $\Ult_0(P_l,E_u)$ agrees below its height with $\Ult_k(U,E_u)$, and $i^{P_l}_{E_u}=i^U_{E_u}\rest\OR^{P_l}$. As in \ref{lem:Dsp3}, $\sup i_{E_u}``\nu_{E_l}=\nu_{E_u\com E_l}$. Moreover, for $A\sub\crit(E_l)$ in $P_l$,
\begin{equation}\label{eqn:assoc} i_{E_u\com E_l}(A)\int\nu_{E_u\com E_l}=i_{E_u}(i_{E_l}(A)\int\nu_{E_l})\int\nu_{E_u\com E_l} \end{equation}
by the definition of $0$-ultrapower. (This associativity generates the associativity overall.)

Now we define a map from $\Ult_k(Q,E_u\com E_l)$ to $\Ult_k(U,E_u)$. For $\tau_q$ a $k$-term defined with parameter $q\in Q$, $a\in\nu_{E_l}^{<\om}$, and $b\in\nu_{E_u}^{<\om}$, let
\[ [\tau_q,i_{E_u}(a)\un b]^Q_{E_u\com E_l}\goesto [\tau'_{(i^Q_{E_l}(q),a)},b]^{\Ult_k(Q,E_l)}_{E_u}, \]
where $\tau'$ is naturally derived from $\tau$ by converting some arguments to parameters. (As in \ref{lem:Dsp3}, $\nu_{E_u}\un i_{E_u}``\nu_{E_l}$ suffices as generators for $E_u\com E_l$.) Los' Theorem and (\ref{eqn:assoc}) shows this is well-defined and $\Sigma_k$-elementary, and it's clearly surjective. This isomorphism commutes with the ultrapower embeddings, which gives $i^Q_{E_u\com_0 E_l}=i^U_{E_u}\com i^Q_{E_l}$.
\renewcommand{\qedsymbol}{$\Box$(Lemma \ref{lem:extass})}\end{proof}

\begin{cor}\label{cor:nass}
Let $P_1,\ldots,P_n$ be active premice, $E_i=F^{P_i}$, with $\crit(E_i)>\crit(E_{i+1})$, and $Q$ be a premouse. Suppose $E_i$ measures exactly $\pow(\crit(E_i))\int P_{i+1}$ and likewise for $E_n$ with respect to $Q$, and that $\crit(E_n)<\rho^Q_k$. Then, with $\com=\com_0$,
\[ ((\ldots(E_1\com E_2)\com\ldots)\com E_n)\com_k Q=E_1\com_k(\ldots\com_k(E_n\com_k Q)); \]
\[ i^Q_{((\ldots(E_1\com E_2)\com\ldots)\com E_n)}=i_{E_1}\com\ldots\com  i^Q_{E_n}. \]
\end{cor}

\begin{dfn}[Dodd core]\label{dfn:Dcore}\index{Dodd core}
Let $G$ be an extender. The \emph{Dodd core} of $G$ is
\[ \core_D(G) = G\rest\sigma_G\un s_G. \]
\end{dfn}

\begin{rem}\label{rem:Dcore}
Let $S$ be a premouse such that every extender on $\es_+^S$ is Dodd sound. Suppose $\Ww$ is a normal tree on $S$ and $G$ is on the $M^\Ww_\alpha$ sequence. Then if $G$ is not Dodd sound, lemmas \ref{lem:Dsp4} to \ref{lem:Dsp6} show $\core_D(G)$ is the active extender of $(M^*)^\Ww_{\beta+1}$, where $\beta$ is the least $\beta'$ such that $\beta'+1\leq_W\alpha$ and $M^\Ww_{\beta'+1}$'s active extender is not Dodd sound. Equivalently, $\beta$ is $W\pred(\beta'+1)$ for the least $\beta'$ such that $\beta'+1\leq_W\alpha$, there's no model drop from $(M^*)^\Ww_{\beta'+1}$ to $M^\Ww_\alpha$ and $\crit((i^*)^\Ww_{\beta'+1,\alpha})\geq\tau_F$, where $F=F^{(M^*)^\Ww_{\beta'+1}}$.
\end{rem}

\begin{dfn}[Core sequence]\label{dfn:coreseq}\index{core sequence}
Let $P',Q'$ be premice, and $j:P'\to Q'$ be fully elementary. The \emph{core sequence} of $j$ is the sequence $\left<Q_\alpha,j_\alpha\right>$ defined as follows. $Q_0=P'$ and $j_0=j$. Given $j_\alpha:Q_\alpha\to Q'$, if $j_\alpha=\id$ or the Dodd-fragment ordinals for $j_\alpha$ are not defined, we finish; otherwise let $s,\sigma$ be the Dodd-fragment ordinals for $j_\alpha$ and $H_\alpha=\rg(j_\alpha)$. Let $Q_{\alpha+1}=\Hull_\om^{Q'}(H_\alpha\un s\un\sigma)$ and $j_{\alpha+1}:Q_{\alpha+1}\to Q'$ the uncollapse map. Take the natural limits at limit ordinals. Since $\crit(j_\alpha)\in H_{\alpha+1}$, the process terminates.
\end{dfn}

\begin{dfn}[Damage]\label{dfn:damage}\index{damage structure}\index{$\dam$}\index{$<_\dam$}
Let $S$, $\Ww$ and $G$ be as in \ref{rem:Dcore}. We define the \emph{damage structure} of $G$ in $\Ww$, denoted $\dam^\Ww(G)$, and the relation $<^\Ww_\dam$. If $G$ is Dodd-sound, then $\dam^\Ww(G)=\empty$. Otherwise let $G$, $\core_D(G)$ be on the sequences of $M^\Ww_{\alpha_G}$, $M^\Ww_\alpha$ respectively. Let $\dam^\Ww(G)$ be the sequence with domain
\[ \{\beta\ |\ \beta+1\in(\alpha,\alpha_G]_W\}, \]
such that
\[ \dam^\Ww(G):\beta\goesto (E^\Ww_\beta,\dam^\Ww(E^\Ww_\beta)). \]
(Obviously there is a little superfluous information here.)

This is well-defined since $\dom(\dam^\Ww(G))\sub\alpha_G$ - the deeper levels of the damage structure involve extenders used earlier in $\Ww$. Now let
\[ \beta<^\Ww_\dam\gamma\ \iff\ E^\Ww_\beta<^\Ww_\dam E^\Ww_\gamma\ \iff\  E^\Ww_\beta\in\transcl(\dam^\Ww(E^\Ww_\gamma)). \]
\end{dfn}

Here is a diagram of a typical damage structure. An extender $E$ is represented by the symbol $\rfloor$, with $\crit(E)$ and $\lh(E)$ corresponding to the bottom and top of the symbol respectively. $E<_\dam F$ iff $E$ is pictured to the left of $F$. So Dodd sound extenders have no extenders to their left.
\[
\setlength{\unitlength}{1mm}
\begin{picture}(100,62)(30,0)
\put(100,0){\line(-1,0){2}}
\put(100,0){\line(0,1){62}}
\put(90,1){\line(-1,0){2}}
\put(90,1){\line(0,1){25}}
\put(90,27){\line(-1,0){2}}
\put(90,27){\line(0,1){12}}
\put(90,41){\circle*{0}}
\put(90,42){\circle*{0}}
\put(90,40){\circle*{0}}
\put(90,49){\line(-1,0){2}}
\put(90,49){\line(0,1){6}}
\put(90,56){\circle*{0}}
\put(90,57){\circle*{0}}
\put(90,58){\circle*{0}}
\put(80,2){\line(-1,0){2}}
\put(80,2){\line(0,1){5}}
\put(80,8){\line(-1,0){2}}
\put(80,8){\line(0,1){11}}
\put(80,20){\circle*{0}}
\put(80,21){\circle*{0}}
\put(80,22){\circle*{0}}
\put(80,28){\line(-1,0){2}}
\put(80,28){\line(0,1){7}}
\put(73,31){\circle*{0}}
\put(70,31){\circle*{0}}
\put(67,31){\circle*{0}}
\put(64,31){\circle*{0}}
\put(73,5){\circle*{0}}
\put(70,5){\circle*{0}}
\put(67,5){\circle*{0}}
\put(64,5){\circle*{0}}
\put(61,5){\circle*{0}}
\put(58,5){\circle*{0}}
\put(55,5){\circle*{0}}
\put(50,4){\line(-1,0){2}}
\put(50,4){\line(0,1){2}}
\put(73,11){\circle*{0}}
\put(70,11){\circle*{0}}
\put(67,11){\circle*{0}}
\put(64,11){\circle*{0}}
\put(60,10){\line(-1,0){2}}
\put(60,10){\line(0,1){2}}
\put(60,13){\circle*{0}}
\put(60,14){\circle*{0}}
\put(60,15){\circle*{0}}
\put(60,30){\line(-1,0){2}}
\put(60,30){\line(0,1){2}}

\end{picture}
\]
\begin{rem}
In the situation of \ref{dfn:damage}, the remark above shows
\[ \beta<_\dam\gamma\ \ \implies\ \ \kappa_\gamma<\kappa_\beta<\lh_\beta<\lh_\gamma. \]
So if $\beta_1,\beta_2\in\dom(\dam(E^\Ww_\gamma))$, with $\beta_1<\beta_2$, then $\Ww$ has (finished damaging and) used $E^\Ww_{\beta_1}$, before using any of the extenders in the damage structure of $E^\Ww_{\beta_2}$. That is, if $\alpha_i<_\dam\beta_i$ ($i=1,2$), then $\alpha_1<\alpha_2$, since
\[ \lh_{\alpha_1}<\lh_{\beta_1}\leq((\kappa_{\beta_2})^+)^{E^\Ww_{\beta_2}}<\kappa_{\alpha_2}<\lh_{\alpha_2}. \]
(Though it may be that the tree has already dropped to $\core_\om(M^\Ww_{\beta_2})$ at some stage before applying $E^\Ww_{\beta_1}$.)\end{rem}

\emph{Notation.} Let $\lambda\in\OR$ and $F_\alpha$ be an extender for $\alpha<\lambda$. Let $P$ be a premouse. We define $P_\alpha$ for $\alpha\leq\lambda$. $P_0=P$, $P_{\alpha+1}=\Ult_k(P_\alpha,F_\alpha)$, and take direct limits at limit $\alpha$, as far as this definition makes sense. Then
\[ _{\alpha<\lambda}\ldots\com_k F_\alpha\com_k\ldots\com_k P \]
denotes $P_\lambda$. (So association is to the right, as above.) Similarly, if $F=F^P$,
\[ _{\alpha<\lambda}\ldots\com_k F_\alpha\com_k\ldots\com_k F \]
denotes $F^{P_\lambda}$. We will also make use of index sets $X\sub\OR$, replacing $\lambda\in\OR$.

\begin{lem}\label{lem:Gcoreseq}
Let $M$ be a premouse with Dodd sound extenders on its sequence and $\Ww$ be a normal tree on $M$. Let $G=E^\Ww_{\alpha_G}$. Let $\left<F_\alpha\right>_{0\leq\alpha}$ enumerate
\[ \{\core_D(E^\Ww_\beta)\ |\ \beta\leq_\dam\alpha_G\} \]
with increasing critical points. Let $G_0=\id$, and
\[ G_\alpha = _{1\leq\gamma<\alpha}\ldots\com F_\gamma\com\ldots\com F_0. \]
Let $G_\infty$ be the limit of all $G_\alpha$'s.

The definition makes sense for each $\alpha$; i.e. $F_\alpha$ measures precisely $\pow(\crit(F_\alpha))^{G_\alpha}$ and the ultrapowers are wellfounded. Also $\rho_1^{G_\alpha}\leq\crit(F_\alpha)$.

Further, $G_\infty=G$, and for $1\leq\alpha\leq\infty$ there's a normal tree $\Ww_\alpha$ on $M$ and $\zeta\in\OR$ and $k\in\om$ such that
\begin{itemize}
 \item $G_\alpha$ is on the sequence of $\Ww_\alpha$'s last model,
 \item $\Ww_\alpha\rest\zeta+1=\Ww\rest\zeta+1$,
 \item $\lh(\Ww_\alpha)=\zeta+k+1$,
 \item $\crit(G)<\nu^{\Ww}_\zeta$,
 \item if $k>0$ then $\crit(G)<\nu^{\Ww_\alpha}_{\zeta}$.
\end{itemize}
Finally, suppose $P$ is a $m$-sound premouse and $\crit(G)<\rho_m^P$. Then
\[ \Ult_m(P,G) = _{0\leq\gamma}\ldots\com_m F_\gamma\com_m\ldots\com_m P, \]
and the associated embeddings (with domain $P$) agree.
\end{lem}

\begin{proof}
Let $F_\alpha=\core_D(E^\Ww_{\alpha'})$. $\Ww_\alpha$ is the part of $\Ww$ doing the damage preceding $E^\Ww_{\alpha'}$, then popping stack back out of the damage structure. Say ${\alpha'}\in\dom(\dam(E^\Ww_{\beta'}))$, and $\zeta=\Tt\pred(\alpha'+1)$. If $G=E^\Ww_{\beta'}$, then we set $\Ww_\alpha=\Ww\rest\zeta+1$.

Assume $E^\Ww_{\beta'}\neq G$. Still $\Ww_\alpha\rest\zeta+1=\Ww\rest\zeta+1$, but here $k>0$. Since $\alpha'+1\leq_\Ww\beta'$ and there is no drop in $(\alpha'+1,\beta']_\Tt$ and $E^\Ww_{\beta'}$ isn't Dodd sound, there is a subextender $I$ of $E^\Ww_{\beta'}$ on the sequence of $M^{\Ww_\alpha}_\zeta$. If $\alpha'$ is least in $\dom(\dam(E^\Ww_{\beta'}))$ then $I=\core_D(E^\Ww_{\beta'})$ is active on $(M^*)^\Ww_{\alpha'+1}\ins M^\Ww_\zeta$; otherwise $I$ is Dodd unsound and is active on $M^\Ww_\zeta$. Set $E^{\Ww_\alpha}_\zeta=I$. (If $I$ isn't active on $M^\Ww_\zeta$, still $\lh^\Ww_\zeta\leq\lh(I)$ since $E^\Ww_{\alpha'}$ triggers a drop to $M^\Ww_\zeta|\lh(I)$.) Say $\beta'\in\dom(\dam(E^\Ww_{\gamma'}))$. Then $E^\Ww_{\beta'}$ applies to a model on the branch leading to $M^\Ww_{\gamma'}$; moreover there is no drop in $(\beta'+1,\gamma']_\Tt$. Since $E^\Ww_{\beta'}$ and $I$ have the same critical point and measure the same sets, and since $\crit(E^\Ww_{\beta'})<\crit(E^\Ww_{\alpha'})<\nu^\Ww_\zeta$, $I$ also applies to the same (initial segment of the same) model. Moreover, $M^{\Ww_\alpha}_{\zeta+1}$ is active with a subextender of $E^\Ww_{\gamma'}$. If $G=E^\Ww_{\gamma'}$, $\Ww_\alpha$ is finished; otherwise set $E^{\Ww_\alpha}_{\zeta+1}$ to be $M^{\Ww_\alpha}_{\zeta+1}$'s active extender. Note it measures the same sets as $E^\Ww_{\gamma'}$. Continue in this way, always using active extenders, until exiting the damage structure (until applying an extender to a model on the branch leading to $G$). This must finish in finitely many steps.

We now prove inductively that $G_\alpha$ is on the sequence of the last model of $\Ww_\alpha$. For $\alpha=1$, this was already observed. Assume it's true for some $\alpha\geq 1$; we consider $G_{\alpha+1}$.

This is an application of \ref{lem:extass}. Let $\zeta,k$ be as in the definition of $\Ww_\alpha$; other notation is also as there. Say $H_0$ is active on $(M^*)^\Ww_{\alpha'+1}$ and if $k>0$, $H_0=E^{\Ww_\alpha}_\zeta,\ldots,E^{\Ww_\alpha}_{\zeta+k-1}$ apply in $\Ww_\alpha$ to the extenders $H_1,\ldots,H_k$, respectively. (If $k>0$, $H_0=E^{\Ww_\alpha}_\zeta$.) So by induction and applying \ref{cor:nass}, the last model of $W_\alpha$ is active with
\[ G_\alpha = ((H_0\com H_1)\com H_2)\com\ldots\com H_k = H_0\com\ldots\com H_k. \]
From the definition of $\Ww_{\alpha+1}$, the active extender of its last model is
\[ ((F_\alpha\com H_0)\com H_1)\com\ldots\com H_k = F_\alpha\com G_\alpha = G_{\alpha+1}.
\]
Moreover, with an appropriate $m$-sound premouse $P$, \ref{lem:extass} gives $G_{\alpha+1}\com_m P = F_\alpha\com_m G_\alpha\com_m P$, and that the ultrapower maps on $P$ agree. By induction, this is equivalent to applying the $F_\beta$'s, for $\beta\leq\alpha$, to $P$.

Let $\lambda$ be a limit, and suppose we have the hypothesis below $\lambda$. Let $\beta\leq_\dam\alpha_G$ and $F_\lambda$ be the Dodd core of $E^\Ww_{\delta_0}$, with $\delta_0\in\dom(\dam(E^\Ww_\beta))$. Apply the following algorithm, as far as it works. Let $\delta_1$ be largest in $\delta_0\int\dom(\dam(E^\Ww_\beta))$, and let $\delta_{i+2}$ be largest in $\dom(\dam(E^\Ww_{\delta_{i+1}}))$. (This process leads backwards in $\Ww$.) Let $i$ be largest such that $\delta_i$ exists.

Note that $\delta_0$ is not least in $\dom(\dam(E^\Ww_\beta))$, as otherwise it follows from remarks earlier that there is no $\kappa\in(\kappa^\Ww_\beta,\kappa^\Ww_{\delta_0})$ which is the crit of an extender in $\dam(G)$, contradicting $\kappa^\Ww_{\delta_0}$ being enumerated at a limit stage.

Assume $i=0$ for now. Then $\zeta=\sup(\delta_0\int\dom(\dam(E^\Ww_\beta)))$ is a limit. It follows from \ref{dfn:damage} that $E^\Ww_{\delta_0}$ applies to $M^\Ww_\zeta$, a limit node in $\Ww$. Let $H_0$ be the active extender of $M^\Ww_\zeta$, and $H_1,\ldots,H_k$ be as for the successor case (again $k=0$ is possible). We may assume inductively that the lemma applies to $H_0$, so letting $F_\alpha$ be the Dodd core of $H_0$ (equivalently of $E^\Ww_\beta$),
\[ H_0=\ _{\alpha<\gamma<\lambda}\ldots\com F_\gamma\com\ldots\com F_\alpha, \]
The active extender of $M^\Ww_\lambda$'s final model is $(H_0\com H_1)\com\ldots\com H_k=H_0\com H_1\com\ldots\com H_k$, which by the definition of $\Ww_\alpha$ and induction is $H_0\com G_\alpha$. Also inductively, the last statement of the lemma gives
\[ H_0\com G_\alpha =\ _{\alpha\leq\gamma<\lambda}\ldots\com F_\gamma\com\ldots\com G_\alpha = G_\lambda. \]
The ``moreover'' clause (for some $m$-sound $P$) also applies to $H_0$ and $G_\alpha$, which gives it for $G_\lambda$. This finishes the $i=0$ case. If $i>0$ it's almost the same.

Now we show $G_\infty=G$. Clearly
\begin{equation}\label{eqn:Gtree} G =\ _{\gamma'\in\dom(\dam(G))}\ldots\com E^\Tt_{\gamma'}\com \ldots\com\core_D(G), \end{equation}
Let $\gamma',\beta'$ be successive elements of this set, or let $\gamma'$ be the largest. Let $X=\{\alpha\ |\ \kappa^\Tt_{\gamma'}<\crit(F_\alpha)<\kappa^\Tt_{\beta'}\}$. By induction, for any premouse $P$,
\begin{equation}\label{eqn:compass} \Ult_0(P,E^\Tt_{\gamma'}) =\  _{\alpha\in X}\ldots\com F_\alpha\com\ldots\com\core_D(E^\Tt_{\gamma'})\com P \end{equation}
and both embeddings on $P$ agree. So for each $\gamma'$ we can substitute a string of $F_\alpha$'s for $E^\Tt_{\gamma'}$ in (\ref{eqn:Gtree}), and associate to the right, to obtain
\[ G = \ldots\com F_\alpha\com\ldots\com F_0. \]

Finally, suppose $P$ is $m$-sound and $\crit(G)<\rho^P_m$. We claim that
\begin{equation}\label{eqn:ultk} \Ult_m(P,G) =\ _{\gamma'\in\dom(\dam(G))}\ldots\com_m E^\Tt_{\gamma'}\com_m\ldots\com_m\core_D(G)\com_m P,\end{equation}
and the corresponding embeddings agree. This is a straightforward extension of \ref{lem:extass}, and we leave the details to the reader. (Note that it may not make sense to associate \emph{arbitrarily} here, as $\gamma'<\beta'$ might be successive elements of $\dom(\dam(G))$ with $\lh(E^\Tt_{\gamma'})<\crit(E^\Tt_{\beta'})$. But it does make sense to associate to the right.) The degree-$m$ version of (\ref{eqn:compass}) allows us to substitute a string of $F_\gamma$'s for each $E^\Tt_{\gamma'}$, in (\ref{eqn:ultk}), which finishes the proof.
\renewcommand{\qedsymbol}{$\Box$(Lemma \ref{lem:Gcoreseq})}\end{proof}

Let $F_\alpha$, $G_\alpha$ and $\Tt_\alpha$ be as in \ref{lem:Gcoreseq} with respect to $G$ and $\Tt\conc\Vv$. By \ref{lem:Gcoreseq}, the action of $G$ on $R$, which results in the model $Q$, decomposes into the action of the $F_\alpha$'s. Let $j_{\alpha,Q}$ be the factor embedding from $Q_\alpha=\Ult(R,G_\alpha)$ to $Q$ (given by the tail end of $F_\beta$'s). We claim that $\left<Q_\alpha,j_{\alpha,Q}\right>$ is the core sequence of $j:R\to Q$.

For this, we need to see the Dodd-fragment ordinals of $j_{\alpha,Q}$ are $s_{j_{\alpha,Q}}=j_{\alpha+1,Q}(t_{F_\alpha})$ and $\sigma_{j_{\alpha,Q}}=\tau_{F_\alpha}$. Since $\crit(F_{\alpha+1})\geq\tau_{F_\alpha}$, this is obvious if $F_\alpha$ is type 1 or 3, or if $\lh(F_\alpha)<\crit(F_{\alpha+1})$, so assume otherwise. Let $\alpha'\leq_\dam G$ be such that $F_\alpha=\core_D(E^{\Tt\supconc\Vv}_{\alpha'})$. Applying \ref{lem:Gcoreseq} to $E^{\Tt\supconc\Vv}_{\alpha'}$,
\[ \Ult(Q_\alpha,E^{\Tt\supconc\Vv}_{\alpha'}) =\ _{\gamma<_\dam\alpha'}\ldots\com F_\gamma\com\ldots\com F_\alpha\com Q_\alpha, \]
and the embeddings agree. If $\beta$ is least such that $\beta\not<_\dam\alpha$, then $\crit(F_\beta)>\lh(E^{\Tt\supconc\Vv}_{\alpha'})$, so $j_{\alpha,Q}$ is compatible with $E^{\Tt\supconc\Vv}_{\alpha'}$ through $\lh(E^{\Tt\supconc\Vv}_{\alpha'})$. As $F_\alpha$ is Dodd-sound, applying the Dodd-fragment preservation of \ref{lem:Dsp5} to the iteration from $F_\alpha$ to $E^{\Tt\supconc\Vv}_{\alpha'}$ gives what we seek.

\begin{rem}\label{rem:compat} The compatibility of $j_{\alpha,Q}$ with $E^\Tt_{\alpha'}$ through $\nu^\Tt_{\alpha'}$ will be needed later. \end{rem}

\begin{clm}\label{clm:factor} There is $\eps$ such that $Q_\eps=\Ult(R,E)$.\end{clm}

\noindent Assuming the claim, we're essentially done: we know from \ref{lem:Gcoreseq} that $\Tt_\eps$ has $G_\eps$ on the sequence of its final model, $\Tt_\eps$ agrees with $\Tt\conc\Vv$ through $R$ and (since $G$ applies to $R$ in $\Vv$ and $\crit(G)=\crit(G_\eps)$) $G_\eps$ applies to $R$ normally. Since $E$ did too, and $M=\Th(M)$, we get $E=G_\eps$. But then the agreement between $\Tt\conc\Vv$ and $\Tt_\eps$ (as in \ref{lem:Gcoreseq}), and the coherence of $E$ with $P$ ($\Tt$'s last model) implies $\Tt_\eps=\Tt$, so $E$ is on the $P$-sequence, as desired.
\\
\begin{proof}[Claim Proof]\setcounter{case}{0}
We will show inductively $j_{\alpha,Q}$ factors through $U=\Ult(R,E)$, giving $j_{\alpha,U}$, such that the following diagram commutes:
\[
\begin{picture}(100,100)(50,0)
\put(46,0){$Q_\beta$}
\put(59,10){\vector(1,4){15}} 
\put(61,11){\vector(1,1){15}} 
\put(62,8){\vector(3,1){70}} 
\put(75,30){$Q_\alpha$}
\put(75,75){$U$}
\put(135,30){$Q$}
\put(80,40){\vector(0,1){30}} 
\put(90,35){\vector(1,0){37}} 
\put(85,70){\vector(3,-2){47}} 
\put(45,40){$\scriptstyle j_{\beta,U}$}
\put(95,10){$\scriptstyle j_{\beta,Q}$}
\put(111,56){$\scriptstyle i^\Uu$}
\put(82,53){$\scriptstyle j_{\alpha,U}$}
\put(100,40){$\scriptstyle j_{\alpha,Q}$}
\end{picture}
\]
Moreover, $\gamma\in\rg(j_{1,U})$, where $\gamma$ is the top generator of $E$ and if $Q_\alpha\neq U$, then $\crit(j_{\alpha,U})<\theta$. The commutativity of the lower right triangle holds since the $Q_\alpha$'s form the core sequence of $j:R\to Q$.

\begin{case} $\alpha=0$.\end{case}
Clearly this is true since $Q_0=R=\Hull_\om^R(\crit(G)\int\crit(E))$, and $\crit(E)<\theta$ (so $\crit(E)=\crit(G)$).

\begin{case} $\alpha=1$.\end{case}
Remember here $E$ may be just beyond a type Z segment. Let $\gamma$ be the top generator of $E$. Define $s\in\OR^{<\om}$ and $\sigma\in\OR$ by $s_0=\gamma$, then $s_{i+1}$ is obtained from $i_E$ and $s\rest(i+1)$, and $\sigma$ is obtained from $i_E$ and $s$, as for the Dodd-fragment objects. Then $\sigma\leq\theta$. Otherwise it must be that $s<t_E$ in the lexiographic order, since $\tau_E\leq\theta$. Thus $\pi:\Hull_\om^{\Ult(R,E)}(s\un\sigma)\to\Ult(R,E)$ is not the identity. So $\theta<\crit(\pi)\leq\gamma$. But $\pi(\crit(\pi))$ is a cardinal of $\Ult(R,E)$, so is at least $\lh(E)$, contradicting $\gamma\in s$.

Then $i^\Uu(s),\sigma$ are the Dodd-fragment objects $s_G,\sigma_G$ for $G$. For Claim \ref{clm:1extG} showed $i^\Uu(\gamma)+1=\nu_G$, so $(t_G)_0=i^\Uu(\gamma)$. Also $\crit(i^\Uu)>\crit(E)$, and the iteration triangle commutes, so the fragments of $E$ in $\Ult(R,E)$ are mapped to fragments of $G$ in $Q$. So we just need the maximality of  $i^\Uu(s),\sigma$ with respect to $i_G$. If they weren't maximal, then $G\rest(\sigma\un i^\Uu(s))\in Q$. But this is isomorphic to $E\rest(\sigma\un s)$, and is coded as a subset of $\sigma\leq\theta\leq\crit(i^\Uu)$.

So $t_G\un\sigma_G\sub\rg(i^\Uu)$, which gives the factorization of $j_{1,Q}$. We also saw $\gamma\in\rg(j_{1,Q})$.

\begin{case} $\alpha=\beta+1>1$.\end{case}
Suppose the factoring holds at $\beta$, and $Q_\beta\neq U$. So $\kappa=\crit(j_{\beta,U})\leq\gamma$. So in fact $j_{\beta,U}(\kappa)\leq\theta$: otherwise, as in Case 2, $j_{\beta,U}(\kappa)\geq\lh(E)$, contradicting $\gamma\in\rg(j_{1,U})\sub\rg(j_{\beta,U})$.

Suppose at first that $j_{\beta,U}(\kappa)<\crit(i^\Uu)$. Then $j_{\beta,U}(\kappa)=j_{\beta,Q}(\kappa)$. But $j_{\beta,Q}=j_{\beta+1,Q}\com i_{F_\beta}$. The fragments of $i_{F_\beta}$ in $\Ult(Q_\beta,F_\beta)$ are all below $i_{F_\beta}(\kappa)$, and $\sigma_{F_\beta}=\tau_{F_\beta}\leq\crit(j_{\beta+1,Q})$, so $j_{\beta+1,Q}$ maps maximal fragments of $i_{F_\beta}$ to those of $j_{\beta,Q}$. So letting $s=j_{\beta+1,Q}(t),\sigma$ be the Dodd-fragment ordinals of $j_{\beta,Q}$, we have $s\un\sigma\sub\theta\sub\rg(i^\Uu)$. Since $Q_{\beta+1}$ is generated by $\rg(j_{\beta,\beta+1})\un t\un\sigma$, we get $j_{\beta+1,Q}$ factors through $U$.

Now suppose $j_{\beta,U}(\kappa)=\theta=\crit(i^\Uu)$. Let $F$ be the extender of length $\theta$ derived from $j_{\beta,U}$. It can't be that $F\in\Ult(R,E)$ since otherwise $i^\Uu(F)\in Q$ is the extender of length $j_{\beta,Q}(\kappa)$ derived from $j_{\beta,Q}$ - this contradicts $j_{\beta,Q}=j_{\beta+1,Q}\com i_{F_\beta}$, as in the previous paragraph. So we again get $s\un\sigma\sub\theta$, and the factoring.

\begin{case} $\alpha$ is a limit.\end{case}
Since the $Q_\alpha$'s are the core sequence, the commutativity of the maps before stage $\alpha$ makes this case is easy.\\

This completes the proof of factoring. Since the $Q_\alpha$'s eventually reach $Q$, it must be that there is some stage $\eps$ with $Q_\eps=U$.\renewcommand{\qedsymbol}{$\Box$(Claim \ref{clm:factor})(Theorem \ref{thm:cohere})}\end{proof}
\renewcommand{\qedsymbol}{}\end{proof}

\pagebreak
\section{Measures and Partial Measures}\label{sec:meas}
Consider a mouse $N$ satisfying ``$\mu$ is a countably complete measure over some set'' (plus say $\ZFC$). \ref{cor:coarse} and \ref{thm:cohere} give different criteria which guarantee $\mu$ is on $\es^N$; e.g. normality is enough. We will now show that in general, $i_\mu$ is precisely the embedding of a finite iteration tree on $N$. This generalizes Kunen's result on the model $L[U]$ for one measurable, that all its measures are finite products of its unique normal measure.

\subsection*{Finite Support of an Iteration Tree\protect\footnote{\label{ftn:type3}Footnote added January 2013: This section is covered better in \cite{hsstm}, where a correction to \ref{dfn:supp} (and consequently, proofs to follow) is given and a result stronger than \ref{lem:supptree} is established, and more details are provided. The correction involves how premice with active type 3 extenders are handled. The stronger result is essentially used in the proof of \ref{thm:homMn}.}}\label{subsec:fst}

Toward our goal, we need to be able to capture a given element of a normal iterate with a finite normal iteration. That is, we want a finite iteration with liftup maps to the original one, with the given element in the range of the ultimate liftup map. The method is straightforward: find a subset of the tree sufficient to generate the given element, then perform a reverse copying construction to produce the finite tree. One must be a little careful, though, to ensure the resulting tree is normal. This tool is also used in \S\ref{sec:hom}.

\begin{dfn}[Finite Support]\footnote{Footnote added January 2013: See footnote \ref{ftn:type3}.}\label{dfn:supp}\index{support}\setcounter{case}{0} Let $\Tt$ be a normal iteration tree on a premouse $M$ of length $\theta+1$, and let $B\sub M_\theta$ be a finite set. A hereditarily finite set $A$ \emph{supports} $B$ (relative to $\Tt$) if the following properties hold.

Let $M_\alpha=M^\Tt_\alpha$. Then $A\sub\{(\alpha,x)\ |\ \alpha\in\theta+1\ \&\ x\in M_\alpha\}$. Let $(A)_\alpha$ denote the section of $A$ at $\alpha$. Let $S\sub\theta+1$ be the left projection of $A$. Then $\theta\in S$ and $B\sub (A)_\theta$. Let $\alpha\in S$, $\alpha>0$.
\begin{case}$\alpha=\beta+1$.\end{case}
Let $\gamma=\Tt\pred(\alpha)$. Then $\beta,\gamma,\gamma+1\in S$. (Note $\beta,\gamma+1\leq\alpha$.)

For $x$ such that $(\alpha,x)\in A$, there are $a_x,q_x$ such that $x=[a_x,f_{\tau_x,q_x}]^{(M^*)_{\alpha}}_{E^\Tt_\beta}$, and $(\beta,a_x),(\gamma,q_x)\in A$. (Here $f_{\tau_x,q_x}$ is the function given by the Skolem term $\tau_x$ and parameter $q_x\in (M^*)_{\alpha}$.)

If $E^\Tt_\beta\in\core_0(M_\beta)$, then $(\beta,\lh(E^\Tt_\beta))\in A$. Suppose $M_\beta$ is active type 3, and $F=F^{M_\beta}$ is its active extender. If $E^\Tt_\beta=F$, there are $a_\nu,f_\nu$ such that  $\nu_F=[a_\nu,f_\nu]_F^{M_\beta}$ and $(\beta,a_\nu),(\beta,f_\nu)\in A$. If $\nu_F<\lh(E^\Tt_\beta)<\OR^{M_\beta}$ (so $E^\Tt_\beta\notin\core_0(M_\beta)$), there are $a_\lh,f_\lh$ such that $\lh(E^\Tt_\beta)=[a_\lh,f_\lh]_F^{M_\beta}$ and $(\beta,a_\lh),(\beta,f_\lh)\in A$.
\begin{case} $\alpha$ is a limit ordinal.\end{case}
Then $S\int\alpha\neq\empty$; let $\beta=\sup(S\int\alpha)$. Then there is $\beta'$ such that
\begin{itemize}
\item $0<_\Tt\beta'<_\Tt\beta<_\Tt\alpha$
\item $\beta'=\Tt\pred(\beta)$ (so $\beta$ is a successor)
\item $i_{\beta',\alpha}$ exists and $\deg^\Tt(\beta')=\deg^\Tt(\alpha)$
\item $(A)_{\alpha}\sub\rg(i_{\beta',\alpha})$
\end{itemize}
Moreover, $A\super\{(\beta,x)\ |\ i_{\beta,\alpha}(x)\in (A)_{\alpha}\}$.
It might seem that $\beta'$ suffices as $\sup(S\int\alpha)$, but choosing $\beta$ instead of $\beta'$ is needed to ensure normality of the finite tree.\end{dfn}

This completes the definition of support. We now give an algorithm that passes from $\Tt,B$ as above to a support $A$ for $B$. We'll recursively define sets $S_i,A_i$ approximating the desired $S,A$, with $S_i\sub S_{i+1}$ and $A_i\sub A_{i+1}$. We'll also define ordinals $\alpha_i\in S_i$. $S_0=\{\theta\}$, $A_0=\{\theta\}\cross B$ and $\alpha_0=\theta$. Given $\alpha_n,S_n,A_n$, if $\alpha_n>0$, we process $\alpha_n$, ensuring $S_{n+1},A_{n+1}$ satisfy the requirements of \ref{dfn:supp} for $\alpha=\alpha_n$. If $\alpha_n$ is a successor, let $S_{n+1}=S_n\un\{\beta,\gamma,\gamma+1\}$ (notation as in \ref{dfn:supp}) and enlarge $A_n$ to $A_{n+1}$ by adding the appropriate $(\beta,a_x)$, $(\gamma,q_x)$, etc., and also if $\gamma+1<\alpha_n$, adding $(\gamma+1,0)$. If $\alpha_n$ is a limit, we can find $\beta>\beta'>\sup(S_n\int\alpha)$ minimal with the required properties, and set $S_{n+1}=S_n\un\{\beta\}$. Define $A_{n+1}$ by adding the appropriate $i_{\beta,\alpha}$ preimages to $A_n$. Finally, $\alpha_{n+1}=\sup(S_{n+1}\int\alpha_n)$; notice this exists by construction. The algorithm can be made definite by minimizing in some way to make choices.

Since $\alpha_{i+1}<\alpha_i$, there's $n$ with $\alpha_n=0$. Fixing this $n$, note that $\{\alpha_0,\ldots,\alpha_n\}=S_n$, so all elements of $S_n\cut\{0\}$ got processed at some stage in the construction. Notice that for $i\leq n$, $(A_i)_{\alpha_i}=(A_n)_{\alpha_i}$, so setting $A=\bigun_{i\leq n}A_i$, it's easy to see that $A$ supports $B$ (and has projection $S=\bigun_{i\leq n} S_i$).

\begin{lem}\label{lem:supptree}
Let $\Tt$ be a normal iteration tree on a premouse $M$, $\lh(\Tt)=\theta+1$, and $B\sub M^\Tt_\theta$ finite. There is a normal tree $\Uu$ on $M$ with $\lh(\Uu)=n+1<\om$, with $\deg^\Uu(n)=\deg^\Tt(\theta)$, and a near $\deg^\Tt(\theta)$-embedding $\pi_n:M^\Uu_n\to M^\Tt_\theta$, with $B\sub\rg(\pi_n)$. Moreover, if $\Tt$'s main branch does not drop then neither does $\Uu$'s, and the main embeddings commute: $\pi_n\com i^\Uu=i^\Tt$.
\end{lem}

\begin{proof} Let $A$ support $B$ relative to $\Tt$. We will perform a ``reverse copying construction'', just copying down the parts of the tree appearing in $A$. The natural indexing set for $\Uu$ is $S$ instead of an ordinal. Let the tree order $U=T\rest S$ and drop/degree structure $D^\Uu=D^\Tt\rest S$. Denote the models $\N_\alpha$. $\Uu$ will actually be padded. Padding occurs just at every limit ordinal: $\N_{\alpha}=\N_{\sup(S\int\alpha)}$ when $\alpha\in S$ is a limit. (Note: we allow $\Uu\pred(\gamma+1)=\alpha$ but not $\Uu\pred(\gamma+1)=\sup(S\int\alpha)$.)

We'll define copy embeddings $\pi_\alpha:\N_\alpha\into M_\alpha$. In case $M_\alpha$ is active, let $\psi_\alpha:\Ult_0(\N_\alpha,F^{\N_\alpha})\to\Ult_0(M_\alpha,F^{M_\alpha})$ be the canonical map induced by $\pi_\alpha$. Otherwise let $\psi_\alpha=\pi_\alpha$. (We'll have enough elementarity of $\pi_\alpha$ that this makes sense.) We have $\psi_\alpha\rest\OR^{\N_\alpha}=\pi_\alpha$. We'll maintan inductively on $\alpha$ that ($\varphi_\alpha$): $\all\gamma,\delta,\xi+1\in (S\int\alpha+1)$,
\begin{itemize}
\item Elementarity: $\pi_\gamma$ is a near $\deg^\Tt(\gamma)$-embedding,
\item Range: $\rg(\pi_\gamma)\super (A)_{\gamma}$,
\item $\Uu$'s extenders: $\psi_\xi(E^\Uu_\xi)=E^\Tt_\xi$ or else $E^\Uu_\xi=F^{\N_{\xi}}$ and $E^\Tt_\xi=F^{M_\xi}$,
\item Exact-$\nu$-$\lh$-preservation: $\psi_\xi(\nu_{E^\Uu_\xi})=\nu_{E^\Tt_\xi}$ and $\psi_\xi(\lh(E^\Uu_\xi))=\lh(E^\Tt_\xi)$,
\item Half-$\nu$-preservation: if $\xi<\gamma$ then $\pi_\gamma(\nu^\Uu_\xi)\geq\nu^\Tt_\xi$,
\item Agreement: if $\xi<\gamma$ then $\psi_\xi$ agrees with $\pi_\gamma$ below $\nu^\Uu_\xi$; if $E^\Uu_\xi$ is type 1 or 2 or $\gamma=\xi+1$ then they in fact agree below $\lh(E^\Uu_\xi)+1$,
\item Commutativity: if $\delta<\gamma$ and $i^\Uu_{\delta,\gamma}$ is defined then $\pi_\gamma\com i^\Uu_{\delta,\gamma}=i^\Tt_{\delta,\gamma}\com\pi_\delta$.
\end{itemize}

Set $\N_0=M_0=M$ and $\pi_0=\id$; clearly $\varphi_0$.

Suppose we have $\Uu\rest S\int\alpha+1$, $\varphi_\alpha$ holds and $\alpha+1\in S$. First we define $E^\Uu_\alpha$ and then show that it is legal. If $E^\Tt_\alpha\in\rg(\pi_\alpha)$, set $E^\Uu_\alpha=\pi_\alpha^{-1}(E^\Tt_\alpha)$. If $E^\Tt_\alpha=F^{M_\alpha}$, set $E^\Uu_\alpha=F^{\N_\alpha}$.

Otherwise since $(A)_\alpha\sub\rg(\pi_\alpha)$, \ref{dfn:supp} implies $M_\alpha$ is type 3, so let $\nu=\nu_{F^{M_\alpha}}$ and $\nubar=\nu_{F^{\N_\alpha}}$. We must have $\nu<\lh(E^\Tt_\alpha)<\OR^{M_\alpha}$, so there are $a_\lh,f_\lh\in\rg(\pi_\alpha)$ with $\lh(E^\Tt_\alpha)=[a_\lh,f_\lh]^{M_\alpha}_{F^{M_\alpha}}$. This gives $E^\Tt_\alpha\in\rg(\psi_\alpha)$; set $E^\Uu_\alpha=\psi_\alpha^{-1}(E^\Tt_\alpha)$. Since ``$[a,f]$ represents an ordinal not in my $\OR$'' is $\Pi_1$ and $\pi_\alpha$ is at least that elementary, $\psi_\alpha(\nubar)=\nu$ and
\begin{equation}\label{eqn:nupres} \psi_\alpha(\OR^{\N_\alpha})=\psi_\alpha((\nubar^+)^{\Ult(\N_\alpha,F^{\N_\alpha})})=(\nu^+)^{\Ult(M_\alpha,F^{M_\alpha})}=\OR^{M_\alpha}.
\end{equation}
So $E^\Uu_\alpha$ is on the $\N_\alpha$ sequence.

We now show that $E^\Uu_\alpha$ is indexed above $\Uu$'s earlier extenders.

If $\alpha=\xi+1$, by exact-$\nu$-$\lh$-preservation and agreement, $\pi_\alpha(\lh(E^\Uu_\xi))=\lh(E^\Tt_\xi)<\lh(E^\Tt_\alpha)$, so $\lh(E^\Uu_\alpha)$ is certainly high enough.

Suppose $\alpha$ is a limit. Let $\beta'=\Tt\pred(\beta)$ where $
\beta=\sup(S\int\alpha)$ as in \ref{dfn:supp}. Let $i^\Tt_{\beta',\alpha}(\lh')=\lh(E^\Tt_\alpha)$ (where $\lh'=\OR^{M_{\beta'}}$ is possible). Then as $\Tt$ is normal, $\crit(i^\Tt_{\beta',\alpha})<\lh'$. So $\lh(E^\Tt_{\beta-1})<i^\Tt_{\beta',\beta}(\lh')$. As in the successor case, we have $\pi_\beta(\lh(E^\Uu_{\beta-1}))=\lh(E^\Tt_{\beta-1})$, and since $\pi_\alpha=i^\Tt_{\beta,\alpha}\com\pi_\beta$, the claim follows.

Unless $\N_\alpha$ is active type 3 and $E^\Uu_\alpha=F^{\N_\alpha}$, exact-$\nu$-$\lh$-preservation for $\xi=\alpha$ is routine. But if so, there's a representation $[a_\nu,f_\nu]$ of $\nu=\nu^{M_\alpha}$ in $\rg(\pi_\alpha)$, and the argument leading to \ref{eqn:nupres} works here too.

Letting $\gamma=\Tt\pred(\alpha+1)$, we have $\gamma,\gamma+1\in S$. Let $\kappa=\crit(E^\Uu_\gamma)$. By half-$\nu$-preservation, $\kappa<\nu^\Uu_\gamma$. Exact-$\nu$-$\lh$-preservation and the agreement between earlier embeddings and $\pi_\alpha$ imply that for $\gamma'<\gamma$ in $S$, $\nu^\Uu_{\gamma'}\leq\kappa$, so setting $\gamma=\Uu\pred(\alpha+1)$ is normal. Moreover, $\psi_\gamma$ agrees with $\pi_\alpha$ below $(\kappa^+)^{\N_\alpha}$. So the hypotheses for the shift lemma (\cite{outline}, 4.2) apply (to the appropriate initial segments of $\N_\gamma, M_\gamma$), which gives $\pi_{\alpha+1}$. As $\pi_\gamma$ is a near $\deg^\Uu(\gamma)$-embedding, the shift lemma and (\cite{near}, 1.3) gives $\deg^\Uu(\alpha+1)=\deg^\Tt(\alpha+1)$ and that $\pi_{\alpha+1}$ is a near $\deg^\Uu(\alpha+1)$-embedding. (Let $k=\deg^\Tt(\alpha+1)$. Then by $\Sigma_{k+1}$-elementarity, $\pi_\gamma(\rho^{M^\Tt_\gamma}_k)\geq\rho^{M^\Tt_\gamma}_k$. Thus the $\Uu$ side doesn't drop in degree whilst the $\Tt$ side maintains.) It also gives the required commutativity, and that $\pi_{\alpha+1}$ agrees with $\psi_\alpha$ below $\lh(E^\Uu_{\alpha})+1$, which yields half-$\nu$-preservation and agreement.

Now let $x\in A_{\alpha+1}$, and let $a_x, q_x, \tau_x$ be as in the definition of support. Inductively, we have $\abar,\qbar$ so that $\pi_\gamma(\qbar)=q_x$ and $\pi_\alpha(\abar)=a_x$. As $\deg^\Uu(\alpha+1)=\deg^\Tt(\alpha+1)$, $\xbar=[\abar,f_{\tau_x,\qbar}]$ is an element of $\N_{\alpha+1}=\Ult_{\deg^\Uu(\alpha+1)}((\N)^*_{\alpha+1},E^\Uu_\alpha)$. Moreover, by definition, $\pi_{\alpha+1}(\xbar)=x$. Therefore $\rg(\pi_{\alpha+1})$ is good.
(Note: To maintain the range condition, it is essential here that the degrees of $\Uu$ are as large as those of $\Tt$. If we only had weak embeddings, the remark in the previous paragraph wouldn't apply, and it seems this might fail, though we don't have an example of such.) This gives $\Uu\rest\alpha+2$ and establishes $\varphi_{\alpha+1}$.

Now suppose we have $\Uu\rest\beta+1$ for some $\beta\in S$, $\beta<\theta$, and $\varphi_{\beta+1}$ holds, but $\beta+1\notin S$. So $\inf(S\cut(\beta+1))$ is a limit $\alpha$. Thus we set $E^\Uu_\beta=\empty$ and $\N_\alpha=\N_\beta$. Let $\pi_\alpha=i^\Tt_{\beta,\alpha}\com\pi_\beta$. This makes sense and yields a near $\deg^\Tt(\alpha)=\deg^\Uu(\alpha)$-embedding since there is no dropping of any kind in $[\beta,\alpha]_T$. Its range is large enough since $\rg(\pi_\beta)\super A_\beta$ and $i^\Tt_{\beta,\alpha}``A_\beta\super A_\alpha$. By \ref{dfn:supp}, $\beta$ is a successor; let $E=E^\Uu_{\beta-1}$. By exact-$\nu$-$\lh$-preservation, \[ \pi_\beta(\nu_E)=\nu^\Tt_{\beta-1}\leq\crit(i^\Tt_{\beta,\alpha}). \]
Therefore $i^\Tt_{\beta,\alpha}\com\pi_\beta$ agrees with $\pi_\beta$ below $\nu_E$. Moreover, if $E$ is type 1/2, then in fact 
\[ \pi_\beta(\lh(E))=\lh(E^\Tt_{\beta-1})<\crit(i^\Tt_{\beta,\alpha}), \]
so we get the stronger agreement. This easily extends to extenders used prior to $E$. Since $\pi_\alpha(\nu_E)\geq\pi_\beta(\nu_E)$, half-$\nu$-preservation is maintained. (But this is where we may lose exact-$\nu$-$\lh$-preservation.) This gives $\Uu\rest\alpha+1$ and shows $\varphi_{\alpha}$, completing the construction.\renewcommand{\qedsymbol}{$\Box$(Lemma \ref{lem:supptree})}\end{proof}

\begin{dfn}\label{dfn:supptree}\index{support tree}\index{$\Tt_x$} Let $\Tt$, $B$ be as in \ref{lem:supptree}, and suppose $A$ captures $B$. The finite support tree $\Tt^A_B$ for $B$, relative to $A$, is the tree $\Uu$ as defined in the proof of \ref{lem:supptree}. Letting $A^*$ be the support given by the algorithm described earlier, the finite support tree $\Tt_B$ is $\Tt^{A^*}_B$.\end{dfn}

\subsection*{Total Measures}

\begin{dfn}\index{$P_E$}
Given a pre-extender $E$ over a premouse $N$, let $P_{N,E}$ (or just $P_E$ where $N$ is understood), denote the unique structure that might be a premouse with $\tc(E)$ active: $(\Ult(N,E)|(\nu_E^+)^{\Ult(N,E)},\tc(E))$. 
\end{dfn}
\begin{rem}
Given $E$, note that $P_E$ is a premouse iff it is wellfounded and it satisfies the initial segment condition.\end{rem}

\begin{dfn}\index{measure}\index{finitely generated extender}\index{total measure}\index{measure type}
Let $E$ be an (possibly long) extender over $\kappa$. Then $E$ is a \emph{measure} or $E$ is \emph{finitely generated} if there is  $s\in\nu_E^{<\om}$ such that for all $\alpha<\nu_E$, there is $f:\kappa\to\kappa$, measured by $E$, such that $\alpha=[a,f]_E$.

$E$ is a \emph{total} measure if it is (equivalent to) an $\om$-complete non-principal ultrafilter.

Let $j:P\to R$ be an elementary embedding of premice with $\crit(j)$ inaccessible in $P$ and $E_j\rest j(\crit(j))\notin R$. Then the Dodd fragments of $j$ are of \emph{measure type} if $\sigma_j=(\crit(j)^+)^P$.
\end{dfn}

\begin{lem}\label{lem:Dfpres} Let $j:P\to R$ be a fully elementary embedding of premice with $\kappa=\crit(j)$ inaccessible in $P$, and $P\sats\ZFmin$. Suppose $E_j\rest j(\kappa)\notin R$. Let $\Uu$ be a normal iteration tree on $R$, such that $i^\Uu$ exists and $\crit(i^\Uu)>\kappa$. Then $i^\Uu(s_j)$, $\sup i^\Uu``\sigma_j$ are the Dodd-fragment ordinals of $i^\Uu\com j$.\end{lem}
\begin{proof}
Consider a single ultrapower of $R$ by an extender $F$ which is close to $R$ (but maybe not in $R$), with $\mu=\crit(F)>\crit(j)$. Since $R\sats\ZFmin$, $\Ult_\om(R,F)=\Ult_0(R,F)$. So if the ultrapower has too large a fragment of $E_{i_F\com j}$, there's a function $f:\mu\to R$ in $R$ and $a\in\nu_F^{<\om}$ such that $[a,f]$ represents it. Also since $R\sats\ZFmin$, the closeness of $F$ to $R$ implies $F_a\in R$. For $\alpha<\sigma_j$, there's $F_a$-measure one many $u$ such that
\[ f(u)\rest\alpha\un s_j=E_j\rest\alpha\un s_j. \]
It follows that $E_j\rest\sigma_j\un s_j\in R$; contradiction. This extends to normal iterates of $R$.\renewcommand{\qedsymbol}{$\Box$(Lemma \ref{lem:Dfpres})}\end{proof}

\begin{thm}\label{thm:exactmeas}\setcounter{clm}{0}
Let $N$ be an $\fully$-iterable mouse, $E,\kappa\in N$, and
\[ N\sats\theory,\ E\ \textrm{is a total wellfounded measure on}\ \kappa\ \textrm{and}\ \kappa^{++}\ \textrm{exists}.\]
Then there is a iteration tree $\Tt$ on $N$ with the following properties:
\begin{itemize}
\item[(a)] $\Tt$ is a finite, normal tree based on $N|(\kappa^{++})^N$.
\item[(b)] $\Ult(N,E)$ is the last model of $\Tt$ (so $\Tt$ results from comparing $N$ with $\Ult(N,E)$), and there is no dropping on $\Tt$'s main branch.
\item[(c)] $E$ is the extender of $i^\Tt$.
\item[(d)] There is a finite linear iteration $\Ll$ of $N$, possibly non-normal, with final model $\Ult(N,E)$ and $i^\Ll=i^\Tt$.
\item[(e)] $\core_D(E)$ is a measure and is the active extender of $P\pins N$.
\item[(f)] For any $F$ used in $\Tt$, $F$ and $\core_D(F)$ are measures.
\item[(g)] $P_E$ satisfies the initial segment condition iff there is exactly one extender used along $\Tt$'s main branch. (Clearly that extender must be $E$.)
\item[(h)] The extenders used in $\Tt$ are characterized in Corollary \ref{cor:damexts}.
\end{itemize}\end{thm}

\begin{proof}
The proof is by induction on mice, so we assume that proper levels of $N|(\kappa^{++})^N$ satisfy the theorem. Conclusions (d) and (h) won't be needed in the induction, however. Let $M=\Hull_\om(N|(\kappa^{++})^N)$. The failure of the theorem is first order over $N|(\kappa^{++})^N$, so as usual we assume $E$ is the least. Let $\Ebar$ be the collapse of $E$. A simplification of the argument at the start of the proof of \ref{thm:coarseDodd} shows $\Ult(M,\Ebar)$ embeds into a level of $N$, so is iterable.\\
\\
\noindent\emph{Notation.} Again we have no further need for the $N$-level, so we'll drop the bar notation.\\

As in \ref{thm:coarseDodd}, comparing $M$ with $\Ult(M,E)$ results in trees $\Tt$ on $M$ and $\Uu$ on $\Ult(M,E)$ with a common last model $Q$. Let $s$ generate $E$. Let $\Ttbar=\Tt_{i_{U,Q}(s)}$ be the finite support tree for $i_{U,Q}(s)$ relative to $\Tt$, as in \ref{dfn:supptree}. Let $\Qbar$ be $\Ttbar$'s final model, and $\tau:\Qbar\to Q$ be the liftup map given by \ref{lem:supptree}. Since $i^\Uu(s)\in\rg(\pi)$, we have the natural factor map $\psi:U\to\Qbar$. The following diagram commutes:
\[
\begin{picture}(100,100)(50,0)
\put(46,0){$M$}
\put(59,10){\vector(1,4){15}} 
\put(61,11){\vector(1,1){15}} 
\put(71,15){$\scriptstyle i^{\Ttbar}$}
\put(62,8){\vector(3,1){70}} 
\put(75,30){$\Qbar$}
\put(75,75){$U$}
\put(135,30){$Q$}
\put(80,70){\vector(0,-1){28}} 
\put(90,35){\vector(1,0){37}} 
\put(85,70){\vector(3,-2){47}} 
\put(52,40){$\scriptstyle i_E$}
\put(95,10){$\scriptstyle i^\Tt$}
\put(111,56){$\scriptstyle i^\Uu$}
\put(82,53){$\scriptstyle\psi$}
\put(100,40){$\scriptstyle\pi$}
\end{picture}
\]

\begin{clm}\label{clm:crits} $\crit(i^\Tt)=\crit(i^{\Ttbar})=\crit(E)<\crit(i^\Uu)$. Therefore $\crit(E)<\crit(\psi)$ and $\crit(E)<\crit(\pi)$.\end{clm}
\begin{proof} Since the diagram commutes, otherwise $\kappa=\crit(i^\Tt)=\crit(i^\Uu)\leq\crit(E)$, and $i^\Tt$ must be compatible with $i^\Uu$ on $\pow(\kappa)$ through $\inf\{i^\Tt(\kappa),i^\Uu(\kappa)\}$, since $M$ is its own hull. But $(\Tt,\Uu)$ was a comparison, so this is false. The second statement now follows commutativity.\renewcommand{\qedsymbol}{$\Box$(Claim \ref{clm:crits})}\end{proof}

\begin{clm}\label{clm:finite} We may assume that for every extender $F$ used in $\Ttbar$, $F$ and $\core_D(F)$ are measures. Moreover, letting $E^*$ be the (long) extender of $\Ttbar$'s main branch embedding, we may assume that $\Ttbar$ witnesses all conclusions of the theorem other than (h), with respect to $E^*$.\end{clm}

We'll first assume the claim, and finish the proof. Let $G$ be the first extender used on $\Ttbar$'s main branch. By Claim \ref{clm:finite}, $\core_D(G)$ is a measure, so by \ref{lem:Dsp5}, the Dodd fragments of $G$, and so those of $i^{\Ttbar}$ (in $\Qbar$), are of measure type. Consider the Dodd fragments of $E$ and $i^\Tt$.  By commutativity and \ref{clm:crits}, $\psi$ maps fragments of $E$ to fragments of $i^{\Ttbar}$, and likewise $\pi$ from $i^{\Ttbar}$ to $i^\Tt$ and $i^\Uu$ from $E$ to $i^\Tt$. By \ref{lem:Dfpres}, since $i^\Uu$ is an iteration, it also preserves the \emph{maximality} of Dodd fragments. By commutativity, $\psi$ and $\pi$ preserve maximality also. In particular, the fragments of $i_E$ and $i^\Tt$ are also of measure type.

Let $\left<F_\alpha\right>$ enumerate the Dodd cores of the extenders $F\leq_\dam G$ for the extenders $G$ used on $\Tt$'s main branch, in order of increasing critical points. Let $\left<Q_\alpha,j_{\alpha,Q}\right>$ be the corresponding core sequence (if $\Tt$'s main branch uses more than one extender, simply string the corresponding core sequences together). Applying \ref{lem:Gcoreseq} to each of the extenders used on $\Tt$'s main branch, we know the core sequence reaches all the way to $Q$, and the analogue to the characterization following \ref{lem:Gcoreseq} also holds. Let $\left<\Fbar_\alpha\right>$ and $\left<\Qbar_\alpha\right>$ enumerate the corresponding objects for $\Ttbar$. Claim \ref{clm:finite} gives $\Fbar_\alpha$ is a measure. From the previous paragraph, we have $\Fbar_0=F_0$, and $i_E,i^{\Ttbar}$ and $i^\Tt$ factor through $\Qbar_1=Q_1=\Ult(M,F_0)$, commutatively. (It doesn't matter whether $\kappa$ is an $s_{i^\Tt}$-generator, since $\kappa<\crit(i^\Uu)$ anyway.)

Suppose $Q_1\neq Q$. As in \ref{rem:compat}, $j_1:Q_1\to Q$ is compatible with an extender used on the $M$ side of the comparison, through its sup of generators. As in Claim \ref{clm:crits}, it follows that $\crit(j_1)<\crit(i^\Uu)$, and the rest of the above argument can be repeated. Maintaining this situation inductively, we eventually (in finitely many stages) reach $U=\Qbar$, so $U$ is a finite iterate of $M$, so $U=Q=\Qbar$, and $\Tt=\Ttbar$ is finite and as desired. So to finish proving the theorem we just need:

\begin{proof}[Claim \ref{clm:finite} Proof]\ \\
\\
\emph{Notation.} For this proof we need only refer to the finite tree, which we'll refer to as $\Tt$ instead of $\Ttbar$ from now on.\\

The motivation for the proof is to refine the finite support tree construction. Suppose $\Tt$ uses an extender $E$ which isn't a measure. Since we only need to support finitely much, the intuition is that only finitely many of $E$'s generators are important, so we should be able to improve $\Tt$ by replacing $E$ with a sub-measure. Doing this to every such extender should yield the type of tree we want. However, executing this takes some care. One problem is that it seems interfering with some part of the tree in this way might affect the normality later on. To get around this, we start from the end of $\Tt$ and work backwards, producing a series of trees $\left<\Tt_i\right>$, converting extenders to measures one by one. $\Tt_{i+1}$ will replicate $\Tt$ until (a version of) the relevant $E$ first appears, then convert the use of $E$ to the production and use of a sub-measure, ensuring that the sub-measure includes enough generators to allow a ``downward'' copying of the remainder of $\Tt_i$. Because we've already processed that remainder, it only involves simple interactions with the earlier part of the tree.

So, $\Tt$ is finite. Let $F_1,\ldots,F_n$ enumerate the extenders $F\leq_\dam G$ for any $G$ used along the main branch of $\Tt$, this time in order of \emph{de}creasing critical point. Note that $F_1$ is Dodd sound. If it's a measure, we set $\Tt_1=\Tt$. Suppose otherwise. Say $F_1$ is on the $P$ sequence and $\kappa=\crit(F_1)$. After using $F_1=G_1$, $\Tt$ just pops out of the damage structure, hitting a sequence of active extenders $G_2,\ldots,G_k$, where $G_i$ is largest in $\dam(G_{i+1})$, until producing its final model $Q=\Ult_\om(M^\Tt_p,G_k)$ ($k=1$ is possible). Let $a_k\in\nu_{G_k}^{<\om}$ be sufficient to support $x$ in this ultrapower. Let $a_{i-1}$ be sufficient to support $a_i$, and $a\in\tau_{F_1}^{<\om}$ be such that $a\un t_{F_1}$ is sufficient to support $a_2$ and $F_1\rest(\max (t_{F_1}))$. (If $F_1$ is type 3 then instead let $\max(a)$ be a generator which indexes a segment of $F_1$; these are unbounded.) Since $(\kappa^+)^{F_1}<\tau_{F_1}$ (and $\tau_{F_1}$ is a cardinal of $P|lh(F_1)$), Dodd soundness implies $\mu=F_1\rest(a\un t_{F_1})\in P|\tau_{F_1}$. Considering the natural factor map $\Ult_0(P|\lh(F_1),\mu)\to\Ult_0(P|\lh(F_1),F_1)$, it's easy to see that our choice of $a$ implies $\mu$ has the initial segment condition. By the minimality of $M$, there's a finite tree $\Uu=\Uu_1$ on a proper segment of $P|(\kappa^{++})^{F_1}$, which uses only measures, such that $\mu$ is on the sequence of its final model. (Use condensation to get a proper level of $P|(\kappa^{++})^{F_1}$ containing $\mu_1$ to apply \ref{thm:exactmeas} to.)

Now let $q$ be least such that $\lh(E^\Tt_q)>(\kappa^+)^{F_1}$. We would like to have $\Tt_1$ begin by followwing $\Tt$ until reaching $M^\Tt_q$, then followwing $\Uu$. This will be fine unless $\nu^\Tt_{q-1}>\kappa$ and $\Uu$ uses an extender with crit $\kappa$; in this case $\Tt\rest q+1\conc\Uu$ isn't normal. To deal with this we need to observe some properties of $\Uu$. First, $\core_D(\mu_1)$ is on the $P$-sequence by \ref{thm:exactmeas}. We first claim that $\Uu$ is equivalent to a tree on $P|\lh(\core_D(\mu_1))$.

If $E^\Uu_0$ exists then $(\kappa^+)^{\mu}<\lh(E^{\Uu}_0)$. If $\mu$ isn't the active extender of $\Uu$'s last model then, since $\mu$ projects to $(\kappa^+)^{\mu}$, in fact $\mu$ is on the $P$-sequence. Assume it is active. As $\crit(\mu)=\kappa$,  $\crit(\core_\om(\mu)\to\mu>\kappa$. As $\core_D(\mu)$ is a measure, so is $\core_\om(\mu)$, so they're equal (by the Dodd-fragment preservation facts), and on the $P$-sequence. Also $\lh(E^{\Uu}_0)\leq\lh(\core_D(\mu))$. So with $i^\Uu:\core_D(\mu)\to\mu$ the (dropping) branch map, $\crit(i^\Uu)>\kappa$. Because $\Uu$ uses only measures, $M^\Uu_0$ is passive, and $\kappa$ is a cardinal of $P$, we get all of $\Uu$'s crits are $\geq\kappa$.

Now suppose $\crit(E^\Uu_i)=\kappa$. Then $(M^*)^\Uu_{i+1}=P$. Since $\crit(i^\Uu)>\kappa$, $M^\Uu_{i+1}$ isn't on the main branch, so $\Uu$ goes back at some point. Because it only uses measures, in fact it goes back to $M^\Uu_i$ at say stage $j$ with $\crit(E^\Uu_j)=|\nu^\Uu_i|^{E^\Uu_i}$, and $\lh(E^\Uu_{i+1})<(|\nu^\Uu_i|^{++})^{M^\Uu_{i+1}}$. So $\Uu\rest[i+1,j]$ is on $M^\Uu_{i+1}|\lh(E^\Uu_{i+1})$ and $E^\Uu_j$ triggers a drop to $M^\Uu_i|\lh(E^\Uu_i)$. It follows that $\Uu$ can be considered a tree on $P|\lh(\core_D(\mu))$, or else on $M^\Tt_q$. (One can also show that if $\crit(E^\Uu_i)=\kappa$, then $E^\Uu_i$ is the image of $\core_D(\mu)$.)

Now consider $\Tt\rest(q+1)\conc\Uu$, and suppose $\kappa<\nu^\Tt_{q-1}$ and $\crit(E^\Uu_i)=\kappa$. $\Uu$ applies this extender to $M^\Uu_0=M^\Tt_q$, but for normality, the correct model to return to is $M^\Tt_{q-1}$, with $(M^*)_{i+1}=M^\Tt_{q-1}|(\kappa^+)^{F_1}$, and since $\rho_1^{(M^*)_{i+1}}=\kappa$ the correct degree is $0$. The preceding paragraph shows that this is fine; i.e. there is a normal tree $\Tt\rest q+1\conc\Uu'$ such that $\Uu'$ is given by copying the extenders of $\Uu$, and returning to $M^\Tt_{q-1}$ when $\kappa$ is the crit. This tree also has $\mu$ on the sequence of its final model.

So let $\Tt_1=\Tt\rest q+1\conc\Uu'_1\conc\Vv_1$, where $\Vv_1$ is a ``downward copy'' of the remainder of $\Tt$, hitting $\mu$, then active extenders $G'_2,\ldots,G'_k$, yielding a final model $Q'$. (Note $\mu$ and $F_1$ apply (normally) to the same model, which yields premice with active extenders $G'_2$ and $G_2$ respectively. Moreover, $G'_2$ is a sub-extender of $G_2$, they both apply to the same model, etc.) Since $G'_k$ is a sub-extender of $G_k$, we get a $\pi_1:Q_1\to Q$ which is fully elementary,  commuting with the $\Tt_1$ and $\Tt$ embeddings, and $x\in\rg(\pi_1)$ by our choice of generators $a$.\\

The general case is a little more complicated; we just sketch it. Suppose we have $\Tt_i$ where $1\leq i<n$; first we describe some of our inductive assumptions. Let $q$ be least such that $\lh(E^\Tt_q)>(\kappa_i^+)^{F_i}$. Then $\Tt_i$ is of the form $\Tt\rest q+1\conc\Uu_i$.  If $E_i<_\dam E_{i+1}$ then we'll have that $\Uu_i$ uses a sub-extender $E^i_{i+1}$ of $E_{i+1}$, whose damage structure consists only of measures, and that $\core_D(E^i_{i+1})=F_{i+1}$. Otherwise $\Tt\rest q+1$ uses $E^i_{i+1}=E_{i+1}=F_{i+1}$. Moreover, $\Uu_i$'s extenders with crit $\geq\tau_{F_{i+1}}$ are all measures, and $\Uu_i$'s crits below $\tau_{F_{i+1}}$ are all of the form $\kappa_m$ for some $m$.

Now we construct $\Tt_{i+1}$. Note that $F_{i+1}$ is on the sequence of $M^\Tt_{q'}$ for some $q'\leq q$. (Either $E_i<_\dam E_{i+1}$ so the critical point $\kappa_i$ applies to $M^\Tt_{q'}$, or $E_{i+1}=F_{i+1}$ is Dodd-sound and $\nu_{E_{i+1}}\leq\kappa_i$.) If $F_{i+1}$ is itself a measure, let $\Tt_{i+1}=\Tt_i$. Otherwise, as in the $i=0$ case, we choose $a\in\tau_{F_{i+1}}$ such that $t_{F_{i+1}}\un a$ is sufficient to (eventually) generate the remainder of $\Tt_i$, the $\Tt_i$ preimage of $x\in Q$, the Dodd-solidity witnesses for $F_{i+1}$ and $\tau_{F_{i+1}}$. (If $F_{i+1}$ is type 3, choose $\max(a)$ as in the $i=0$ case and work relative to $F_{i+1}\rest\max(a)+1$.) Let $\mu_{i+1}=F_{i+1}\rest a\un t_{F_{i+1}}$. As before, $\mu_{i+1}$ is on the sequence of the last model of a tree $\Uu^0_{i+1}$ on $P|(\kappa_{i+1}^{++})^{F_{i+1}}$. Let $q_{i+1}$ be least such that $\lh(E^\Tt_{q_{i+1}+1})>(\kappa_{i+1}^+)^{F_{i+1}}$. Then $\Tt_{i+1}=\Tt\rest (q_{i+1}+1)\conc(\Uu^0_{i+1}\conc\Uu^1_{i+1})'\conc\Uu^2_{i+1}$; we'll now define $\Uu^1_{i+1}$ and $\Uu^2_{i+1}$. (Here $\Vv'$ is the modification of $\Vv$ we used in the $i=0$ case.)

If $E_i\not<_\dam E_{i+1}$ then $E_{i+1}$ is the ``top'' extender in the damage structure of the extender used on a branch immediately preceding $E_i$. I.e. let $q'$ be such that $E_i$ applies to $M^\Tt_{q'}$. Suppose $E_i\in\dom(\dam(E_m))$. Then $E_{i+1}<_\dam E_m$. Immediately after using $E_{i+1}$, $\Tt$ successively applies $E_k$ for all $k$ with $E_{i+1}<_\dam E_k<_\dam E_m$ (in their nested order), which results in $M^\Tt_{q'}$. Clearly either $q'=q_i$ or $\lh(E^\Tt_{q'})=(\kappa_i^+)^{F_i}$. In the latter case, $E_i$ triggers a drop to $M^\Tt_{q'}|\lh(E^\Tt_{q'})$, so $E^\Tt_{q'}$ is on the branch leading from $\core_D(E_m)$ to $E_m$, so $\crit(E^\Tt_{q'})=\kappa_m$. In fact $E_{i+1}<_\dam E^\Tt_{q'}$, so $E^\Tt_{q'}$ is the active extender of $M^\Tt_{q'}$. Alternatively $E_i$ may be applied directly along $\Tt$'s main branch. In this case $E_i$ can't trigger a drop, so $q'=q_i$. So set $\Uu^1_{i+1}=\empty$, and as in the $i=0$ case, set $\Uu^2_{i+1}$ to copy down the above activity by hitting $\mu_{i+1}$, then the resulting active extenders, until reaching a preimage of $M^\Tt_{q_i}$, and then to copy down $\Uu_i$. Notice that we've maintained the inductive restrictions on $\Tt_{i+1}$'s extenders and crits.

Otherwise $E_{i}$ is the least extender in $\dam(E_{i+1})$. Using the natural map reducing $\mu=\mu_{i+1}$ to $F_{i+1}$, let $\Uu^1_{i+1}$ be the downward copy of the segment of $\Tt_i$ damaging $F_{i+1}$ (which is essentially on $M^{\Tt_i}_{p}|\lh(F_{i+1})$ where $p=q_i$ or $p=q_i-1$), to the last model of $\Uu^0_{i+1}$ (so $\Uu^1_{i+1}$ is on that last model). We will show $\Uu^0_{i+1}\conc\Uu^1_{i+1}$ is normal. Let $\pi:\mu\to F_{i+1}$ be the reduction map and $\pi(\taubar)=\tau_{F_{i+1}}$. We claim that for all extenders $F$ used in $\Uu^0_{i+1}$, $\nu_F\leq\taubar$. Let $t=t_{F_{i+1}}=s_{F_{i+1}}$ and $\pi(\tbar)=t$. Note $\tbar\sub s_{\mu}$ by our choice of $a$. Let $j=i^{\Uu^0_{i+1}}$. By \ref{thm:exactmeas}, as $\mu$ has the initial segment condition, $\mu_\core=\core_D(\mu)$ is on the $P$-sequence, is the domain of $j$, and by the Dodd-fragment preservation facts, $j(s_{\mu_\core})=s_\mu\super\tbar$. Since $\tbar\un\taubar$ generates $\mu$, it follows that all generators along $\Uu^0_{i+1}$'s main branch are below $\taubar$. As $\Uu^1_{i+1}$ only uses measures, its crits are at least $\tau$ or exactly $\kappa$. (note $\tau\leq\kappa_i<\nu_F$ where $F$ is the first extender used). The normality of $\Uu^0_{i+1}\conc\Uu^1_{i+1}$ follows. Then $\Uu^2_{i+1}$ copies down the remainder of $\Uu_i$. Again we've maintained the inductive restrictions on extenders.

Finally, having gotten this far, the reader should be happy to check that our measure-only tree $\Tt_{\textrm{meas}}$ also satisfies the other conclusions of the theorem with respect to its (long) main branch extender (which is a measure itself). The linear tree $\Ll$ of (d) is given by applying all extenders in the damage structures of the measures used along the main branch of $\Ttmeas$, but in order of increasing critical point, not length.

\renewcommand{\qedsymbol}{$\Box$(Claim \ref{clm:finite})(Theorem \ref{thm:exactmeas})}\end{proof}
\renewcommand{\qedsymbol}{}\end{proof}

\begin{rem} During the proof, we dealt with the possibility that $\Tt$ uses extenders not in the damage structure of $E$. This does occur; for example, suppose $N\sats\ZFC$ is a mouse, and $F$ is a finitely generated total extender on $\es^N$. Let $\kappa$ be the largest cardinal of $N|\lh(F)$. Suppose $\kappa$ is measurable in $\Ult(N,F)$, and $D$ is a witnessing normal measure. Then $E=D\com F$ is a measure of $N$, and the resulting tree $\Tt$ uses 3 extenders, $E^\Tt_0=F$, $E^\Tt_1=D$, and $E^\Tt_2=E$. But the damage structure of $E$ just involves $E$ and $D$. This is (a simple case of) the only exception - we leave the proof of the following corollary to the reader.\end{rem}

\begin{cor}\label{cor:damexts}
With $\Tt$ as in \ref{thm:exactmeas}, let $F_1,\ldots,F_r$ be the extenders used along the main branch of $\Tt$. Suppose $E^\Tt_m$ exists and let $P=M^\Tt_m|\lh(E^\Tt_m)$. Then $\core_D(E^\Tt_m)$ is the active extender of $\core_\om(P)$. Also there are unique $n,k$ such that $\core_D(E^\Tt_m)=\core_D(E^\Tt_n)$, $m\leq_\Tt n$, and $E^\Tt_n\leq_\dam F_k$.
\qed\end{cor}

\subsection*{Submeasures}

We now move on to consider submeasures of normal measures in mice, and prove some condensation-like facts in this context. Suppose $N$ is a type 1 mouse and $\kappa=\crit(F^N)$. Given $\Aa\sub\pow(\kappa)^N$ of size $\kappa$ in $N$, we show that the submeasure $F^N\rest\Aa$ is often on $\es^N$. The basic structure of the proofs are like the main proofs in the earlier sections, in that we'll compare $N$ with a phalanx derived from it and the submeasure. It was Steel's idea to use this approach here. Establishing the iterability of the phalanx is simple for \ref{thm:fragpassive} and \ref{thm:fragactivetp2}, as it is inherited directly from the mouse's active normal measure. In the case of \ref{thm:fragactivetp3} there are fine structural complications.

\begin{thm}\label{thm:fragpassive}\setcounter{clm}{0}
Let $M$ be an $(0,\omega_1+1)$-iterable type 1 mouse, with active measure $\mu$, with crit $\kappa$. Let $\kappa<\beta<(\kappa^+)^M$ with $M|\beta$ passive, and $\mubar=\mu\int M|\beta$. Suppose $\Ult_0(M|\beta,\mubar)\sats\beta=\kappa^+$. Then $\mubar$ is on the $M$ sequence.\end{thm}

\begin{proof}
Since the failure of the theorem is a $\Sigma_1$ fact about $M$, we may assume $\rho^M_1=\om$ and $M$ is $1$-sound (replacing $M$ with $\Hull_1^M(\empty)$ if necessary).

Let $\gamma$ be least such that $\beta\leq\gamma$ and $M|\gamma$ projects to $\kappa$. (So in fact $\beta<\gamma$ since $M|\beta\sats\ZFmin$.) Let $k$ be largest such that $\rho^{M|\gamma}_{k+1}\leq\kappa<\rho^{M|\gamma}_k$. Note that $\Ult_k(M|\gamma,\mubar)$ makes sense (i.e. $\mubar$ measures its subsets of $\kappa$, and if $M|\gamma$ is active with a type 3 extender $E$, then $\kappa<\beta\leq\nu_E$, since $\beta=\kappa^+$ in $M|\gamma$), and agrees with $\Ult_0(M|\beta,\mubar)$ beyond $\beta$, so is passive at $\beta$.

\begin{clm}\label{clm:mfphit}
$\Ult_k(M|\gamma,\mubar)|\beta=M|\beta$ and
the phalanx $\ph=(M,\Ult_k(M|\gamma,\mubar),\beta)$ is $\omega_1+1$-iterable.\end{clm}

\begin{proof}
An iteration on this phalanx can be reduced to a freely dropping iteration on $M$, by reducing to a freely dropping iteration on the phalanx $(M,i_\mu(M|\gamma),(\kappa^+)^M)$. Let $\pi:\Ult_k(M|\gamma,\mubar)\to i_\mu(M|\gamma)$ be the natural factor map. Then $\pi$ is a weak $k$-embedding, $\pi\com i_{\mubar}=i_\mu\rest M|\gamma$, $\pi\rest\beta=\id$, and $\pi(\beta)=(\kappa^+)^M$. So $\Ult_k(M|\gamma,\mubar)|\beta=M|\beta$. Moreover, we may use $\id:M\to M$ and $\pi$ as initial copy maps to copy an iteration up. If $E_\alpha$ has crit $\kappa$, then so does $\pi_\alpha(E_\alpha)$. In this case $E_\alpha$ measures exactly $\pow(\kappa)\int M|\gamma$, while $\pi_\alpha(E_\alpha)$ measures $\pow(\kappa)\int M$. So there is a drop in model below, to $M|\gamma$, but no drop above, and $\pi_{\alpha+1}:M_{\alpha+1}\to i_{\pi_\alpha(E_\alpha)}(M|\gamma)$. Upon leaving $N_{\alpha+1}$, we impose the appropriate drop in model and degree.\renewcommand{\qedsymbol}{$\Box$(Claim \ref{clm:mfphit})}\end{proof}

By the first part of the claim, a comparison between $M$ and  $\ph$ begins above $\beta$, so by the second part, there is a successful comparison, giving trees $\Tt$ and $\Uu$ respectively. Say $\Tt$ has last model $N$ and $\Uu$ has last model $Q$. $Q$ isn't fully sound. (The Closeness Lemma (\cite{fsit}, 6.1.5) doesn't show extenders applied to $M|\gamma$ (or all of $M$) are close, so the usual fine structure preservation arguments don't show $\rho_{k+1}$ is preserved when an extender hits $M|\gamma$. However enough preservation holds that $Q$ isn't sound.) It can't be that $N\pins Q$, since $M=\Th_1^M(\empty)$. So $Q=N$. By almost usual arguments, $Q$ is above $U$ in $\Uu$. (Add the observation that if $Q$ is above $M|\gamma$, in that the first extender used on the main branch of $\Uu$ has crit $\kappa$, then in fact $\rho_{k+1}^Q=\rho_{k+1}^{M|\gamma}=\kappa$, since $Q$ also results from $\Tt$, which starts above $\beta$, and drops in model on its main branch.) As $Q$ isn't sound, $b^\Tt$ drops, so $b^\Uu$ doesn't. So $\core_{k+1}(Q)=M|\gamma$. So $b^\Tt$'s last drop is to $M|\gamma$ and $\lh(E^\Tt_0)\leq\gamma$. Let $E^\Tt_\alpha$ be the extender hitting $M|\gamma$ along $b^\Tt$. Then $E^\Tt_\alpha$ is compatible with the core embedding, and since the exchange ordinal of $\ph$ is $\beta>\kappa$, this is compatible with $\mubar$ (and they measure the same sets). Therefore $\mubar$ is the normal measure derived from $E^\Tt_\alpha$. Finally, all extenders used had length above $\beta$, so $\mubar$ is in fact on the $M$ sequence.\renewcommand{\qedsymbol}{$\Box$(Theorem \ref{thm:fragpassive})}\end{proof}

\begin{thm}\label{thm:fragactivetp2}
Let $M$ be an $\omega_1+1$-iterable type 1 mouse, with active measure $\mu$, with critical point $\kappa$. Let $\kappa<\beta<(\kappa^+)^M$ be such that $M|\beta$ has largest cardinal $\kappa$, and is active with a type 2 extender, and let $\mubar=\mu\int M|\beta$. Let $U=\Ult_0(M|\beta,\mubar)$. Suppose $U|\beta\neq M|\beta$. Then $\mubar$ is on the $\Ult_0(M,E^M_\beta)$ sequence.\end{thm}

\begin{rem} It is the case that $U||\beta=M||\beta$, as $\mu$ coheres with $M$, and since $\beta=\kappa^+$ in $\Ult_0(M,E^M_\beta)$, there's enough closure below $\beta$ that $\Ult_0(M|\beta,\mubar)$ agrees with $\Ult_0(M,\mu)$ below $\beta$. So the hypothesis $U|\beta\neq M|\beta$ says that either $U|\beta$ is passive, or $E^U_\beta\neq E^M_\beta$. The conclusion of the theorem shows that in fact $U|\beta$ is passive.\end{rem}

\begin{rem} The hypothesis $U|\beta\neq M|\beta$ doesn't follow from the other hypotheses, so is necessary. For suppose $\beta$ is least satisfying the other hypotheses. Then $\beta=[\{\kappa\},f]^M_\mu$ where $f$ is definable over $M|\kappa$. Moreover, since $M|\beta$ projects to $\kappa$, the same holds for any $\alpha<\beta$. But then $\Ult_0(M|\beta,\mubar)|\beta=M|\beta$.\end{rem}

\begin{proof}[Proof of Theorem \ref{thm:fragactivetp2}] This is just as in the previous case, except that now $\beta=\gamma$, and $E^\Tt_0=E^M_\beta$ (where $\Tt$ is the tree on $M$). This means that all extenders used on either side of the comparison with critical point $\kappa$ measure exactly $M|\beta$. Since $E^M_\beta$ is type 2, there's no problem taking an ultrapower of $M|\beta$ with an extender whose critical point is $\kappa$.\renewcommand{\qedsymbol}{$\Box$(Theorem \ref{thm:fragactivetp2})}\end{proof}

\begin{thm}\label{thm:fragactivetp3}\setcounter{clm}{0}
Let $M$ be a type 1 mouse with active measure $\mu$, and suppose that $\Ult_0(M,\mu)$ is $(0,\omega_1+1)$-iterable. Let $\kappa=\crit(\mu)$. Let $\kappa<\beta<(\kappa^+)^M$ be such that $M|\beta$ has largest cardinal $\kappa$, is active with a type 3 extender, and let $\mubar=\mu\int M|\beta$. Suppose $\mubar$ coheres with the $\Ult(M,E^M_\beta)$ sequence (meaning if $U=\Ult(M||\beta,\mubar)$, and $\lambda=(\kappa^{++})^U$, then $U|\lambda=\Ult(M,E^M_\beta)||\lambda$, so in particular, $\beta=(\kappa^+)^U$). Then $\mubar$ is on the $\Ult(M,E^M_\beta)$ sequence.\end{thm}

\begin{proof}
Again we assume $M=\Hull_1^M(\empty)$. Let $P=\Ult(M,E^M_\beta)$ and
\[ \ph=(M,\Ult(P,\mubar),\kappa+1).\]
Note $\lambda=\kappa^{++}$ in $\Ult(P,\mubar)$.

One must be a little careful with iterations of $\ph$. Say $\Uu$ is on $\ph$, $\beta<\lh(E^\Uu_0)$ and $\crit(E^\Uu_\gamma)=\kappa$. Then $E=E^\Uu_\gamma$ goes back to $M$, but measures only $\pow(\kappa)\int M|\beta$, so applies exactly to $M|\beta$. But $M|\beta$ is type 3, and $\kappa$ is the height of $(M|\beta)^{\sq}$, so we can't form $\Ult((M|\beta)^{\sq},E)$. This problem arose in the proof of condensation, (\cite{outline}, 5.1). We need to generalize what's done there: just set $M^\Uu_{\gamma+1}=\Ult_0(M|\beta,E)$, where the ultrapower is formed without squashing, then carry on as usual. $F^{M|\beta}$ is shifted to $F^\Ult$ in the usual way (as in \ref{lem:extass}). $M^\Uu_{\gamma+1}$ is not a premouse, as $F^\Ult\rest\kappa=F^{M|\beta}\rest\nu_{M|\beta}\notin\Ult$, though $\kappa$ is a generator of $F^\Ult$. Also, $\nu_E$ is the sup of generators of $F^\Ult$ (as in \ref{lem:extass}), though $\nu_E<i_E(\kappa)$, the largest cardinal of $\Ult$. So it fails the initial segment condition, and $F^\Ult$ is not its own trivial completion. $\Ult$ does satisfy the remaining premouse axioms: as in \ref{lem:extass}, applying $F^\Ult$ to $M$ is equivalent to the two-step iteration starting with $M$, and applying $E^M_\beta$ to form $P$, followed by $E$. Coherence of $F^\Ult$ follows.

If $\gamma+1<_\Uu\delta$ and there is no drop from $\gamma+1$ to $\delta$, the above situation generalizes to $M^\Uu_{\delta}$. In particular, if $E^\Uu_{\delta}$ is the active extender of $M^\Uu_{\delta}$, then it applies to $M$, and $\Ult_0(M,E^\Uu_{\delta})$ is the model given by the iteration starting with $P$, and using the extenders
\begin{equation}\label{eqn:anomiter} \{ E^\Uu_{\gamma'}\ |\ \gamma+1\leq_\Uu\gamma'+1\leq_\Uu\delta\}. \end{equation}
In this case we'll say $M^\Uu_\delta$ is \emph{anomalous}, and so is its active extender.

Around anomalous structures, we also need to tweak the rule dictating which model an extender applies to. Given $E$ with $\crit(E)>\kappa$, $E$ will apply to (the largest possible segment of) $M^\Uu_\delta$, where $\delta$ is least such that $\crit(E)<\nu^\Uu_\delta$, or $E^\Uu_\delta$ is anomalous, and $\crit(E)<i^\Uu_{M|\beta,\delta}(\kappa)$. (Anomalous extenders are active, so coherence ensures enough agreement between models for this. The rule also guarantees generators are never moved, as usual.)

\begin{clm}\label{clm:mfphit3} $\ph$ is $\omega_1+1$-iterable for iterations following the rules described above, and whose extenders are indexed above $\lambda$.\end{clm}

\begin{proof} Let $i_\mu(M)=\Ult(M,\mu)$ and
\[ \ph'=(i_\mu(M),\Ult(i_\mu(M),i_\mu(E^M_\beta)\rest\kappa^+),\kappa^+).\]
Since $i_\mu(M)$ is iterable by assumption, $\ph'$ is clearly iterable for iterations using extenders indexed above $(\kappa^{++})^{\Ult(i_\mu(M),i_\mu(E^M_\beta)\rest\kappa^+)}$, the length of the trivial completion of $i_\mu(E^M_\beta)\rest\kappa^+$. We will reduce the necessary iteration of $\ph$ to such an iteration of $\ph'$. We start with $\pi_0=i_\mu:M\to i_\mu(M)$ and define $\pi_1:\Ult(P,\mubar)\to\Ult(i_\mu(M),i_\mu(E^M_\beta)\rest\kappa^+)$. First let $\pi:\Ult(P,\mubar)\to i_\mu(P)$ be the canonical factor map, which is $1$-elementary. Now $i_\mu(P)=\Ult(i_\mu(M),i_\mu(E^M_\beta))$, and $\rg(\pi)=\Hull_1^{i_\mu(P)}(\kappa+1)$. Since the natural factor embedding
\[ \Ult(i_\mu(M),i_\mu(E^M_\beta)\rest\kappa^+)\to \Ult(i_\mu(M),i_\mu(E^M_\beta))=i_\mu(P) \]
is the identity below $\kappa^+$, we get $\pi_1$. Note $\pi_1$ is $1$-elementary, $\crit(\pi_1)=\beta$, and $\pi_1(\beta)=\kappa^+$. We are only considering iterations of $\ph$ in which the first extender is indexed above $(\kappa^{++})^{\Ult(P,\mubar)}$, so the lifted iteration begins with high enough index.

The copying process is standard except at anomalies. We're lifting $\Uu$ to $\Vv$. Say $E=E^\Uu_\alpha$ has crit $\kappa$, so it applies exactly to $M|\beta$. Then $F=E^\Vv_\alpha=\pi_\alpha(E^\Uu_\alpha)$ also has crit $\kappa$ (since $\pi_1(\kappa)=\kappa$), so $F$ is to be applied to $i_\mu(M)$. We'll define
\[ \pi_{\alpha+1}:\Ult_0(M|\beta,E)\to i_F\com i_\mu(M|\beta). \]
Since $\crit(\pi_0)=\kappa$ but $\crit(\pi_\alpha)>\kappa$, one needs to use the simple variation on the shift lemma employed in the proof of (\cite{cmip}, 6.11). I.e., define $\pi_{\alpha+1}$ by
\[ [a,f]\mapsto [\pi_\alpha(a),i_\mu(f)\rest\kappa]. \]
Because all our copy maps fix points below $\kappa$, it's easy to see that the proof of the shift lemma still goes through with this definition (\cite{cmip} has some details). $\pi_{\alpha+1}$ is a weak $0$-embedding. (It seems likely that it won't be $1$-elementary because of the extra functions used in forming $i_F\com i_\mu(M|\beta)$.) If $E^\Uu_{\alpha+1}$ is the active extender of $\Ult_0(M|\beta,E)$, then set $E^\Vv_{\alpha+1}=i_F\com i_\mu(E^M_\beta)$ - clearly $\pi_{\alpha+1}$ then suffices for the shift lemma. By commutativity, 
\[ \pi_{\alpha+1}(i_E(\kappa))=i_F\com i_\mu(\kappa)=\nu^\Vv_{\alpha+1}. \]
Therefore if $\alpha+1<\gamma$ then
\[  \crit(E^\Uu_\gamma)<i_E(\kappa)\ \ \iff\ \ \crit(E^\Vv_\gamma)<\nu^\Vv_{\alpha+1}. \]
So our rule for which model to return to in $\Uu$ lifts to the usual rule for $\Vv$. Finally, suppose $E^\Uu_\gamma$ is to return to an anomalous structure $M^\Uu_\delta$, without triggering a drop in $\Uu$. Then $\crit(E^\Uu_\gamma)<i^\Uu_{M|\beta,\delta}(\kappa)$, and since $\pi_{\delta}\com i^\Uu_{M|\beta,\delta}(\kappa)$ is a cardinal of $M^\Vv_\delta$, there's also no drop triggered in $\Vv$ by $E^\Vv_\gamma$. So we define things at the $\gamma+1$ level as we did for $\alpha+1$ in the preceding (though using the standard shift lemma), and it works the same.
\renewcommand{\qedsymbol}{$\Box$(Claim \ref{clm:mfphit3})}\end{proof}

Armed with iterability, we complete the proof. Compare $M$ with $\ph$. Let $\Tt$ be the tree on $M$ and $\Uu$ the tree on $\ph$. All extenders used in $\Uu$ do have length above $\lambda$, since $\Ult(P,\mubar)|\lambda=\Ult(M,E^M_\beta)||\lambda$. (So $E^\Tt_0=E^M_\beta$.) Since $M=\Th_1(M)$ and the models in $\ph$ have the same $\Sigma_1$ theory, $\Tt$ and $\Uu$ have the same final model $Q$ and there is no dropping on either main branch. Suppose $Q$ is above $M$ in $\Uu$. Unless the branch begins with an anomalous extender, a contradiction is achieved by compatible extenders as usual.

Suppose $b^\Uu$'s first extender $F=E^\Uu_\delta$ is anomalous. We adopt the notation around (\ref{eqn:anomiter}); in particular, $\gamma$ is least such that $\gamma+1\leq_\Uu\delta$. $F$'s action on $M$ is an iteration beginning with $E^M_\beta$ and $E^\Uu_\gamma$. Since $F$'s generators aren't moved along $b$, $P=\Hull_1^Q(\kappa)$, and the hull embedding from $P$ to $Q$ is compatible with $E^\Uu_\gamma$ through $\nu^\Uu_\gamma$. The usual arguments then give $E=E^M_\beta=E^\Tt_0$ is the first extender used on $b^\Tt$, resulting in $P$. Since $i^\Tt_{P,Q}$ has crit $\geq\kappa$, it agrees with the hull embedding, implying $E^\Uu_\gamma$ is compatible with the second extender used along $b^\Tt$.

So $Q$ is above $U=\Ult(P,\mubar)$. Since $\crit(\mubar)=\kappa$ and $\crit(i^\Uu_{U, Q})>\kappa$, the argument of the last paragraph shows $b^\Tt$ uses $E^M_\beta$ first and an extender compatible with $\mubar$ second, so that $\mubar$ is on the $P$-sequence.\renewcommand{\qedsymbol}{$\Box$(Theorem \ref{thm:fragactivetp3})}
\end{proof}

\begin{rem} Suppose $M$ is a sound mouse with active extender $E$, such that $E\rest\gamma+1$ is a type Z segment. The initial segment condition for premice doesn't appear to give $E\rest\gamma+1\in M$, but (\cite{deconstruct}, 2.7) does (as we're assuming $M$ is iterable). There is no proof of this theorem in \cite{deconstruct}. The foregoing argument can easily be adapted to prove the following: let $\mubar$ be the normal measure over $\pow(\gamma)\int\Ult(M,E\rest\gamma)$ given by the factor embedding
\[ \Ult(M,E\rest\gamma)\to\Ult(M,E\rest\gamma+1). \]
Then $\mubar$ is on the $\Ult(M,E\rest\gamma)$ sequence. It's easy to see that $(\gamma^+)^{\Ult(M,E\rest\gamma)}$ is the next generator of $E$ after $\gamma$, so this gives $E\rest\gamma+1\in M$.\\
For this, we may assume $M$ has only one more generator after $\gamma$. Just compare $M$ with the phalanx $(M,\Ult(M,E\rest\gamma+1),\gamma+1)$. The iterability is obtained by reducing an iteration to one on $(M,\Ult(M,E),\nu_E)$. One faces the same complications as in the above proof.\end{rem}

\pagebreak
\section{Stacking Mice}\label{sec:stacking}
Given a mouse $M$, one might ask whether $\es^M$ is definable without parameters over $M$'s universe. If so, clearly $V=\HOD$ in $M$. Steel showed that for $n\leq\om$, the answer is ``yes'' for $M_n$. He actually showed $M_n$ satisfies $\VK$ (and $\es=\es^K$) ``between'' consecutive Woodins. An argument is given in \cite{cmpw}. Although this works for mice higher in the mouse order, exactly how high seems unknown. In this section, we give a new proof of ``yes'' for $M_n$ ($n\leq\om$), and extend the result to various other mice. We argue without $K$, just using the following type of internal definition of $\es$: given $\es\rest\kappa$ for some cardinal $\kappa$, $\es\rest\kappa^+$ is obtained by stacking appropriate mice projecting to $\kappa$. (Clearly this is related to $K$, especially in light of Schindler's result \cite[3.5]{Kstack}, that above $\aleph_2$, $K$ is the stack of projecting mice.\footnote{Footnote added January 2013: See \cite[3.5]{Kstack} for the precise statement.}) For this definition to work, $M$ must at least be sufficiently self-iterable. Not far into non-tame mice, our methods break down because of this requirement (see \ref{fact:nontame}). We'll need to appeal to the argument of \ref{thm:cohere} to see that the iterability of candidate mice actually guarantees that they are initial segments of $\J^\es$. First we consider the type of iterability needed.

\begin{dfn}\index{extender-full} Let $P$ be a sound premouse with a cardinal $\rho\leq\rho^P_\om$. Then $\Sigma$ is an $\alpha,\rho$\emph{-extender-full} iteration strategy for $P$ if:
\begin{itemize}
\item $\Sigma$ is an $\alpha$-iteration strategy for $P$ above $\rho$ (meaning for trees above $\rho$).
\item  Whenever $Q$ is the last model of an iteration via $\Sigma$, and $E$ is such that $(Q||\lh(E),E)$ is a premouse with $\crit(E)<\rho$,
\[ \Ult_\om(P,E) \textrm{ is wellfounded }\ \ \iff\ \  E \textrm{ is on the } Q \textrm{ sequence}. \]
\end{itemize}
Let $M$ be a premouse satisfying $\ZFC$, and $\eta$ a cardinal of $M$. $M$ is \emph{extender-full self-iterable at $\eta$} if for each $P\ins M$ such that $\rho_\om^P=\eta$,
\[ M\sats P \textrm{ is } (\eta^++1),\eta\textrm{-extender-full iterable}. \]
\end{dfn}

\begin{lem}\label{lem:extfull}
Suppose $M$ is a premouse satisfying $\ZFC$, and $M$ is extender-full self-iterable at $\eta$. Suppose $P\in M$ is a sound premouse extending $M|\eta$, projecting to $\eta$, and such that
\[ M\sats P \textrm{ is } (\eta^++1),\eta\textrm{-extender-full iterable}. \]
Then $P\ins M$.

Therefore if $M$ is extender-full self-iterable at all of its cardinals, $M\sats V=\HOD$.
\end{lem}
\begin{proof}
We have $P\in M|\beta$ where $M|\beta$ projects to $\eta$. In $M$, we can compare $P$ with $M|\beta$ using their extender-full iteration strategies. For suppose $E$ is on the sequence of  an iterate of $P$, with $\crit(E)<\eta$, so by extender-fullness, $\Ult_\om(P,E)$ is wellfounded. Since $\eta$ is a cardinal of $M$ and $P$ agrees with $M|\beta$ below $\eta$, $E$ is total over $M$. Moreover, $\Ult(M,E)$ is wellfounded, since $M|(\crit(E)^+)^M=P|(\crit(E)^+)^P$ (as in the 3rd paragraph of the proof of \ref{thm:coarseDodd}, witnesses to illfoundedness can be collapsed into $M|(\crit(E)^+)^M$). So $\Ult_\om(M|\beta,E)$ is also wellfounded.

Therefore $E$ can't be used during comparison. For doing so would require agreement at that stage below $\lh(E)$, so by extender-fullness, $E$ would also be on the model above $M|\beta$, so it \emph{doesn't} get used in comparison; contradiction. The same argument applies in the opposite direction, so only extenders with crit $\geq\eta$ need be used during comparison, so the iteration strategies suffice. Since the models being compared project to $\eta$ and both sides of the comparison are above $\eta$, $P\pins M|\beta$.
\renewcommand{\qedsymbol}{$\Box$(Lemma \ref{lem:extfull})}\end{proof}

\begin{cor}\label{cor:MnV=HOD} For $n\leq\om$, $M_n\sats V=\HOD$.\end{cor}
\begin{proof} Let $\delta_0,\ldots,\delta_{n-1}$ be the Woodins of $M_n$, $\delta_{-1}=0$. $M_n|\delta_k$ knows its own iteration strategy for trees above $\delta_{k-1}$ of length $<\delta_k$. (Given some tree $\Tt$, choose the branch $b$ such that $\Col(\om,\Tt)$ forces $Q(b,\Tt)$ to be $\Pi^1_{n-k}$-iterable. Here ``$\Pi^1_1$-iterable'' just means wellfounded. See the next section or \cite{projwell} for discussion.) By \ref{thm:cohere} relativized to  iteration strategies above some cut-point $\alpha$ (see the remark following \ref{thm:cohere}), this strategy is $\delta_k,\rho$-extender-full for each $M_1$-cardinal $\rho$ in the interval $[\delta_{k-1},\delta_k)$. Therefore \ref{lem:extfull} applies.

If $\delta_i$ is the $i^\nth$ Woodin of $M_\om$, $M_\om|\delta_i$ also knows its own strategy above $\delta_{i-1}$. This follows from [(\cite{outline}, \S7), but we provide an argument here. Consider a tree $\Tt$ on and in $M_\om|\delta_0$. The correct branch $b$ is that for which $Q(b,\Tt)$ is weakly $(\deg^\Tt(b),\om)$-iterable. This condition is $\Sigma_1^{L(\RR)}$. Since $M_\om$ can compute $L(\RR)$ truth via the symmetric collapse of its Woodins, and the collapse is homogeneous, $b\in M_\om$, and $M_\om$ defines the correct strategy. But therefore in $M_\om|\delta_0$, $Q(b,\Tt)$ is also $\lambda$-iterable above $\delta(\Tt)$, for any $\lambda<\delta_0$. Since $Q(b,\Tt)$ has size $<\delta_0$, such iterability suffices to identify it. Now apply \ref{thm:cohere} (relativized above a cut-point) and \ref{lem:extfull}.
\renewcommand{\qedsymbol}{$\Box$(Corollary \ref{cor:MnV=HOD})}\end{proof}

\begin{rem}The self-iterability of $M_\om$ established above is not optimal. For example, let $\alpha<\delta_0$; then using genericity iterations, one can see that $M_\om|(\alpha^+)^{M_\om}$ can also define its own strategy restricted to trees it contains (see (\cite{outline},\S7)). The above argument only gives this at limit cardinals.\end{rem}

We now move on to more complex mice, where self-iterability is more difficult to establish. Given some $\Tt$, we'll build an appropriate fully backgrounded $L[\es]$ construction over $M(\Tt)$ to search for its Q-structure. To prove this search is successful (the construction reaches the Q-structure), we need to restrict to mice which ``rebuild themselves''. That is, the mouse will be minimal for some hypothesis $\varphi$, and we'll establish that an unsuccessful $L[\es]$ construction leads to a model satisfying $\varphi$. We'll also establish that each stage of the construction is in fact an iterate of the Q-structure, above $\delta(\Tt)$. The Q-structure will have no level modelling $\varphi$, so the construction \emph{must} be successful.

The method employed for maintaining that the levels of a construction with good enough background certificates are iterates of some mouse was described to the author by Steel, and is an amalgamation of ideas from 3.2 and 3.3 of \cite{maxcore}. This is the core of the argument; our contribution was in noticing that these ideas lead to \ref{cor:reachM} and, combined with the argument for \ref{thm:cohere}, to \ref{thm:tamesi}.

\begin{dfn}\index{ms-array}
Let $P$ be a sound premouse and $\lambda\leq\OR+1$. A sequence $\left<N_\alpha\right>_{\alpha<\lambda}$ is an \emph{ms-array above} $P$ if $N_0=P$, $\OR^P\leq\rho_\om^{N_\alpha}$ for $\alpha+1<\lambda$, and the sequence satisfies the requirements of a $K^c$ construction (see \cite{outline}) \emph{other than}
\begin{itemize}
\item[(a)] $N_0=V_\om$, and
\item[(b)] The existence of background certificates.
\end{itemize}
If $P=V_\om$, the sequence is simply an \emph{ms-array}.
\end{dfn}

The following definition is a variation on that of a weak $\Aa$-certificate (\cite{maxcore}, 2.1).

\begin{dfn}\label{dfn:xcert}\index{certificate} Suppose $P$ is a premouse with active extender $F$. Let $\kappa=\crit(F)$, $\nu=\nu_F$, and $x\in\RR$.
An \emph{$x$-certificate (for $P$)} is an elementary $\pi:N\to V_\theta$ such that
\begin{itemize}
\item[(a)] $N$ is a transitive, power admissible,
\[ \pow(\kappa)^P\un\{x\}\sub N, \]
\item[(b)] $F\rest(P\cross[\nu]^{<\om})=E_\pi\rest(P\cross[\nu]^{<\om})$,
\item[(c)] $\pi(P|\kappa)||\lh(F)=P||\lh(F)$.
\end{itemize}
\index{certified}A premouse $P$ with active extender $F$ is \emph{$x$-certified} if there is an $x$-certificate for $P$.

\index{$\om$-mouse}An \emph{$\om$-mouse} is an $\om$-sound, $\om_1+1$-iterable mouse projecting to $\om$.

\index{mouse certified}\index{mouse maximal}An ms-array $\Cc=\left<N_\alpha\right>_{\alpha<\lambda}$ is \emph{mouse certified} if for each active $N_{\alpha+1}$: if there's an $\om$-mouse not in $N_{\alpha+1}$, and $M$ is least such, then $N_{\alpha+1}$ is $M$-certified. $\Cc$ is \emph{mouse maximal} if it's mouse certified, and $N_{\alpha+1}$ is active whenever possible under this constraint.

\index{reach}An ms-array $\left<N_\alpha\right>_{\alpha<\lambda}$ \emph{reaches} $M$ if there is $\alpha$ such that $M=\core_\om(N_\alpha)$.
\end{dfn}

\begin{rem} Having the codomain of an $x$-certificate $\pi$ be a $V_\theta$ isn't that important; all we really need is that it is transitive and contains $V_{\om+2}$.\end{rem}

\begin{lem}\label{lem:Mcomp}
Let $\left<N_\alpha\right>_{\alpha<\lambda}$ be a mouse certified ms-array, where $\lambda\leq\om_1$. Let $M$ be an $\om$-mouse, and let $\Sigma$ be $M$'s unique $\om_1^++1$ strategy. Then for each $\alpha$, either $M\ins N_\alpha$, or there is a $\Sigma$-iterate $P$ of $M$ such that $N_\alpha\ins P$. In other words, $N_\alpha$ does not move in $\Sigma$-comparison with $M$.\end{lem}
\begin{proof} The proof is like part of the proof of \ref{thm:tamesi}, so we omit it.\renewcommand{\qedsymbol}{$\Box$(Lemma \ref{lem:Mcomp})}\end{proof}

\begin{cor}\label{cor:uniquemaxcon} Let $\Cc,\Cc'$ be countable mouse maximal constructions of the same length. If $\Cc$ does not reach some $\om$-mouse, then $\Cc=\Cc'$. Therefore there's a unique maximal construction of minimal (limit) length reaching all reachable $\om$-mice; denote this by $\Cc^*$.\qed\end{cor}

\begin{lem}\label{lem:reachM}
Suppose $\Cc$ is a mouse maximal construction of length $\om_1$ and $M$ is an $\om$-mouse. Then $\Cc$ reaches $M$. Thus if there are $\om_1$ many $\om$-mice, the output of $\Cc^*$ is the stack of them all.\end{lem}
\begin{proof}
Assume otherwise, and let $M$ be the least counterexample, with iteration strategy $\Sigma$. Let $W=N_{\om_1}$ be the natural limit of the construction; notice $W$ has height $\om_1$, and $W|\om_1^W=M|\om_1^M\pins M$. By \ref{lem:Mcomp}, $W\ins M^\Tt_{\om_1}$ for a $\Sigma$-iteration $\Tt$ of $M$ of length $\om_1+1$. Let $\theta$ have high cofinality and $V_\theta\prec_k V$ for some large $k$. Let $X\prec V_\theta$ be countable,
transitive below $\om_1$, $\left<N_\alpha\right>\in X$ and $M\in X$. Let $\pi:N\to V_\theta$ be the hull uncollapse of $X$. Let $\kappa=\crit(\pi)$, so $\pi(\kappa)=\om_1$. $\Cc$ is $M$-maximal after $\beta$, where $M|\om_1^M=N_\beta$; we will show that $\pi$ contradicts this maximality.

Let $b=\Sigma(\Tt\rest\om_1)$; note $\pi(b\int\kappa)=b\in X$. So $b$ is unbounded below $\kappa$ by elementarity, and as $b$ is club in $\om_1$, we get $\kappa\in b$. Thus $b\int\kappa=\Sigma(\Tt\rest\kappa)\in N$ and $\pi(\Tt\rest\kappa+1)=\Tt$. In particular, $M^\Tt_\kappa\in N$.
\begin{equation}\label{eqn:mapagree} \pi\rest M^\Tt_\kappa = i^\Tt_{\kappa,\om_1}; \end{equation}
this is as in the proof that comparison terminates: for $\eta<\kappa$ and $x\in M^\Tt_\eta$, since $\eta,x$ are countable in $N$,
\[ \pi(i^\Tt_{\eta,\kappa}(x)) = i^\Tt_{\eta,\om_1}(x)=i^\Tt_{\kappa,\om_1}\com i^\Tt_{\eta,\kappa}(x). \]

Since $\crit(i^\Tt_{\kappa,\om_1})=\kappa$, $\pow(\kappa)\int W=\pow(\kappa)\int M^\Tt_\kappa$. Let $F=E_\pi\rest(W\cross[\om_1]^{<\om})$. By (\ref{eqn:mapagree}), $F$ is compatible with some $E^\Tt_\gamma$. ($E^\Tt_\gamma$ also measures exactly $\pow(\kappa)\int W$.) Now $\lh(E^\Tt_\gamma)<\om_1$ is a successor cardinal of $W$. So $W|\lh(E^\Tt_\gamma)=N_\alpha$ where $\alpha$ is the limit of stages $\beta$ so that $N_\beta$ projects strictly below $\lh(E^\Tt_\gamma)$. But $\pi$ is an $M$-certificate for $(N_\alpha,E^\Tt_\gamma)$. So by mouse maximality, $N_{\alpha+1}=(N_\alpha,E)$ for some $E$, so $N_{\alpha+1}$ projects below $\lh(E^\Tt_\gamma)$. Contradiction.\renewcommand{\qedsymbol}{$\Box$(Lemma \ref{lem:reachM})}\end{proof}

\begin{cor}\label{cor:reachM} All $\om$-mice are reachable. Therefore $\Cc^*$ (as in \ref{cor:uniquemaxcon}) reaches all $\om$-mice.\end{cor}
\begin{proof}
Suppose $M$ is an unreachable $\om$-mouse, with iteration strategy $\Sigma$. Let $N$ be the last model of a mouse maximal construction $\Cc$. By \ref{lem:Mcomp}, there's a $\Sigma$-iterate $Q$ of $M$ such that $N\ins Q$. Therefore $\core_\om(N)$ exists, and $\Cc$ can be properly extended. Any ms-array of limit length can be uniquely properly extended. By \ref{cor:uniquemaxcon}, there is at most one mouse maximal construction of any given countable length. So we can build a unique mouse maximal construction of length $\om_1$. By \ref{lem:reachM}, this construction reaches $M$ at some countable stage; contradiction.\renewcommand{\qedsymbol}{$\Box$(Corollary \ref{cor:reachM})}\end{proof}

\begin{lem}\label{lem:cardnu} Let $\left<N_\alpha\right>$ be an ms-array, with last model $N$. Let $E$ on $\es_+^N$ be total over $N$ and such that $\nu_E$ is a cardinal of $N$. Then there is $\alpha$ such that $N_{\alpha+1}=(N_\alpha,E)$.\end{lem}
\begin{proof}
Let $E'$ be the original ancestor of $E$ (so there is a stage $\beta$ so that $N_{\beta+1}=(N_\beta,E')$ and $E'$ eventually collapses to $E$ during later stages of construction). Because $\nu_E$ is a cardinal of $N$, if $E'\neq E$ then $E$ is the trivial completion of $E'\rest\nu_E$, and if $\gamma\geq\beta$ then $\rho_\om^{N_\gamma}\geq\nu_E$. So $\nu_E$ is a cardinal of $N_\beta$ and $E$ is on the $N_\beta$ sequence. Now apply induction with $N_\beta$ and $E$.\renewcommand{\qedsymbol}{$\Box$(Lemma \ref{lem:cardnu})}\end{proof}

\begin{dfn}\label{dfn:appropriate}\index{appropriate}
Suppose $V=L[\es]$. Let $P$ be a sound premouse. An ms-array $\Cc=\left<N_\alpha\right>$ above $P$ is \emph{appropriate} if whenever $N_{\alpha+1}=(N_\alpha,E)$, there is a total extender $G$ such that
\begin{itemize}
\item $G$ is indexed on $\es$,
\item $\nu_G\geq\nu_E^+$ and is a cardinal,
\item $G\rest(N_\alpha\cross[\nu_E]^{<\om}) = E\rest(N_\alpha\cross[\nu_E]^{<\om})$,
\item $i_G(N_\alpha)||\lh(E)=N_\alpha$.
\end{itemize}
\index{maximal}An appropriate ms-array is \emph{maximal} if it adds an extender whenever possible. By minimizing the choice of background extenders, one obtains a canonical such construction.
\end{dfn}

\begin{thm}[Steel, Schlutzenberg]\label{thm:tamesi} Let $M$ be the least non-tame mouse. Then $M|\crit(F^M)$ is extender-full self-iterable at its cardinals, so satisfies $V=\HOD$.
\end{thm}
\begin{proof}\setcounter{clm}{0}
By \ref{cor:reachM} there is a countable mouse maximal construction $\left<N_\alpha\right>_{\alpha\leq\xi+1}$ such that $M=\core_\om(N_{\xi+1})$. Note that with $\beta$ such that $N_\beta=M|\om_1^M$, for all $\alpha+1>\beta$, if $N_{\alpha+1}$ is active, then it's $M$-certified. Note $N_{\xi+1}=(N_\xi,E)$, $N_\xi$ is tame, and $\delta=\nu_E$ is a Woodin cardinal in $N_\xi$. Let $N=N_\xi$ and $\kappa=\crit(E)$.

Work in $N$. Let $\eta<\kappa$ be a cardinal and $N|\beta$ project to $\eta$. Let $\Gamma$ be the following partial strategy for trees on $N|\beta$ above $\eta$. Suppose $\Tt$ is of limit length via $\Gamma$. Let $Q'$ be the Q-structure for $\delta(\Tt)$ reached by the canonical appropriate ms-array above $M(\Tt)$ through $\delta$ stages. Let $Q$ be the $\delta$-hull of $Q'$: i.e. if $\rho_\om^{Q'}\geq\delta$ then $Q=\core_\om(Q')$; otherwise $Q=\Hull_{n+1}^{\core_n(Q')}(\delta)$ where $\rho_{n+1}^{Q'}<\delta\leq\rho_n^{Q'}$. Then $\Gamma(\Tt)$ is the unique $b$ such that $Q\ins M^\Tt_b$. (If this definition fails or yields an illfounded branch, $\Gamma(\Tt)$ is undefined.) We will show that $\Gamma$ is a $\kappa,\eta$-extender-full strategy for $N|\beta$. Moreover (from outside $N$), $\Gamma$ agrees with the strategy $\Sigma_{N|\beta}$ for $N|\beta$ inherited from $M$'s strategy and \ref{lem:Mcomp}.

Let $\Tt$ be of limit length via both $\Gamma$ and $\Sigma_{N|\beta}$. Let $\left<P_\alpha\right>$ be the models of $N$'s canonical appropriate ms-array above $M(\Tt)$ through $\kappa$ stages, or until a Q-structure for $M(\Tt)$ is reached. The construction doesn't break down before reaching a Q-structure, using Claim \ref{lem:Mcomp} and that $M$ is iterable. In fact, let $Q(\Tt)$ be the $\Sigma$-blessed Q-structure for $M(\Tt)$. Then:

\begin{clm}\label{clm:innercon}
For each $\alpha$, $P_\alpha$ is a segment of a $\Sigma$-iterate of $Q(\Tt)$ above $\delta(\Tt)$ (and so a $\Sigma$-iterate of $M$).\end{clm}
\begin{proof}
We proceed by induction on $\alpha$. By Claim \ref{lem:Mcomp}, $N|\beta$ is a segment of a $\Sigma$-iterate of $M$. Since $\Tt$ is via $\Sigma_{N|\beta}$ and $P_0=M(\Tt)\ins Q(\Tt)$, the claim holds at $\alpha=0$.

The case $\alpha=0$ is trivial as $P_0=M(\Tt)$.

Assume $P_\alpha\ins R$, where $R$ is the last model of the tree $\Uu$ on $Q(\Tt)$ above $\delta(\Tt)$, and that $P_\alpha$ is not a Q-structure for $M(\Tt)$.

If $P_\alpha$ is unsound, then $P_\alpha=R$ and $M(\Tt)\ins\core_\om(P_\alpha)\pins M^\Uu_\gamma$ for some $\gamma$. (There must be a drop in model on $\Uu$'s main branch as we haven't yet reached a Q-structure.) Therefore $P_{\alpha+1}=\J_1(\core_\om(P_\alpha))\ins M^\Uu_\gamma$, so $\Uu\rest\gamma+1$ works.

Suppose $P_\alpha$ is sound. Again here, $\J_1(P_\alpha)\ins R$. So assume $P_{\alpha+1}=(P_\alpha,F)$. We must show $F$ is the last extender used in $\Uu$. Let $F^*$ be the canonical background extender for $F$ in $N$. Then $\mu=\crit(F)=\crit(F^*)$ is not Woodin in $P_\alpha$ (by tameness of $P_\alpha|\mu$); let $P_\alpha|\gamma$ be the Q-structure for $P_
\alpha|\mu$. Now $\nu_{F^*}$ is a cardinal of $N$. By \ref{lem:cardnu} there is an $M$-certificate $\pi:S\to V_\theta$ for $N|\lh(F^*)$.

Since the construction $\left<N_\alpha\right>$ reaching $M$ was countable, $\crit(\pi)=\mu=\om_1^S$. We have $\Sigma\in V_\theta$. Let $\Sigma^S$ be $S$'s (unique) $\mu+1$-strategy for $M$. It's easy to see $\Sigma^S$ actually agrees with $\Sigma$ (where they both apply). (If $b=\Sigma^S(\Vv)$ where $\Vv$ has length $\mu$, then $\pi(b)\int\mu=b$, so $b=\Sigma(\Vv)$ also.) So by comparing $M$ with $P_\alpha|\mu$, $S$ obtains $\Uubar=\Uu\rest\mu+1$ (with $M(\Uu\rest\mu)=P_\alpha|\mu$).

In $V_\theta$, $\pi(\Uubar)$ results from comparing $M$ (via $\Sigma$) with $\pi(P_\alpha|\mu)$. Since $\lh(F)<\nu_{F^*}$ and $P_\alpha\in N$, \ref{dfn:xcert} and \ref{dfn:appropriate} give
\[ \pi(P_\alpha|\mu)||\lh(F)=i_{F^*}(P_\alpha|\mu)||\lh(F)=P_\alpha. \]
Since $\Uu$ is also via $\Sigma$, we get $\Uu\rest\lambda+1=\pi(\Uubar)\rest\lambda+1$, where $\lambda$ is least such that $\lh(E^\Uu_\lambda)\geq\lh(F)$ or $\lh(E^{\pi(\Uubar)}_\lambda)\geq\lh(F)$.

Now as in (\ref{eqn:mapagree}) of \ref{lem:reachM}, $M^\Uu_\mu=M^{\pi(\Uubar)}_\mu$ and
\[ i^{\pi(\Uubar)}_{\mu,\pi(\mu)} = \pi\rest M^\Uu_\mu. \]
But $M^\Uu_\mu\int\pow(\mu)\super P_\alpha\int\pow(\mu)$ and $\pi$, $F^*$ and $F$ agree on $P_\alpha\int\pow(\mu)\cross[\nu_F]^{<\om}$. So if $\xi+1$ is least in $(\mu, \pi(\mu))_{\pi(\Uubar)}$, then $E^{\pi(\Uubar)}_\xi$ is compatible with $F$. The agreement between $\Uu$ and $\pi(\Uubar)$ implies $\lh(E^{\pi(\Uubar)}_\xi)\geq\lh(F)$. If $F$ and $E^{\pi(\Uubar)}_\xi$ measure the same sets, the initial segment condition implies $F$ is on the $M^{\pi(\Uubar)}_\xi$ sequence, and therefore on the $M^{\pi(\Uubar)}_\lambda=M^\Uu_\lambda$ sequence as required. Otherwise $F$ measures less sets, so $(\mu^+)^F<(\mu^+)^{E^{\pi(\Uubar)}}_\xi$, which means $F$ is type 1. (If $(\mu^+)^F<\nu_F$, then the identity of $(\mu^+)^F$ is coded directly into $F\rest\nu_F$). So \ref{thm:fragpassive} implies $F$ is on the $M^{\pi(\Uubar)}_\xi$ sequence in this case. (It seems plausible that this should arise when some normal measure $E$ on $\mu$ is added to the construction, which remains total after constructing through all the ordinals. Condensation then gives cofinally many levels $\xi<(\mu^+)^E$ which are active with submeasures. Their original ancestors may have higher critical points, but if the $F$ above was a submeasure's ancestor (with crit $\mu$), then $F$ measures less than $E$, but $E^{\pi(\Uubar)}_\xi$ measures all of $M^\Uu_\mu$, which contains all sets measured by $E$.)

Finally consider a limit $\alpha$. $P_\alpha$ is the lim inf of the agreeing segments of the sequence $\left<P_\beta\right>_{\beta<\alpha}$. Let $\Uu_\beta$ witness the claim for $P_\beta$. Let $\Uu$ be the lim inf of the $\Uu_\beta$'s and let $M^\Uu_\gamma$ be its last model. If $M^\Uu_\gamma|\OR^{P_\alpha}$ is passive then $P_\alpha\ins M^\Uu_\gamma$. Otherwise $\Uu\conc F$ works, where $F$ is $M^\Uu_\gamma|\OR^{P_\alpha}$'s active extender: by induction, $\Uu$ is above $\delta(\Tt)$, so $\delta(\Tt)<\lh(F)$. Since $\Uu$ is on $Q(\Tt)$ and $Q(\Tt)$ is tame (as $N|\beta$ is tame), $\crit(F)>\delta(\Tt)$ also. Therefore $\Uu\conc F$ is above $\delta(\Tt)$.\renewcommand{\qedsymbol}{$\Box$(Claim \ref{clm:innercon})}\end{proof}

Now suppose a Q-structure is not reached and let $P=P_\kappa$. Let $\pi:S\to V_\theta$ be an $M$-certificate for $E$. $(N,E)$ is an iterate of $M$, and $\pi(N|\kappa)$ is a segment of an iterate of $M$. By \ref{dfn:xcert}(c), $\pi(N|\kappa)||\lh(E)=N$, so $E$ was used in the iteration producing $\pi(N|\kappa)$. So $\delta$ is Woodin in $\pi(N|\kappa)$ (as it's Woodin in $\Ult(N,E)$), and therefore also in $\pi(P)$. (Showing Woodinness goes into $\pi(P)$ is easy in our situation, using the same method we're about to use to show the strength of $\kappa$ goes in.) Let $\Uu$ be the tree on $M$ iterating out to $P$. As before, $\Uu\in S$ and has length $\kappa+1$, and $\pi(\Uu)$ is the tree iterating out to $\pi(P)$, using an extender compatible with
\[ \pi\rest(\pow(\kappa)\int\pi(P))\cross[\delta]^{<\om}.\]
As in the last paragraph of \ref{lem:reachM}, that extender can't have length $\lambda<\delta$. Therefore it has natural length at least $\delta$, so it is a non-tame extender. This contradicts Claim \ref{clm:innercon}, since $Q(\Tt)$ is tame.

So the construction does reach a Q-structure $Q'$, and by Claim \ref{clm:innercon}, its $\delta$-hull $Q$ (described at the start of this proof) is $Q(\Tt)$. Since $\Sigma(\Tt)$ is the unique branch $b$ so that $Q(\Tt)\ins M^\Tt_b$, we get that $\Gamma(\Tt)=\Sigma(\Tt)$, as desired.

The facts established so far reflect down to $M$ (the sound version of $N$), so (confusing notation a little) let $\kappa$,$\eta$,$\beta$,$\Gamma$ now play the analagous roles in $M$ that they did in $N$ above (so $\kappa$ is the critical point of $M$'s active extender, etc.). It remains to show that $\Gamma$ is a $\kappa,\eta$-extender-full strategy in $M$. So suppose $\Tt$ is a normal tree on $M|\beta$, above $\eta$, via $\Gamma$, of length $<\kappa$, and $F$ fits on the sequence of $\Tt$'s last model, $\crit(F)<\eta$ and $\Ult_\om(N|\beta,E)$ is wellfounded. As usual, since $\eta$ is a cardinal of $M$, the wellfoundedness of $\Ult(M,E)$ follows. One can run the argument of \ref{thm:cohere}, with a small alteration. Note that $\Tt$ is actually via $\Sigma$. For it is guided by Q-structures whose iterability is guaranteed by the fact that they are built by full background extender constructions inside $M$ (or alternatively, by the fact that they lift to the correct Q-structures found inside $N$). By tameness, the correct branch is chosen, so $\Gamma$ agrees with $\Sigma$. So the phalanx of Claim \ref{clm:iter} of \ref{thm:cohere} will be $\om_1+1$-iterable in $V$, since iterations reduce to iterations of the phalanx $\Phi(\Tt)$, which is $\om_1+1$-iterable in $V$. So we get $E$ is on the sequence of $\Tt$'s last model, as desired.
\renewcommand{\qedsymbol}{$\Box$(Theorem \ref{thm:tamesi})}\end{proof}

\begin{rem} A similar argument (but simpler toward the end) can be used to show that other tame mice, such as the least mouse with an inaccessible limit of Woodins, are also extender-full self-iterable at their cardinals.

However, non-tame mice quickly produce situations where our method for proving \ref{thm:tamesi} can't work. (Note that the self-iteration strategy obtained in \ref{thm:tamesi} agreed with $V$'s strategy.) Moreover, the approach used in proving \ref{cor:AC} also fails in the following example, which is a simple variation on an observation possibly due to Steel, given in (\cite{sile}, 1.1).\end{rem}

\begin{fact}\label{fact:nontame}
Suppose $N$ is a countable mouse modelling $\ZFC$, $\tau$ is an $N$-cardinal, $N$ has a cut-point $\eta$ with $\tau<\eta<(\tau^+)^N$, and $P'\ins P''\pins N$ are such that $P'$ is active with $E'$, $\crit(E')<\tau$, and $\delta>\tau$ is Woodin in $P'$, and $P''$ projects to $\tau$. Let $\Sigma$ be an iteration strategy for $N$. Then $N$ does not know $\Sigma$ restricted to trees on $P''$, above $\tau$, of length $\leq(\eta^+)^N$.\end{fact}

\begin{proof}
Let $P\ins N$ be least extending $P'$ projecting strictly less than $\delta$. Let $\rho^P_{n+1}<\delta\leq\rho_n^P$. Let $\kappa>\rho^{P'}_{n+1}$ be least that's measurable in $P'$. Let $\PP$ be the extender algebra of $P'$ with $\delta$ generators, using only crits above $\kappa$. In $N$, iterate $P$, first linearly with a normal measure on $\kappa$ and its images $\eta$ times, then iterate $M^\Tt_\eta$ to make $N|\eta$ generic over $i^\Tt_{0,\xi}(\PP)$, where $M^\Tt_\xi$ is $\Tt$'s last model. Let $E=i_{0,\xi}(E')$. Note $N|\eta$ is $\Ult(N,E)$ generic over $i^\Tt_{0,\xi}(\PP)$. Let $Q\ins N$ be least projecting to $\tau$, with $i^\Tt_{0,\xi}(\delta)\leq\OR^{Q}$. Since $\tau<i^\Tt_{0,\xi}(\delta)$ and the latter is a cardinal in $\Ult(N,E)[N|\eta]$, $Q\notin\Ult(N,E)[N|\eta]$. Let $G$ be $\Ult(N,E)[N|\eta]$-generic for the collapse of $\OR^{Q}$. In $\Ult(N,E)[N|\eta][G]$, let $S$ be the tree of attempts to build a sound premouse $R$ looking like $Q$: it should extend $N|\eta$, have $\eta$ a cut-point, be sound and project to $\tau$; $S$ also builds an elementary $\sigma:R\to i_E(Q)$. $S$ is illfounded because of $Q$ and $i_E\rest Q$; therefore $\Ult(N,E)[N|\eta][G]$ has such an $R,\sigma$. But clearly $i_E(Q)$ is iterable above $i_E(\eta)$, so $R$ is iterable above $\eta$, and it follows that $R=Q$. This was independent of the particular $G$, so $Q\in\Ult(N,E)[N|\eta]$; contradiction.
\renewcommand{\qedsymbol}{$\Box$(Fact \ref{fact:nontame})}\end{proof}

However, neat non-tame mice might satisfy $\VHOD$ some other way. $M_{\AD_\RR}$ is the minimal proper class mouse with a limit $\lambda$ of Woodins which is a limit of cardinals strong below $\lambda$.

\begin{ques}[Steel] Does $M_{\AD_\RR}$ satisfy $\VHOD$?\end{ques}

\begin{ques}[Steel] Suppose $M$ is a mouse modelling $\ZFC$. Does $M$ satisfy $\VHOD(X)$ for some $X\sub\om_1^M$?\end{ques}

A contender for the set $X$ here is the set of all countable elementary substructures of levels of $M$. (This was used in Woodin's portion of the proof of \ref{cor:AC}.)

\pagebreak
\section{Homogeneously Suslin Sets\protect\footnote{Footnote added January 2013: The material in this section is covered better in \cite{hsstm}, where things are explained more clearly and extensions are obtained. Also, a correction and extension of the material on Finite Support in \S 4 is made (see footnotes in \S 4), and the extension is essentially used in the proof of \ref{thm:homMn}.}}\label{sec:hom}
Kunen's analysis of the measures in $L[U]$ lead to the following observation of Steel:
\[ L[U]\sats\textrm{The homogeneously Suslin sets of reals are the } \bfPi^1_1\textrm{ sets.} \]
Here we consider the situation in $M_n$ ($n<\om$), and in certain models below $M_1$. We establish in $M_n$ an upper bound on the homogeneously Suslin sets a little below $\bfDelta^1_{n+1}$. Certainly all $\bfPi^1_n$ sets there are homogeneously Suslin (by \cite{projdet}), but we don't see how to improve either bound beyond this. We also show that in mice not too far above $0^\pistol$ and modelling $\ZFC$, all homogeneously Suslin sets are $\bfPi^1_1$.

We first define the phrase ``a little below''.

\begin{dfn}\index{correctly $\bfDelta^1_{n+1}$} Let $N$ be an inner model of $\ZFC$, which is $\bfSigma^1_{n+1}$-correct. Let $U\sub\RR\cross\RR$ be a standard universal $\bfSigma^1_{n+1}$ set, and for $z\in\RR$ let $U_z$ be the section of $U$ at $z$. Suppose $A\sub\RR$ is a $\bfDelta^1_{n+1}$ set in $N$. Then $A$ is ($N$)-\emph{correctly}-$\bfDelta^1_{n+1}$ iff there are $a,b\in\RR^N$ so that $A=U_a\int N$ and $U_a=\RR\cut U_b$.\end{dfn}

The definition is made in $V$, so $N$ might not know which sets are correctly-$\bfDelta^1_{n+1}$. But given $a,b\in\RR^N$, whether $(a,b)$ witnesses $U_a$ is correctly-$\bfDelta^1_{n+1}$ is a $\Pi^1_{n+2}(a,b)$ question. Since $M_n$ can compute $\Pi^1_{n+2}$ truth (of reals in $M_n$) by consulting the extender algebra (see \cite{outline}), we get: the class of $M_n$-correctly $\bfDelta^1_{n+1}$ sets is definable over $M_n$.

\begin{lem}\label{lem:correctbaire} Let $\delta_0$ be the least Woodin of $M_n$. The $M_n$-correctly $\bfDelta^1_{n+1}$ sets are precisely the the $\Col(\om,\delta_0)$-universally Baire sets of $M_n$.
\end{lem}

\begin{thm}\label{thm:homMn}
In $M_n$, every homogeneously Suslin set of reals is correctly $\bfDelta^1_{n+1}$; equivalently, they are $\Col(\om,\delta_0)$-universally Baire, where $\delta_0$ is the least Woodin cardinal.\end{thm}

\begin{cor}\label{cor:whomMn}
In $M_n$, the weakly homogeneously Suslin sets of reals are precisely the $\bfSigma^1_{n+1}$ sets.\end{cor}
\begin{proof}
One direction follows the theorem and the fact that weakly homogeneously Suslin sets are just projections of homogeneously Suslin sets (see \cite{stattower}). The other follows the Martin-Steel result of \cite{projdet}.\renewcommand{\qedsymbol}{$\Box$(Corollary \ref{cor:whomMn})}\end{proof}

We expect that throughout the interval of mice from $0^\pistol$ to $M_1$, the homogeneously Suslin sets become steadily more complex descriptive set theoretically, culminating in:

\begin{conj} In $M_1$, the homogeneously Suslin sets, the correctly $\bfDelta^1_2$ sets and the $\Col(\om,\delta_0)$ sets coincide.\end{conj}

Other related results have been known for some time. One was the following result of Woodin, which is discussed in \cite{kai}:

\begin{fact}[Woodin] Assume $\AD+\DC$, and suppose there exists a normal fine measure $\mu$ on $\pow_{\om_1}(\RR)$. For measure one many $\sigma\in\mu$, if $g$ is a generic enumeration of $\RR\int M_1(\sigma)$ in order-type $\om_1^{M_1(\sigma)}$, then the weakly homogeneously Suslin sets of $M_1(\sigma)[g]$ coincide with the $\bfSigma^1_2$ sets of $M_1(\sigma)[g]$.\end{fact}

The analogous statement about $M_n(\sigma)$ and $\bfSigma^1_{n+1}$ sets also holds. Lower in the mouse order, Schindler and Koepke generalized the fact about $L[U]$ in \cite{hsssim} in a couple of ways. They showed that either if $0^\longsword$ does not exist, or if $0^\pistol$ does not exist, $V=K$ and $K$ is below a $\mu$-measurable, then all homogeneously Suslin sets are $\bfPi^1_1$. In fact they make do with a notion weaker than homogeneously Suslin. ($0^\longsword$ and a $\mu$-measurable are both well below $0^\pistol$. See \ref{dfn:0pistol} for a definition of $0^\pistol$.)

We now move on the proof of \ref{lem:correctbaire}, and then \ref{thm:homMn}. However, the proof of the theorem shows independently that the homogeneously Suslin sets are correctly $\bfDelta^1_{n+1}$, then just quotes the lemma for universal Baireness. So the reader can skip the proof of the lemma if desired.

The following is established in (\cite{projwell}, 4.6).

\begin{fact}[Woodin]\label{fact:Mncorrect} Let $n\in\om$, $N$ be iterable and active, $\PP\in N|\alpha$, and suppose $N$ has $n$ Woodin cardinals above $\alpha$. Let $G$ be $\PP$-generic over $N$. Then $N[G]$ is $\bfSigma^1_{n+1}$-correct, and if $n$ is even, $N[G]$ is $\bfSigma^1_{n+2}$-correct.\qed\end{fact}

\begin{cor}\label{cor:neven} For $n$ even, $M_n\sats A \textrm{ is } \bfDelta^1_{n+1}$ iff $A$ is $M_n$-correctly $\bfDelta^1_{n+1}$.\end{cor}
\begin{proof} Apply \ref{fact:Mncorrect} to $M^\#_n$ and the $\Pi^1_{n+1}(z)$ statement ``$A$ is a $\Delta^1_{n+1}(z)$ set''.\\
\renewcommand{\qedsymbol}{$\Box$(Corollary \ref{cor:neven})}\end{proof}

\begin{rem} The proof we give for the ``$\implies$'' direction of the following lemma (in the $n>1$ case) was provided by Steel. Our original attempt for the $n>1$ case (which we sketch below) didn't work at all until Grigor Sargsyan pointed out the existence of \ref{fact:Mncorrect}, which makes it work for $n$ odd. Thanks to both for the guidance.\end{rem}

\begin{proof}[Proof of Lemma \ref{lem:correctbaire}]
Consider first $M_1$. If $A\in M_1$ is correctly-$\bfDelta^1_2$, use Shoenfield trees for $A$ and its complement to witness the $\Col(\om,\delta_0)$-universal Baireness of $A$ in $M_1$. Conversely, suppose
\[ M_1|\eta\sats\ZFmin\ +\ \Col(\om,\delta_0)\forces S,T\textrm{ project to complements}. \]
We may assume $S,T$ are definable points in $M_1|\eta$.
Let $M=\Hull_\om^{M_1|\eta}(\empty)$ and let $S^M,T^M$ be the collapses of $S,T$. Note that in iterating $M$, there is always a unique wellfounded branch (using 1-smallness and that $M=\Th(M)$). Now define $A$ by
\[ x\in A\ \ \iff\]
\[\textrm{there's a non-dropping normal iterate } P \textrm{ of } M \textrm{ such that } x\in p[i_{M,P}(S^M)]; \]
then
\[ x\in\RR\cut A\ \ \iff\]
\[\textrm{there's a non-dropping normal iterate } P \textrm{ of } M \textrm{ such that } x\in p[i_{M,P}(T^M)]. \]
This is standard: given $x\in\RR$, there is a non-dropping iterate of $M$ so that $x$ is generic over $\Col(\om,\delta_0)^P$ (see (\cite{outline}, \S7)), and $x\in p[i_{M,P}(S^M)]$ or $x\in p[i_{M,P}(T^M)]$ because these trees project to complements in $P[x]$. If there are iterates $P_1$ and $P_2$ so that $x\in p[i_{M,P_1}(S^M)]$ and $x\in p[i_{M,P_2}(T^M)]$, then let $P$ be the result of coiterating $P_1$ and $P_2$. Since $M=\Th(M)$, we get the same final model on both sides and $i_{P_1,P}\com i_{M,P_1}=i_{M,P}=i_{P_2,P}\com i_{M,P_2}$. But then $x\in p[i_{M,P}(S)]\int p[i_{M,P}(T)]$, so by absoluteness, there is some such $x'\in P$, a contradiction. Since the formulas above are $\Sigma^1_2(M)$, the lemma holds when $n=1$.

Now let $n>1$. To show all $\Col(\om,\delta_0)$-universally Baire sets are simple, the same proof as for $n=1$ works, except that the extra complexity of the iteration strategy for the hull $M$ leads only to a $\Delta^1_{n+1}(M)$ definition. (See \cite{projwell}, where the definability of the iteration strategy is discussed, or \ref{thm:homMn} for the $n=2$ case.)

For the converse, consider first $M_2$ and (correctly) $\bfDelta^1_3$. If $G$ is generic for $\Col(\om,\delta_0)$, $M_2[G]$ isn't $\Pi^1_3$-correct, but it can still compute $\Pi^1_3$ truth with its remaining extender algebra. This leads to a tree $T\in M_2$ which projects to $(\Pi^1_3)^V\int M_2[G]$. For this, let $\varphi(v_3)=\all v_1\ex v_2\psi(v_1,v_2,v_3)$ define a universal $\Pi^1_3$ set, with $\psi$ $\Pi^1_1$. Fix $\eta>\delta_1$ such that $M_2|\eta\sats\ZFmin$. Let $T$ be the tree building $(x,g,N,\pi)$, where $\pi:N\to M_2|\eta$ is elementary, $g$ is $N$-generic over $\pi^{-1}(\Col(\om,\delta_0))$, $x\in N[g]$, and
\[ N[g]\sats \textrm{The extender algebra at } \pi^{-1}(\delta_1) \textrm{ forces } \varphi(x). \]
Using genericity iterations, the reader can check that
\[ p[T]\int M_2[G] = \{ x\in \RR\int M_2[G]\ |\ \varphi(x)\}. \]
(Notice that the corresponding tree $S$ for $\Sigma^1_3$ doesn't work, since a $\Pi^1_2$ statement true in some $N[g]$ may be false in $V$.)

Now if $A\sub\RR$ in $M_2$ is (correctly) $\bfDelta^1_{3}$, then using $T$ and $\bfPi^1_3$ definitions for $A$ and its complement, it follows that $A$ is $\Col(\om,\delta_0)$-universally Baire in $M_2$.

In $M_3$, one obtains trees for $\Sigma^1_4$, etc. This completes the proof.
\renewcommand{\qedsymbol}{$\Box$(Lemma \ref{lem:correctbaire})}\end{proof}

\begin{rem}
We sketch another proof of the above lemma for $n$ odd. Work in $M_3$. Let $S_1$ be a homogeneous tree for $\Pi^1_2$ obtained as in \cite{projdet}, with completeness $\kappa$, $\delta_0<\kappa$. Let $S_2$ be the natural weakly homogeneous tree for $\Sigma^1_3$ obtained from $S_1$, and $S_3$ the corresponding Martin Solovay tree for $\Pi^1_3$, and $S_4$ the natural tree for $\Sigma^1_4$ obtained from $S_3$ (see \cite{stattower} for details). Then since the completeness of all measures used is above $\delta_0$, $S_4$ is still a tree for $\Sigma^1_4$ in $M_2[G]$, where $G$ collapses $\delta_0$. Given some $M_3$-correctly $\bfDelta^1_4$ definition, let $T_1$ and $T_2$ be the corresponding slices of $S_4$. The definition extends to a $\bfDelta^1_4$ set in $M_3[G]$, since $M_3[G]$ is $\Sigma^1_4$-correct, by \ref{fact:Mncorrect}. Therefore $T_1$ and $T_2$ project to complements in $M_3[G]$. This proof doesn't go through when $n$ is even because $M_n[G]$ isn't sufficiently correct. But notice the resulting tree projects to $(\Sigma^1_{n+1})^{M_n[G]}$ in $M_n[G]$, whereas the tree used in the proof of \ref{lem:correctbaire} projects to $(\Pi^1_{n+1})^V\int M_n[G]$. This gives another proof of the failure of $\Sigma^1_{n+1}$ correctness in $M_n[G]$: if these sets are complementary, then every $\bfSigma^1_{n+1}$ set is $\delta_0$-universally Baire in $M_n$, but then by \ref{lem:correctbaire}, they're $\bfDelta^1_{n+1}$; contradiction.
\end{rem}

\begin{proof}[Proof of Theorem \ref{thm:homMn}]\setcounter{clm}{0}
First consider $M_1$, where we now work. Let $\delta$ be (the) Woodin. Suppose $T\in M_1$ is a homogeneous tree, definable as a point in $N=M_1|\eta$, where $\eta$ is a double successor cardinal in $M_1$ above $\delta_0$. Let $M=\Hull^N(\empty)$. We will give a canonical representation of $p[T]$ as a correctly-$\Delta^1_2(M)$ set. This implies the theorem, since the least homogeneous tree whose projection is not correctly-$\bfDelta^1_2$ is definable.

Let $\pi:M\to N$ be elementary. $M$ is $(\omega_1+1)$-iterable since $N$ is. Moreover, because $M$ is $1$-small and pointwise definable, its strategy must always choose the unique wellfounded branch. Let $\left<\mu_s\right>_{s\in{^{<\om}\om}}\in\rg(\pi)$ be a homogeneity system for $S$. For $x\in{^\om\om}$,  let
\[ \Ubar_x=\dirlim_{m\leq n<\om}(\Ult(M,\mubar_{x\rest n}),i^M_{x\rest m,x\rest n}), \]
where bars denote preimages under $\pi$.

\begin{dfn}\index{$\Pi^1_2$-iterable}
For this proof, a countable premouse $P$ is $\Pi^1_2$\emph{-iterable} iff for every $\alpha,\Tt,x\in\RR$ such that $\alpha$ codes an ordinal and $\Tt$ codes an iteration tree on $U$, either (a) there is a $T$-maximal branch $b$, cofinal in $\lambda$, such that $M^\Tt_b$ and $M^\Tt_\gamma$, $\gamma<\lambda$, are wellfounded through $\alpha$; or (b) $\Tt$ has a final model, $x$ codes a normal one step extension $\Tt'$ of $\Tt$, and all models of $\Tt'$ are wellfounded through $\alpha$.
\end{dfn}

\begin{clm}\label{clm:p[T]def}
For $x\in{^\om}\om$, the following are equivalent:
\begin{itemize}
\item[(1)] $x\in p[T]$;
\item[(2)] There is an elementary $\psi:\Ubar_x\to M_1|\eta$;
\item[(3)] $\Ubar_x$ is $\Pi^1_2$-iterable;
\item[(4)] There is a countable $\Sigma$-iteration of $M$ to a model $Q$, and an elementary $\sigma:\Ubar_x\to Q$.
\end{itemize}
Moreover, the condition in (4) is $\Sigma^1_2(M)$.\end{clm}
\pagebreak
\begin{proof}\ \\
(1) $\implies$ (2).

Let $U_x$ be the (wellfounded) direct limit in the $S$-system, with base model $N$. $\Ubar_x\in U_x$, and clearly $\pi$ yields an elementary $\pi_x:\Ubar_x\to U_x$. Let $j:N\to U_x$ be the direct limit map. We can choose $\alpha<\delta$ such that $N|\alpha\elem N$ and $\rg(\pi_x)\sub U_x|j(\alpha)$. By absoluteness, there is a $\psi':\Ubar_x\to U_x|j(\alpha)$ with $\psi'\in U_x$. Since $j(\Ubar_x)=\Ubar_x$, there's a $\psi:\Ubar_x\to N|\alpha$.\\
\\
(2) $\implies$ (3).

Immediate.\\
\\
(3) $\implies$ (4).

As in \cite{projwell}, the $\Pi^1_2$-iterability of $\Ubar_x$ allows us to run a comparison with $M$: whilst $\Tt$ on $M$ chooses non-maximal branches, they provide $Q$-structures, which guide the branch choice for the tree $\Uu$ on $\Ubar_x$. If $\Tt$ chooses a maximal branch, then $M(\Tt)$ has the same theory as $M$ and $\Ubar_x$, so the $\Pi^1_2$-iterability also provides a maximal branch for $\Uu$. Note that if the comparison reaches such a maximal stage, then it has finished. The comparison cannot run for $\omega_1$ stages, since otherwise it can be replicated in $L[M,\Ubar_x]$, which has a smaller $\omega_1$. So it terminates successfully, and since $M=\Hull^M(\empty)$, the same final model is produced by both trees, with no dropping on the main branches.\\
\\
(4) $\implies$ (1).

Let $\Ttbar$ be the $\Sigma$-tree on $M$, with final model $\Qbar$. Let $\Tt=\pi\Ttbar$ be the copied tree on $N$ and $\pi_{\Qbar}$ be the final copy map. Let $j:N\to U_x$, $\jbar:M\to\Ubar_x$, $\pi_{\infty}:\Ubar_x\to U_x$ be the natural maps. Then (ignoring $\psi$) the following diagram commutes:
\[
\setlength{\unitlength}{1mm}
\begin{picture}(100,70)(0,20)
\put(40,25){$M$}
\put(20,40){$\Ubar_x$}
\put(60,40){$\Qbar$}
\put(40,65){$N$}
\put(20,80){$U_x$}
\put(60,80){$Q$}
\put(40,28){\vector(-4,3){14}}
\put(27,31){$\jbar$}
\put(45,28){\vector(4,3){14}}
\put(55,31){$i^{\Ttbar}$}
\put(26,41){\line(1,0){15.5}}
\put(43.5,41){\vector(1,0){16}}
\put(50,42){$\sigma$}

\put(40,68){\vector(-4,3){14}}
\put(27,71){$j$}
\put(45,68){\vector(4,3){14}}
\put(55,71){$i^\Tt$}
\put(26,81){\line(1,0){3}}
\put(31,81){\line(1,0){3}}
\put(36,81){\line(1,0){3}}
\put(41,81){\line(1,0){3}}
\put(46,81){\line(1,0){3}}
\put(51,81){\line(1,0){3}}
\put(56,81){\vector(1,0){3}}
\put(41.5,82.5){$\psi$}

\put(42.5,30){\vector(0,1){33}}
\put(43.5,49){$\pi$}
\put(22.5,45){\vector(0,1){33}}
\put(23.5,59){$\pi_\infty$}
\put(62.5,45){\vector(0,1){33}}
\put(63.5,59){$\pi_\Qbar$}
\end{picture}
\]
We are missing one edge in this triagonal prism, which we need to complete the proof. We want to define $\psi:U_x\to Q$ in the only commuting way:
\begin{equation}\label{eqn:psidef} \psi(j(f)(\pi_{\infty}(b)))=i^\Tt(f)(\pi_{\Qbar}\com\sigma(b)). \end{equation}
(All elements of $U_x$ are of the form $j(f)(\pi_{\infty}(b))$ since the generators $a'$ of $\mu_{x\rest n}$ are in the range of $\pi$, so $\pi_{\infty}$ hits $j_{n,\infty}(a')$.)

We need to see that this is well-defined and elementary. This requires certain measures derived from $j$ and $i^\Tt$ to be identical.\\
\\
\noindent\emph{Notation.} Let $\Ww$ be a normal iteration tree on a premouse $P$ with last model $R$, such that $i^\Ww$ exists. Let $x\in\OR^{<\om}$, and suppose $x\in(i^\Ww(\kappa))^{<\om}$, and that $\kappa$ is in fact least such. Then $\mu^\Ww_x$ denotes the measure on $\kappa^{|x|}$ derived from $i^\Ww$ with generator $x$.\\

We may assume that in (\ref{eqn:psidef}), $b=\jbar_{n,\infty}(a)$, where $a$ is the generator of $\mubar_{x\rest n}$. So $\pi(a)=a'$ where $a'$ is as above. Now since the bottom triangle commutes, $\mubar_{x\rest n}=\mu^{i^{\Ttbar}}_{\sigma(b)}$. Let $\Ttbar_{\sigma(b)}$ be the a finite support tree for $\sigma(b)$ derived from $\Ttbar$ and $\taubar:\Rbar\to\Qbar$ be the final copy map. Then since $\taubar\com i^{\Ttbar_{\sigma(b)}}=i^{\Ttbar}$, $\mubar_{x\rest n}=\mu^{\Ttbar_{\sigma(b)}}_{\taubar^{-1}(\sigma(b))}$. Since $\Ttbar_{\sigma(b)}$ is a finite tree on $M$, it is in $M$. Moreover, $\pi(\Ttbar_{\sigma(b)})=\pi\Ttbar_{\sigma(b)}$ (the copied tree). Also, it's easy\footnote{Footnote added January 2013: See \cite{hsstm} for the generalization of \ref{lem:supptree} covering this.} to see that $\pi\Ttbar_{\sigma(b)}$ is a finite support tree for $\pi_{\Qbar}(\sigma(b))$ derived from $\Tt$. (Whilst extracting a support for $\sigma(b)$ from $\Ttbar$, maintain inductively that the copy maps $\pi_{\alpha}$ lift it to a support for $\pi_{\Qbar}(\sigma(b))$; then the derived $\Tt_{\pi_{\Qbar}(\sigma(b))}$ is just $\pi\Ttbar_{\sigma(b)}$.\footnote{\label{ftn:correction}Footnote added January 2013: Correction: Here and in the equation to follow, ``$\Tt_{\pi_{\Qbar}(\sigma(b))}$'' should be replaced by ``$\Tt^A$'', where $A$ is the support for $\pi_{\Qbar}(\sigma(b))$ given by lifting the support $\bar{A}$ in $\Ttbar$ for $\sigma(b)$, up to $\Tt$ with the copy maps. The tree $\Tt_{\pi_{\Qbar}(\sigma(b))}$ is that defined in \ref{dfn:supptree}, which depends on precisely what minimization process is used in the algorithm described after \ref{dfn:supp} for building the support $A^*$. It seems that possibly $A\neq A^*$ and $\Tt^A\neq\Tt^{A^*}$. The fact that $\Tt^A=\pi\Ttbar_{\sigma(b)}$ is proved in \cite{hsstm} (for the notion of support used there; see footnote \ref{ftn:type3} on page \pageref{ftn:type3}).}) Therefore\footnote{Footnote added January 2013: See footnote \ref{ftn:correction}.}
\[ \mu^{\Tt}_{\pi_{\Qbar}\com\sigma(b)}=\mu^{\Tt_{\pi_{\Qbar}(\sigma(b))}}_{\tau^{-1}(\pi_{\Qbar}(\sigma(b)))}=\pi(\mu^{\Ttbar_{\sigma(b)}}_{\taubar^{-1}(\sigma(b))})=\pi(\mubar_{x\rest n})=\mu_{x\rest n}. \]
\renewcommand{\qedsymbol}{$\Box$(Claim \ref{clm:p[T]def})}\end{proof}

The definability of (4) follows since a $\Sigma$-iterate of $M$ is one which chooses wellfounded branches. This completes the proof of \ref{thm:homMn} for $M_1$.\\

Now consider $M_2$. The argument is almost as for $M_1$, with appropriate modifications to the conditions (1) - (4). The difference lies in the increased complexity of the iteration strategies for the hull $M$ and the ultrapowers $\Ubar_x$. Instead of $\Pi^1_2$-iterability, we need $\Pi^1_3$-iterability, which we presently define. This notion is also taken from \cite{projwell}.

Consider a $2$-small, $\om$-sound mouse $P$ projecting to $\om$. Its unique strategy, having built a limit length tree $\Tt$, must choose the unique branch $b$ such that $Q(b,\Tt)$ is $\Pi^1_2$-iterable above $\delta(\Tt)$. As in the proof of (3) $\implies$ (4) above, such a $Q(b,\Tt)$ can be compared to the Q-structure of the correct branch, so standard arguments show they're identical, and since $P$ is sound, therefore so are the branches. This implies the statement ``$\Tt$ is a correct normal tree on $P$'' is $\Pi^1_2(P)$. Moreover, the correct branch is $\Delta^1_3$ in any real coding $\Tt$. Assuming $\bfDelta^1_2$-determinacy (true in $M_2$), $\Pi^1_3(x)$ is closed under ``$\ex b\in\Delta^1_3(x)$'' (see (\cite{mosch}, 4D.3, 6B.1, 6B.2)). This leads to the following formulation:

\begin{dfn}\label{dfn:Pi^1_3it}\index{$\Pi^1_3$-iterable}
Assume $\bfDelta^1_2$-determinacy. Let $P$ be a $2$-small, $\om$-sound premouse $P$ projecting to $\om$. $P$ is $\Pi^1_3$\emph{-iterable} iff for each countable normal tree $\Tt$ on $P$, either $\Tt$ has a last model and every one-step normal extension produces a wellfounded next model, or there is a maximal branch $b$ of $\Tt$ in $\Delta^1_3(\Tt)$ such that $Q(b,\Tt)$ is $\Pi^1_2$-iterable above $\delta(\Tt)$.
\end{dfn}

With reals coding the elements of $\HC$ in this definition, the determinacy implies $\Pi^1_3$-iterability is indeed a $\Pi^1_3$-condition.

Now in $M_2$, the hull $M$ is defined as before (in particular, it embeds in a level $M_2|\eta$ containing all Woodin cardinals). Although $M$ doesn't literally project to $\om$, by \ref{lem:esccond} we can instead work with $\J_1(M)$, which is also sound by the same lemma. Conditions (1), (2) and (4) are as in the $n=1$ case (with $M_2$ replacing $M_1$). For (4) though, we must define $\Sigma$. Since $M_2$ satsisfies ``I'm $\delta_0$-iterable'', its unique $\om_1+1$-iteration strategy for $M$ is the pullback of its strategy for $M_2|\eta$. $\Sigma$ denotes this strategy for $M$. The discussion above shows that condition (4) is then $\Sigma^1_3(M)$.

Condition (3) becomes ``$\Ubar_x$ is $\Pi^1_3$-iterable''. However, $\J_1(\Ubar_x)$ \emph{isn't} sound in general, so we need to check that $\Pi^1_3$-iterability works in this context. (John Steel pointed out that it does in fact work, thereby simplifying our original argument, which instead used ``$\Pi^1_3$-$M$-comparability''.) Assume that $\Ubar_x$ is fully iterable, via some $\Gamma$, and $\Tt$ is a normal tree of limit length on $\J_1(\Ubar_x)$, with $\Tt\conc b$ via $\Gamma$. At issue is the definability of $b$ from $\Tt$; we need to see that $b$ is the unique $b'$ such that $Q(b',\Tt)$ is $\Pi^1_2$-iterable above $\delta(\Tt)$. The reader can check that things work as in the sound case unless $b$ does not drop, and $i^\Tt_b(\delta_0^x)=\delta(\Tt)$, where $\delta_0^x$ is the least Woodin of $\Ubar_x$. In this case, $Q(b,\Tt)=M^\Tt_b$, and is $\om_1+1$-iterable above $\delta(\Tt)$. Suppose $c\neq b$ and $Q(c,\Tt)$ is $\Pi^1_2$-iterable above $\delta(\Tt)$. As in the sound case, $Q(b,\Tt)$ and $Q(c,\Tt)$ can be successfully compared, and since they're Q-structures for $M(\Tt)$, they iterate to the same model $Q$, with no dropping, and there was no drop along $c$. Let $j:\J_1(M)\to Q$ be the canonical embedding, which is continuous at $\delta_0^M$ since it's composed of ultrapower embeddings of degree $0$. Since $M$ is pointwise definable, cofinally many points below $\delta(\Tt)$ are definable in $Q|j(\OR^M)$. But these points are included in $\rg(i^\Tt_b)\int\rg(i^\Tt_c)$, a contradiction (as in (\cite{outline}, \S6)).

Using similar arguments, one can show that if $\Ubar_x$ is $\Pi^1_3$-iterable, then $\Ubar_x$ is embeddable in a correct iterate of $M$. (One must compare 3 Q-structures simultaneously to see that during such comparison, the branches chosen by $\Pi^1_3$-iterability are always cofinal in the tree on $\Ubar_x$.) Moreover, the comparison can be executed in $M_2$. For if the comparison ran through $\om_1^{M_2}$ stages, then it could be replicated in $M_1(M,\Ubar_x)$, using the extender algebra of that model to compute the correct branches. But then it runs through $\om_1^{M_1(M,\Ubar_x)}+1$ many stages there, a contradiction. This last statement also holds in $M_2$, since $\RR\int M_2$ is closed under the $M^\#_1$ operator.

The rest of the argument inside $M_2$ is as in the $n=1$ case. By \ref{cor:neven}, the resulting definition for the homogeneously Suslin set is in fact correctly-$\bfDelta^1_3$. This finishes the $M_2$ case.\\

Now consider $M_3$. Things work basically as for $M_2$; however the definition of $\Pi^1_4$-iterability has to differ from that of $\Pi^1_3$ because $\Pi^1_4$ isn't normed. The reader should consult \cite{projwell} for the elegant solution. Otherwise the only thing to check is that the $\Delta^1_4(M)$ definition obtained in $M_3$ extends to one over $V$. Since $M_3$ is only $\Sigma^1_4$-correct, this isn't as immediate as for $M_2$. However, defining $\Pi^1_4$-iterability requires only $\bfDelta^1_2$-determinacy, and the resulting closure of $\Sigma^1_3$ under $\all y\in\Delta^1_3(x)$. The statement ``$\bfDelta^1_2$ determinacy holds'' is $\Pi^1_4$, so its truth in $M_3$ implies it in $V$. Investigating the proofs of 4B.3 and 6B.1 of \cite{mosch}, one sees that $M_3$ and $V$ therefore agree about the definition of the resulting norm for $\Pi^1_3$, and in turn, the $\Pi^1_3$ definition of the quantifier $\ex y\in\Delta^1_3(x)$. So the $\Sigma^1_4(M)$ and $\Pi^1_4(M)$ formulae defining the homogeneously Suslin set in $M_3$ have the same interpretation (in terms of iterability and correct iterates) in $V$. So the argument that (3) is equivalent to (4) also works in $V$, as desired. $M_4$ and beyond involve no new ideas.
\renewcommand{\qedsymbol}{$\Box$(Theorem \ref{thm:homMn})}\end{proof}

Finally, for models of $\ZFC$ in the region of $0^\pistol$ or below, we do get an exact characterization of the homogeneously Suslin sets.

\begin{dfn}\label{dfn:0pistol}\index{$0^\pistol$} $0^\pistol$ is the least active mouse $N$ such that $N|\crit(F^N)$ satisfies ``there is a strong cardinal''.\end{dfn}

\begin{thm}\label{thm:0pistol} Let $N\sats\ZFC$ be an $\fully$-iterable mouse satisfying ``if $\mu<\kappa$ are measurables, then $\mu$ is not strong to $\kappa$''. Then in $N$, all homogeneously Suslin sets are $\bfPi^1_1$. In particular, this holds if $N\sats\ZFC$ and is below $0^\pistol$.\end{thm}
\begin{proof}
This is a corollary of the proof of \ref{thm:homMn} and that below $0^\pistol$, every mouse is an iterate of its core, probably due to Jensen. We give a proof of the latter in our setting. With notation as in the proof of \ref{thm:homMn}, consider the comparison of $M$ with $\Ubar_z$, for some iterable $U=\Ubar_z$, with last model $Q$. We claim $U=Q$. Otherwise let $\kappa=\crit(i_{U,Q})$. Then $\kappa$ is measurable in $U$, so no $\mu<\kappa$ is strong to $\kappa$ in $U$, or therefore in $Q$. Therefore $\kappa$ cannot be overlapped by any extender used in $\Tt$. Now there's $E=E^\Tt_\alpha$ used on $\Tt$'s main branch with $\crit(E)=\kappa$, since $\kappa\notin\Hull^Q_\om(\kappa)$. But $U$, $M^\Tt_\alpha$ and $Q$ agree about $\pow(\kappa)$, and $\pow(\kappa)\int Q\sub\Hull^Q(\kappa)$, which leads to $E$ being compatible with the first extender used on the $U$ side.

So $\Ubar_z$ is iterable iff $\Ubar_z$ is a correct iterate of $M$, iff it is a wellfounded iterate of $M$, which is $\Pi^1_1(M)$.
\renewcommand{\qedsymbol}{$\Box$(Theorem \ref{thm:0pistol})}\end{proof}

\pagebreak
\section{The Copying Construction \& Freely Dropping Iterations}\label{sec:copying}
Here we identify and solve some problems with the copying construction of \cite{fsit}. Actually, since it doesn't take much more work, we prove a generalization of ``every initial segment of a mouse is a mouse'', since this is really needed for the iterability of the phalanges used in the previous sections.

First we'll briefly discuss the problems ignored in \cite{fsit}, and give examples where these arise.

Suppose $M$ is a type 3 premouse and $\pi:M^\sq\to N^\sq$ is a lifting map being used during a copying construction. It might be that the exit extender $E$ from $M$ has $\nu_M<\lh(E)<\OR^M$, so $E\notin\dom(\pi)$, but $E$ is not the active extender of $M$. \cite{fsit} ignores this. Let $\psi:\Ult(M^\sq,F^M)\to\Ult(N^\sq,F^N)$ be the canonical embedding. If one has $\psi(\nu_M)\leq\nu_N$, the natural solution is to let $\psi(E)$ be the exit extender from $N$; otherwise one can first use $F^N$ in the upper tree, then use $\psi(E)$.

Another problem arises if $\psi(\nu_M)<\nu_N$ and $F^M$ is the exit extender from $M$. Here \cite{fsit} uses $F^N$ as the exit extender from $N$. But we have $\psi(\lh(F^M))<\nu_N$, so the next extender $E'$ used below might be such that $\psi(\lh(E'))<\nu_N$. This causes a break in the increasing length condition of the upper tree.

We now give an example of these two situations. Suppose $\nu_M=[a,f]^M_{F^M}$; then $\psi(\nu_M)=[\pi(a),\pi(f)]^N_{F^N}$. The statement ``$[a,f]\geq\nu$'' is $\Pi_1$; ``$[a,f]\leq\nu$'' is $\Pi_2$. So if $\pi$ is $\Pi_2$-elementary then $\psi(\nu_M)=\nu_N$ but it seems it might be that
\begin{itemize}
 \item[(a)] $\pi$ is a $0$-embedding and $\psi(\nu_M)>\nu_N$; or
 \item[(b)] $\pi$ is a weak $0$-embedding and $\psi(\nu_M)<\nu_N$.
\end{itemize}
Now suppose $\kappa=\cof^M(\nu_M)<\rho^M_1$, $\kappa$ is measurable in $M$, and $M$ is $1$-sound. Let $E$ be an extender over $M$ with crit $\kappa$. Let $i_0:M^\sq\to\Ult_0(M^\sq,E)$, $i_1:M^\sq\to\Ult_1(M^\sq,E)$, and $\tau:\Ult_0\to\Ult_1$ be the canonical maps. Then $i_0$ is cofinal. Let $f:\kappa\to\nu_M\in M$ be cofinal, strictly increasing and continuous. Let $f=[a,g_f]^M_{F^M}$. Then $\nu_M=\sup\rg([a,g_f])$ in $\Ult(M,F^M)$. $[i_0(a),i_0(g_f)]$ represents a strictly increasing, continuous function in $\Ult(\Ult_0,F^{\Ult_0})$, with domain $i_0(\kappa)$; denote this by $i_0(f)$. Now $i_0$ is cofinal in $\nu_{\Ult_0}$, and for any $\alpha<\kappa$, $i_0(f(\alpha))=i_0(f)(\alpha)$, so $i_0(f)(\kappa)=\nu_{\Ult_0}$. Therefore $\sup\rg(i_0(f))>\nu_{\Ult_0}$. Since $i_0$ is a $0$-embedding, it is an example of (a).

Now $i_1$ is a $1$-embedding, so is $\Pi_2$-elementary, so $\sup\rg(i_1(f))=\nu_{\Ult_1}$ (where $i_1(f)$ is defined as for $i_0(f)$). Therefore $\tau(\nu_{\Ult_0})<\nu_{\Ult_1}$. It's easy to check that $\tau$ is a weak $0$-embedding, so it is an example of (b).

Now we consider another problem with the copying construction. Suppose $\pi:M\to N$ is a weak $k+1$-embedding, and $M$ and $N$ have degree $k+1$ in some iteration trees. If $E$ is applied to $M$ with $\crit(E)=\kappa$, and $\rho^M_{k+1}\leq\kappa$, but $\pi(\kappa)<\rho^M_{k+1}$, then $E$ triggers a drop in degree in the lower tree, but the lifted extender, with crit $\pi(\kappa)$, should not cause a drop if the upper tree is to be normal. So we are forced to let the degrees of corresponding ultrapowers differ between trees; we will see this works fine.

It seems these situations can arise in the proofs of condensation and solidity of the standard parameter (see (\cite{outline}, \S5) and \cite{fsit}). In the case of condensation, suppose $\sigma:H\to M$ is the (fully elementary) embedding under consideration, with $\crit(\sigma)=\rho=\rho^H_\om=(\kappa^+)^H$. In the proof, a tree $\Tt$ on the phalanx $(M,H,\rho)$ is lifted to a tree $\Uu$ on $M$, using $\sigma$ and $\id$ as initial lifting maps. $\lh(E^\Tt_0)>\rho$, so $\lh(E^\Uu_0)>(\kappa^+)^M$. Say $E=E^\Tt_\alpha$ has crit $\kappa$. Then $E$ measures exactly $\pow(\kappa)^H$, but goes back to $M$, so the tree drops to the least $M|\xi\supseg M|\rho$ projecting to $\kappa$. However, $E^\Uu_\alpha$ measures all of $M$, so $\Uu$ does not experience a drop. We can naturally set $\pi_{\alpha+1}:M^\Tt_{\alpha+1}\to i^\Uu_{M,\alpha+1}(M|\xi)$; $\pi_{\alpha+1}$ will be a weak $k$-embedding, where $\deg^\Tt(\alpha+1)=k$. Now it might be that $M|\xi$ is type 3, with
\[ \rho_{k+1}=\rho_{k+1}^{M|\xi}\leq\kappa<\rho_k^M=\rho_k\ \ \&\ \ \cof^{M|\xi}(\rho_k)=\kappa. \]
Then similarly to the earlier example, $i^\Tt_{M|\xi,\alpha+1}$ is discontinuous at $\rho_k^{M|\xi}$ and
\[ \pi_{\alpha+1}(\rho_k^{M^\Tt_{\alpha+1}})<\rho_k^{i^\Uu_{M,\alpha+1}(M|\xi)}. \]
Suppose the exit extender $F$ from $M^\Tt_{\alpha+1}$ has \[ \rho^{M^\Tt_{\alpha+1}}_k\leq\crit(F)<i^\Tt_{M|\xi,\alpha+1}(\rho^{M|\xi}_k), \] and measures all of $M^\Tt_{\alpha+1}$. Then $F$ applies normally to $M^\Tt_{\alpha+1}$, dropping to degree $k-1$, but
\[ \rho_{k+1}^{i^\Uu_{M,\alpha+1}(M|\xi)} = i^\Uu_{M,\alpha+1}(\kappa)<\crit(\pi_{\alpha+1}(F))<\rho_k^{i^\Uu_{M,\alpha+1}(M|\xi)}. \]
It's also easy to adapt this to give the situation with type 3 premice described earlier.
\\

Before proceeding to give details of a copying construction dealing with the above problems, we show that for simple enough copying none of the above problems occur, so things work exactly as described in \cite{fsit}. (As seen above though, copying isn't as smooth in general for lifting iterations on phalanges.)

\begin{thm}\label{thm:nicecopy}
Suppose $\pi:M\to N$ is a near $k$-embedding such that if $M$ is type 3, then $\pi$ preserves representation of $\nu$; i.e. $\psi_\pi(\nu_M)=\nu_N$. Then a normal tree on $M$ lifts to a normal tree on $N$ by the prescription in \cite{fsit}. Moreover, for each $\alpha$, $\pi_\alpha$ is a near $\deg^\Tt(\alpha)$-embedding, and if $M^\Tt_\alpha$ is type 3, $\pi_\alpha$ preserves representation of $\nu_{M_\alpha}$.
\end{thm}

\begin{proof}
As in \cite{near}, the nearness of embeddings is maintained inductively. This immediately knocks out the problems of differing degrees and possibility (b) above. (Note ``$\alpha<\rho_k$'' is $\Sigma_{k+1}$, so if $\pi:M\to N$ is $\Sigma_{k+1}$-elementary then $\pi(\rho^M_k)\geq\rho^N_k$.) Since ``$\nu$-preservation'' is preserved when taking ultrapowers of degree $\geq 1$, we consider only degree $0$.

Suppose $M$ is type 3, $\kappa<\nu_M$, and $E$ is an extender over $M$ with $\crit(E)=\kappa$. Say $\pi:M^\sq\to N^\sq$ is our copy map, and $E$ is lifted to $F$ over $N$. Let $\tau:\Ult_0(M^\sq,E)\to\Ult_0(N^\sq,F)$ be as usual. We need that $\tau$ preserves $\nu$-representation.

In the case that $\nu_M$ is singular in $M$, this is straightforward to show using the approach of the earlier example (which lead to examples of (a) and (b) above); we leave it to the reader.

Suppose $\nu_M$ is regular in $M$. Here we simply show that $i_E$ preserves $\nu$-representation; this suffices. Fixing $h:\mu^M\to\mu^M\in M$ and $\gamma\in M$, notice there are only boundedly many $\alpha<\nu_M$ of the form $i_{F^M}(h)(c)$ for $c\in[\gamma]^{<\om}$. Let $\delta_{h,\gamma}$ be the supremum of these ordinals. The same applies to a $\mu^M$-sequence of functions in $M$. Let $[a,f]=\nu_M$, $U^\sq=\Ult_0(M^\sq,E)$ and $b,g\in U^\sq$; suppose $[b,g]^U_{F^U}<[i(a),i(f)]^U_{F^U}$. Let $\gamma\in M$ be such that $b\in [i_E(\gamma)]^{<\om}$ and $a\sub\gamma$.

If $\mu^M<\kappa$, then we may assume $i_E(g)=g\in M$. Let $\delta=\delta_{g,\gamma}$. Let $F'=F\rest\delta$. Then $F'$ encodes the bounding property of $\delta$:
\[ \all u\in[\gamma]^{<\om}\ ([u,g]<[a,f]\ \implies\ [u,g]<\delta). \]
Clearly this is preserved by $i_E$, so $[b,g]<i_E(\delta)<\nu_{U}$.

If $\kappa\leq\mu^M$, let $g=i_E(G)(u)$, where $G\in M$ produces only functions $\mu^M\to\mu^M$. Then using a bound $\delta_{G,\gamma}$, the same argument works.
\renewcommand{\qedsymbol}{$\Box$(Theorem \ref{thm:nicecopy})}\end{proof}

\begin{dfn}\index{free iteration game}
Let the \emph{free (earliest model)}
$(n,\theta,\lambda)$ iteration game be as follows. The players build
a stack of $\lambda$ iteration trees, one in each round, with II
choosing wellfounded branches and creating wellfounded limits at all
limit stages. If in round $\alpha$, I is to play, and $\Tt$ is the
current tree, with models $M_\gamma$ and exit extenders $E_\gamma$,
$\Tt$ with length $\delta+1$, he
\begin{itemize}
\item May move to the next round (and must if the present tree has
length $\theta$); in this case he also chooses an initial segment of
$M_\delta$ and a degree $n\leq\omega$ for the root of the next tree,
where $n\leq\deg(\delta)$ if $M_\delta$ is chosen;
\item Must choose an extender $E$ from the present model with length
greater than all those already played in the current round;
\item Must choose some $\beta\leq\alpha$, such that
$\crit(E)<\nu^\Tt_\beta$ (see the remark below for a clarification); for earliest model iterations, he must choose the least such $\beta$;
\item Must choose some initial segment $M^*=(M^*)_{\delta+1}$ of
$M_\beta$ with $\crit(E)<\OR^{M^*}$ and $\pow(\crit(E))\int M^*$ measured by $E$;
\item Must choose $\deg(\delta+1)=n\leq\omega$ for the ultrapower,
such that $\crit(E)<\rho^{M^*}_n$ (clarification below) and $(M^*=M_\beta\ \implies\ n\leq\deg(\beta))$.
\end{itemize}
Payoff is as expected.
\end{dfn}

\begin{rem}\index{anomalous} 
We allow $M^*$ to be type 3, with $\crit(E)=\nu_{M^*}$. In this case $\deg(\delta+1)=0$, and $\Ult_0(M^*,E)$ is formed without squashing $M^*$, as in \S\ref{sec:meas}, Submeasures. If $\delta+1\leq_\Tt\gamma$ and $i^\Tt_{\delta+1,\gamma}$ exists, then $\gamma$ is called \emph{anomalous}. If $E^\Tt_\gamma$ is the active extender of $M^\Tt_\gamma$, we set $\nu^\Tt_\gamma=i^\Tt_{M^*,\gamma}(\crit(E))$, the largest cardinal of $M^\Tt_\gamma$. $\nu=\nu_{E^\Tt_\gamma}$ may be less than $\nu^\Tt_\gamma$ here; in fact $\nu=\sup_{\alpha<\gamma}\nu_{E^\Tt_\alpha}$. $M^\Tt_\gamma$ isn't a premouse: if $\gamma>\delta+1$ then $M^\Tt_\gamma$ doesn't satisfy the initial segment condition; if $\gamma=\delta+1$ then $\nu$ is less than $M^\Tt_\gamma$'s largest cardinal. But the iteration tree still makes sense. We'll also call $M^\Tt_\gamma$ an \emph{anomalous} structure, and let $\nu_{M^\Tt_\gamma}$ denote $\nu$. When dealing with definability over an anomalous structure, there is no squashing.
\end{rem}

\begin{dfn} The maximal (earliest model) iteration game is the game
with rules as above, except that player I must always choose
$(M^*)_{\delta+1}$, then $\deg(\delta+1)$, as large as possible.
\end{dfn}

\begin{dfn}\index{weak $m$-embedding} The definition of weak $m$-embedding is given on page 52 of \cite{fsit}, except that the condition ``$\rho_m\in X$'' shouldn't be there. Here we will also call a $\pi:M\to N$ an \emph{(anomalous) weak $0$-embedding} when $M$ is an anomalous structure, $N$ is a type 3 premouse, $\pi$ is $\Sigma_0$-elementary, there's a cofinal subset of $\OR^M$ on which $\pi$ is $\Sigma_1$-elementary, and $\pi(\nu_M)=\nu_N$. (As stated above, we're not squashing, and $\nu_M$ is the largest cardinal of $M$.)\end{dfn}

\begin{thm}\label{thm:freeiter} Suppose $N$ is maximally (earliest model)
$(n,\theta,\lambda)$-iterable, and that $\pi:M\to N|\eta$ is a weak
$m$-embedding, where $m\leq n$ if $\eta=\OR^N$. Then $M$ is freely
(earliest model) $(m,\theta,\lambda)$-iterable.\end{thm}

\begin{rem} We leave to the reader the second half of the copying we require for the applications of earlier sections. That is the copying of an iteration on some phalanx to a freely-dropping one on $N$. It's not much different to the construction of \ref{thm:freeiter}.\end{rem}

\begin{proof}[Proof of Theorem \ref{thm:freeiter}]\setcounter{case}{0}
The idea is just to reduce a free iteration on $M$ to a normal one on $N$ via a
copying construction. There are more details than usual - but mostly
due to patches for the usual copying construction.

Given a (possibly anomalous) weak $0$-embedding $\pi:M\to N$ between active premice (or their squashes in the non-anomalous type 3 case), let
\[ \psi_\pi:\Ult(M|((\mu^M)^+)^M,F^M)\to\Ult(N|((\mu^N)^+)^N,F^N) \]
be the canonical embedding. Note $\pi\sub\psi_\pi$, and when $M$ is type 1 or 2, $\psi_\pi(\lh(F^M))=\lh(F^N)$.

Let's assume that $\lambda=1$. As $\Tt$ is being built on $M$, with
models $M_\alpha$ and extenders $E_\alpha$, we build $\Uu$ on
$N_0=N$, with models $N_\alpha$ and extenders $F_\alpha$. One new
detail is that $\Uu$ may have nodes in its tree order not corresponding to extenders used in $\Tt$. For bookkeeping purposes, $\Tt$ and $\Uu$ will in general be padded. If $E_\beta=\empty$, we'll have $M_\beta$ is type 3, and we'll set
\[ M_{\beta+1}=M_\beta||\OR^{M_\beta}, \]
the reduct of $M_\beta$. We'll also set $\nu^\Tt_\beta=\nu_{M_\beta}$ and have $\nu^\Tt_\beta\leq\nu^\Tt_{\beta+1}<\lh(E_{\beta+1})$. So an extender with crit in $[\nu_{M_\beta},\nu^\Tt_{\beta+1})$ applies to a proper segment of $M_{\beta+1}$, since $M_{\beta+1}$ is passive, and $\nu_{M_\beta+1}$ is its largest cardinal. Thus we will never take an ultrapower of the entire $M_{\beta+1}$. (So $\deg^\Tt(\beta+1)$ isn't relevant; we can set $\Tt\pred(\beta+1)=\beta$.)

For any $\alpha$, there'll be an ordinal $\eta_\alpha$ and a (possibly anomalous) weak $\deg^\Tt(\alpha)$-embedding
\[ \pi_\alpha:M_\alpha\to N_\alpha|\eta_\alpha. \]
If $M_\alpha$ is anomalous, we'll have that $\eta_\alpha<\OR^{N_\alpha}$. If $M_\alpha$ is active let
$\psi_\alpha=\psi_{\pi_\alpha}$; otherwise let $\psi_\alpha=\pi_\alpha$. When $E^\Tt_\alpha\neq\empty$, we'll have a structure $R_\alpha$ and a $\Sigma_0$-elementary
\[ \psi_\alpha\rest M_\alpha|\lh(E_\alpha):M_\alpha|\lh(E_\alpha)\to R_\alpha. \]
Here when $E_\alpha$ is type 3, the map is to literally apply to $M_\alpha|\lh(E_\alpha)$, not its squash. Usually $R_\alpha=N_\alpha|\lh(F_\alpha)$. In any case, $\psi_\alpha$ provides a map lifting $E_\alpha$ to (possibly a trivial extension of) $F_\alpha$, sufficient for the proof of the shift lemma.

For $\beta<\alpha$, we'll maintain:
\begin{equation}\label{ag:pis}
\psi_\beta\rest(\lh(E_\beta)+1)\sub\pi_\alpha.
\end{equation}
\begin{equation}\label{ag:nu}
\psi_\beta(\nu^\Tt_\beta)\geq\nu^\Uu_\beta.
\end{equation}
\begin{equation}\label{ag:<nu}
\psi_\beta``\nu^\Tt_\beta\sub\nu^\Uu_\beta.
\end{equation}
And when $E_\beta\neq\empty$,
\begin{equation}\label{ag:lh}
\psi_\beta(\lh(E_\beta))\geq\lh(F_\beta).
\end{equation}

We start at $\pi_0=\pi$. Suppose we have everything up to stage $\alpha$. Say player I chooses an exit extender $E$ from $M_\alpha$. Let $E_\alpha=E$ except for in the last case below, in which $E_\alpha=\empty$. We need to define $F_\alpha$ (and in the last case, $F_{\alpha+1}$). Let $M=M_\alpha$, $N=N_\alpha|\eta_\alpha$, $\pi=\pi_\alpha$ and $\psi=\psi_\alpha$.

\begin{case} $E\in\dom(\pi)$ or $E$ is the active type 1 or 2 extender of $M$.\end{case}

Here $F_\alpha$ is defined as usual: If $\lh(E)\in\dom(\pi)$, let $F_\alpha=\pi(E)$. Otherwise let $F_\alpha$ be the active extender of $N$. By the preservation hypotheses (\ref{ag:nu}), we then have $\lh(F_\alpha)$ is larger than previous extenders on $\Uu$. Also $R_\alpha=N|\lh(F_\alpha)$.

\begin{case} $E$ is the active type 3 extender of $M$ and $\psi(\nu_M)\geq\nu_N$, or $E$ is the active extender of an anomalous $M$.\end{case}

Let $F_\alpha=F^N$ and
\[ R_\alpha = (N|(\psi(\nu_M)^+)^{\Ult(N,F^N)},(F^N)^*), \]
where $(F^N)^*$ is the extender of length $(\psi(\nu_M)^+)^{\Ult(N,F^N)}$ derived from $i_{F^N}$, coded amenably. If $\psi(\nu_M)=\nu_N$, $R_\alpha$ is just $N$. Otherwise $R_\alpha$ isn't a premouse, but we still get
\[ \psi\rest M:M\to R_\alpha \]
is $\Sigma_0$-elementary. This is sufficient for the proof of the shift lemma. (This may seem unnecessary, since we already had all the generators of $E$ within the domain of $\pi$. But dropping to a level below $((\mu^M)^+)^M$, and applying $E\rest\nu_E$, yields a smaller ultrapower than applying $E$; in particular, the smaller ultrapower does not agree with $M$ below $\lh(E)$.) Note that when $M$ is anomalous, our assumptions on $\pi$ give $\psi(\nu^\Tt_\alpha)=\nu^\Uu_\alpha$.

\begin{case}\label{case:Fdef.tp3lownuimage} $M$ is a premouse, $E$ is its active type 3 extender, and $\psi(\nu_M)<\nu_N$.\end{case}

Here we can't set $F_\alpha=F^N$, as discussed earlier. We use a proper segment of $F^N$ instead.
$F^N\rest\sup\pi``\nu_M$ may not be on $N$'s sequence,
but it is reasonable to set
\[ F_\alpha=\trivcom(F^N \rest \psi(\nu_M)). \]
This segment is in fact on $N$'s sequence, since $\psi(\nu_M)$ is a cardinal of $N$. Clearly $\psi$ factors through
\[ \psi':\Ult(M|(\mu^M)^+,F^M)\to\Ult(N|(\mu^N)^+,F_\alpha) \]
and the canonical map $\Ult(N,F_\alpha)\to\Ult(N,F^N)$ has crit $\lh(F_\alpha)$. It follows that setting $R_\alpha=N|\lh(F_\alpha)$ works.

In either of the previous cases, if $\beta<\alpha$, then $\lh(E_\beta)<\OR^{M^\sq}=\nu_M$, so the
agreement hypotheses give $\lh(F_\beta)<\sup\pi``\nu_M<\lh(F_\alpha)$, maintaining the increasing length condition of $\Uu$.

\begin{case} $M$ is a type 3 premouse, $\OR^{M^\sq} < \lh(E) < \OR^M$ and $\psi_\pi(\nu_M)\leq\nu_N$.\end{case}

Here $\psi_\pi(\OR^M)\leq\OR^N$. Let $E_\alpha=F_\alpha=\empty$, $\nu^\Tt_\alpha=\nu_M$ and $\nu^\Uu_\alpha=\psi(\nu_M)$. Let
\[ M_{\alpha+1}=M||\OR^M, \]
and
\[ N_{\alpha+1}=N||\psi(\OR^M)=N||(\psi(\nu_M)^+)^N. \]
Let $\pi_{\alpha+1}=\psi\rest\OR^M$. Now set $F_{\alpha+1}=\psi_\pi(E)$. Again the increasing length condition
follows from (\ref{ag:nu}) and we set $R_\alpha=N|\lh(F_\alpha)$.

\begin{case} $M$ is type 3, $\OR^{M^\sq} < \lh(E) < \OR^M$ and $\psi_\pi(\nu_M)>\nu_N$.\end{case}

Here it is not clear that any extender on $N$'s sequence corresponds
to $E$. A solution is to use an extra ultrapower in the upstairs tree.
Set $E_\alpha=\empty$ but $F_\alpha=F^N$. Let
$N_{\alpha+1}$ be the maximal degree ultrapower of the model as
chosen in a normal tree. Let $\eta_{\alpha+1}=\psi(\OR^M)$ and $\pi_{\alpha+1}=\psi\rest\OR^M$. (Note $N_{\alpha+1}$ agrees with $\Ult(N,F^N)$ past $\eta_{\alpha+1}$.) Note $\psi(\nu_M)$ is the largest cardinal of $N_{\alpha+1}|\eta_{\alpha+1}$. Now set $F_{\alpha+1}=\pi_{\alpha+1}(E)$. Increasing length holds as
\[ \pi_{\alpha+1}(\lh(E))>\pi_{\alpha+1}(\nu_M)\geq\lh(F^N) \]
by case hypothesis, and that $\pi_{\alpha+1}(\nu_M)$ is a cardinal of $\Ult(N,F^N)$. Here $R_{\alpha+1}=N_{\alpha+1}|\lh(F_{\alpha+1})$.

This defines $F_\alpha$ in all cases, and $F_{\alpha+1}$ where needed. We now notationally assume there was no padding used (so $F_{\alpha+1}$ is not yet defined), but otherwise the same discussion holds with $\alpha+1$ replacing $\alpha$.

Suppose player I chooses appropriate
$M^*=(M^*)_{\alpha+1}=M_{\beta}|\xi$ and $\deg(\alpha+1)=n$. Let $P\ins
M_\beta$ be the longest possible that $E_\alpha$ can apply to, so
$(M^*)_{\alpha+1}\ins P$. Let $\kappa=\crit(E_\alpha)<\nu^\Tt_\beta$. Our use of padding gives $\kappa\in\dom(\pi_\beta)$. Let
\[ \kappa'=\crit(F_\alpha)=\pi_\alpha(\kappa)=\pi_\beta(\kappa)<\nu^\Uu_\beta. \]
(The inequality follows the agreement hypotheses.)

If we're dealing with earliest model trees and $\gamma<\beta$, then $\nu^\Tt_\gamma\leq\kappa$. So preservation gives $\nu^\Uu_\gamma\leq\kappa'$, and $N_\beta$ is the correct model to return to in $\Uu$.

Let $N^*=(N^*)_{\alpha+1}=P'\ins N_\beta$ be largest measured by $F_\alpha$. If $M^*=P=M_\beta$ let $\eta^*_{\alpha+1}=\eta_\beta$; otherwise let $\eta^*_{\alpha+1}=\psi_\beta(\OR^{M^*})$. We need to check $\eta^*_{\alpha+1}\leq\OR^{P'}$, so that $M^*$ is embedded into a level $Q^*$ of $N^*$.

If $E_\beta=\empty$ then $P=M_\beta$ as $\nu^\Tt_\beta=\nu_{M_\beta}$ is a cardinal of $M_\beta$. Also $N_\beta|\eta_\beta\ins P'$ (even when $\psi_\beta(\nu^\Tt_\beta)<\nu_{N_\beta|\eta_\beta}$, $\psi_\beta(\nu^\Tt_\beta)$ is still a cardinal of $N_\beta|\eta_\beta$).

So assume $E_\beta\neq\empty$. Suppose $\lh(E_\beta)\in\dom(\pi_\beta)$. If $P\pins M_\beta$ then $P\in\dom(\pi_\beta)$, and $P$ is the least $P_1$ such that $M_\beta|\lh(E_\beta)\ins P_1\pins M_\beta$ and $P_1$ projects to $\kappa$. It follows that $\pi_\beta(P)=P'$. If instead $P=M_\beta$ then
$(\kappa^+)^{M_\beta}<\lh(E_\beta)$ and $\pi_\beta((\kappa^+)^{M_\beta})=(\kappa'^+)^{N_\beta|\eta_\beta}$, so $N_\beta|\eta_\beta\ins P'$.

Now suppose $E_\beta$ is the active extender of $M_\beta$. So
$P=M_\beta$. If $F_\beta$ isn't active on $N_\beta|\eta_\beta$, then case \ref{case:Fdef.tp3lownuimage} applies, and the cardinality of $\psi_\beta(\nu^\Tt_\beta)$ in $N_\beta|\eta_\beta$ gives $N_\beta|\eta_\beta\ins P'$.

This covers all cases. Let $Q^*=N^*|\eta^*_{\alpha+1}$. If $M^*$ is anomalous, we inductively have $Q^*\pins N_\beta$. Note $\kappa<\nu_{M^*}$, so $\kappa'<\nu_{Q^*}$. Since $\nu_{Q^*}$ is the image of a critical point leading to $N_\beta$, it's a cardinal of $N_\beta$. Therefore $Q^*\pins N_\beta=(N^*)_{\alpha+1}$. If $M^*$ isn't anomalous but $\alpha+1$ will be, then $M^*$ is below $(\kappa^+)^{E_\alpha}$, which gives $Q^*\pins (N^*)_{\alpha+1}$. This will give $\eta_{\alpha+1}<\OR^{N_{\alpha+1}}$ later.

Letting $\pi^*=\pi^*_{\alpha+1}=\psi_\beta\rest M^*$,
\[ \pi^*:M^*\to Q^* \]
is a (possibly anomalous) weak $n$-embedding ($n=\deg^\Tt(\alpha+1)$). By the agreement hypotheses, $\psi_\alpha$ agrees with $\pi^*$ on $\pow(\kappa)$. So the shift lemma goes through with $\pi^*$ and $\psi_\alpha: M_\alpha|\lh(E_\alpha)\to R_\alpha$. So there's a weak $n$-embedding
\[ \tau:\Ult_n(M^*,E)\to\Ult_n(Q^*,F) \]
such that $\tau\comp i^{M^*}_{E,n} =
i^{Q^*}_{F,n}\comp\pi^*$. (If $M^*$ is anomalous, both ultrapowers are to be formed at the unsquashed level. It doesn't really matter here whether $i^{Q^*}_{F,n}(\nu_{Q^*})$ is the sup of generators of $\Ult_0(Q^*,F)$, but it is, since $\nu_{Q^*}$ is regular in $Q^*$.)\\

\noindent\emph{Agreement.}

We have (\ref{ag:nu}), (\ref{ag:<nu}) and (\ref{ag:lh}) hold at $\alpha$ by definition of $F_\alpha$. $\psi_\alpha\rest\lh(E_\alpha)+1\sub\tau$ by definition of $\tau$. We will now define $\pi_{\alpha+1}=\sigma\com\tau$ where $\crit(\sigma)>\lh(F_\alpha)$, which will establish (\ref{ag:pis}) between $\pi_\alpha$ and $\pi_{\alpha+1}$.

We just have to set
\[ \sigma:\Ult_n(Q^*,F)\to i^{N^*}_{F,m}(Q^*); \]
\[ \sigma([a,f]_{F,n}^{Q^*}) = [a,f]_{F,m}^{N^*}, \]
where $m=\deg^\Uu(\alpha+1)$.

First note that given $[a,f]$ in the domain, we do have that it
represents an element of the larger ultrapower. (If $Q^*\pins N^*$ then $f\in N^*$. If $Q^*=N^*$ then $n\leq m$ since $\pi^*$ is a weak $n$-embedding and $\crit(E)<\rho_n^{M^*}$.) Also
it's clear that $[a,f]_{F,m}^{N^*}$ is an element of
$i^{N^*}_{F,m}(Q^*)$, and that the map is well defined. We get $\lh(F)<\crit(\sigma)$, and the following diagram commutes (with the canonical embeddings):

\[
\begin{picture}(100,100)(50,0)
\put(30,30){$M^*$}
\put(-5,75){$\Ult_n(M^*,E)$}
\put(115,30){$Q^*$}
\put(90,75){$\Ult_n(Q^*,F)$}
\put(185,75){$i^{N^*}_{F,m}(Q^*)$}
\put(37,40){\vector(0,1){30}} 
\put(50,35){\vector(1,0){65}} 
\put(122,40){\vector(0,1){30}} 
\put(60,80){\vector(1,0){28}} 
\put(132,38){\vector(3,2){49}} 
\put(152,80){\vector(1,0){30}} 
\put(72,40){$\scriptstyle \pi^*$}
\put(72,85){$\scriptstyle \tau$}
\put(165,85){$\scriptstyle \sigma$}
\end{picture}
\]

Now $\sigma\com\tau$ is a weak $m$-embedding. If $Q^*=N^*$ and $n=m$, then $\sigma=\id$. If $n=\om$ then all maps in the diagram are fully elementary. So suppose $n<\om$. $\sigma$ is $\Sigma_n$ elementary
by Los' theorem. The preservation
properties required of a weak $n$-embedding hold because we know they hold of all embeddings other than $\sigma$ in the commuting diagram, and that $i^{M^*}_{E,n}$ is cofinal in the $\rho_n$
of $\Ult_n(M^*,E)$. Given a cofinal $X_1\sub\rho^{M^*}_n$ on which $\pi^*$ is $\Sigma_{n+1}$ elementary, $X=i^{M^*}_{E,m}``X_1$ works for
$\sigma\com\tau$. For $\tau``X = i^{Q^*}_{F,n}``(\pi_\beta``X_1)$, and both $i^{Q^*}_{F,n}$ and $i^{N^*}_{F,m}\rest Q^*$ are $\Sigma_{n+1}$ elementary, so by commutativity, $\sigma$ is $\Sigma_{n+1}$ elementary on $\tau``X$.

This finishes the successor stage of construction. Limit stages are as usual.
\renewcommand{\qedsymbol}{$\Box$(Theorem \ref{thm:freeiter})}\end{proof}

\printindex
\end{document}